\documentclass[a4wide,11pt,reqno]{article}
\usepackage{bm}
\usepackage{mathrsfs}
\usepackage[margin=1in]{geometry}
\usepackage{fullpage}
\usepackage{enumerate}
\usepackage[utf8]{inputenc}
\usepackage{amssymb, amsmath, amsfonts}
\usepackage{epsfig,amsbsy,amsthm} 
\usepackage{graphics}
\usepackage{caption}
\usepackage{subcaption}
\usepackage{float,epsfig}
\usepackage{fixmath}
\usepackage{graphics}
\usepackage{pgfpages}
\usepackage{appendix}
\numberwithin{equation}{section}  
\usepackage{hyperref} 
\usepackage{graphicx}
\usepackage{pst-all}
\usepackage{calligra}
\usepackage{tikz}
\usepackage{calligra}
\usepackage{color}
\usepackage{tikz}
\usepackage{cleveref}
\usepackage[utf8]{inputenc}
\usepackage[autostyle]{csquotes}
\usepackage[section]{algorithm}
\usepackage{multirow}
\usepackage{array}
\newcolumntype{H}{>{\setbox0=\hbox\bgroup}c<{\egroup}@{}}
\DeclareMathAlphabet{\mathpzc}{OT1}{pzc}{m}{it}
\DeclareMathAlphabet{\mathcalligra}{T1}{calligra}{m}{n}
\DeclareMathAlphabet{\mathpzc}{OT1}{pzc}{m}{it}
\DeclareMathAlphabet{\mathcalligra}{T1}{calligra}{m}{n}
\DeclareMathAlphabet{\mathpzc}{OT1}{pzc}{m}{it}
\DeclareMathAlphabet{\mathcalligra}{T1}{calligra}{m}{n}
\begin{document}
	\newtheorem{theorem}{\bf Theorem}[section]
	\newtheorem{proposition}[theorem]{\bf Proposition}
	\newtheorem{definition}{\bf Definition}[section]
	\newtheorem{corollary}[theorem]{\bf Corollary}
	\newtheorem{exam}[theorem]{\bf Example}
	\newtheorem{remark}[theorem]{\bf Remark}
	\newtheorem{lemma}[theorem]{\bf Lemma}
	\newtheorem{assum}[theorem]{\bf Assumption}
	\newcommand{\von}{\vskip 1ex}
	\newcommand{\vone}{\vskip 2ex}
	\newcommand{\vtwo}{\vskip 4ex}
	\newcommand{\ds}{\displaystyle}
	\def \noin{\noindent}
	\newcommand{\be}{\begin{equation}}
		\newcommand{\ee}{\end{equation}}
	\newcommand{\beno}{\begin{equation*}}
		\newcommand{\eeno}{\end{equation*}}
	\newcommand{\ba}{\begin{align}}
		\newcommand{\ea}{\end{align}}
	\newcommand{\bano}{\begin{align*}}
		\newcommand{\eano}{\end{align*}}
	\newcommand{\bea}{\begin{eqnarray}}
		\newcommand{\eea}{\end{eqnarray}}
	\newcommand{\beano}{\begin{eqnarray*}}
		\newcommand{\eeano}{\end{eqnarray*}}
	\def \noin{\noindent}
	\def\arraystretch{1.3}
	\def \tcK{{\tilde {\mathcal K}}}    
	\def \O{{\Omega}}
	\def \cT{{\mathcal T}}
	\def \cV{{\mathcal V}}
	\def \cE{{\mathcal E}}
	\def \R{{\mathbb R}}
	\def \V{{\mathbb V}}
	\def \S{{\mathbb S}}
	\def \N{{\mathbb N}}
	\def \Z{{\mathbb Z}}
	\def \Mc{{\mathcal M}}
	\def \Cc{{\mathcal C}}
	\def \Rc{{\mathcal R}}
	\def \Ec{{\mathcal E}}
	\def \Gc{{\mathcal G}}
	\def \Tc{{\mathcal T}}
	\def \Qc{{\mathcal Q}}
	\def \Ic{{\mathcal I}}
	\def \Pc{{\mathcal P}}
	\def \Oc{{\mathcal O}}
	\def \Uc{{\mathcal U}}
	\def \Yc{{\mathcal Y}}
	\def \Ac{{\mathcal A}}
	\def \Bc{{\mathcal B}}
	\def \k{\mathpzc{k}}
	\def \Rp{\mathpzc{R}}
	\def \Os{\mathscr{O}}
	\def \Js{\mathscr{J}}
	\def \Es{\mathscr{E}}
	\def \Qs{\mathscr{Q}}
	\def \Ss{\mathscr{S}}
	\def \Cs{\mathscr{C}}
	\def \Ds{\mathscr{D}}
	\def \Ms{\mathscr{M}}
	\def \Ts{\mathscr{T}}
	\def \LL{L^{\infty}(L^{2}(\Omega))}
	\def \LH{L^{2}(0,T;H^{1}(\Omega))}
	\def \B {\matheorem{BDF}}
	\def \el {\matheorem{el}}
	\def \re {\matheorem{re}}
	\def \e {\matheorem{e}}
	\def \div {\matheorem{div}}
	\def \CN {\matheorem{CN}}
	\def \Rs   {\mathbf{R}_{{\matheorem es}}}
	\def \Rb {\mathbf{R}}
	\def \Jb {\mathbf{J}}
	\def  \apos {\emph{a posteriori~}}
	
	\def\mean#1{\left\{\hskip -5pt\left\{#1\right\}\hskip -5pt\right\}}
	\def\jump#1{\left[\hskip -3.5pt\left[#1\right]\hskip -3.5pt\right]}
	\def\smean#1{\{\hskip -3pt\{#1\}\hskip -3pt\}}
	\def\sjump#1{[\hskip -1.5pt[#1]\hskip -1.5pt]}
	\def\jumptwo{\jump{\frac{\p^2 u_h}{\p n^2}}}
\title{Adaptive SIPG method for approximations of parabolic boundary control problems with bilateral box constraints on Neumann boundary}
\author{Ram Manohar\thanks{Department of Mathematics \& Statistics, Indian Institute of Technology Kanpur, Kanpur - 208016, India \tt{(rmanohar267@gmail.com)}.}, ~~B. V. Rathish Kumar\thanks{Department of Mathematics \& Statistics, Indian Institute of Technology Kanpur, Kanpur - 208016, India   \tt{(drbvrk11@gmail.com)}.}, ~~Kedarnath Buda\thanks{Department of Mathematics \& Statistics, Indian Institute of Technology Kanpur, Kanpur - 208016, India   \tt{Kedarnath.buda@gmail.com}.},
	~~ and ~~Rajen Kumar Sinha\thanks{Department of Mathematics, Indian Institute of Technology Guwahati, Guwahati - 781039, India \tt{(rajen@iitg.ac.in)}.}}

\date{}

\maketitle
\textbf{ Abstract.}{\small{ This study presents an a posteriori error analysis of adaptive finite element approximations of parabolic boundary control problems with bilateral box constraints that act on a Neumann boundary.  The control problem is discretized using the symmetric interior penalty Galerkin (SIPG) technique. We derive both reliable and efficient type residual-based error estimators coupling with the data oscillations.  The implementation of these error estimators  serves as a guide for the adaptive mesh refinement process, indicating whether or not more refinement is required. Although the control error estimator accurately captured control approximation errors, it had limitations in terms of guiding refinement localization in critical circumstances. To overcome this, an alternative control indicator was used in numerical tests. The results demonstrated the clear superiority of adaptive refinements over uniform refinements, confirming the proposed approach's effectiveness in achieving accurate solutions while optimizing computational efficiency.  numerical experiment showcases the effectiveness of the derived error estimators.}}  \\ 
		
		
\vspace{.1in}
		
\maketitle 

\section{Introduction}
Let $\Gamma:=\partial \Omega$ be the boundary  of an open bounded polygonal domain $\Omega \subset \mathbb{R}^{2}$ such that $\Gamma=\overline{\Gamma_{D}} \cup \overline{\Gamma_{N}}$ with $\Gamma_{D} \cap \Gamma_{N}=\emptyset$. For adaptive finite element approximation,  we take into the account of following control-constrained problem:
\begin{equation}
\underset{q \in \mathpzc{Q}_{Nad}}{\min} \int_0^T\Big\{\frac{1}{2}\int_{\O}(y-y_{d})^2\, dx+\frac{\alpha}{2} \int_{\Gamma_{N}}(q-q_{d})^2\,ds+\int_{\Gamma_{N}} r_{_N}\,y\, ds\Big\}dt \label{contfunc}
\end{equation}
subject to the conditions
\begin{subequations}
	\begin{eqnarray}
		\frac{\partial y}{\partial t}- \Delta y+a_0 y &=&f \quad \;\;\;\;\;\;\;\;\;\text {in } \Omega\times (0,T], \label{contstate}\\
		y &=&g_{_D} \quad \;\;\;\;\;\; \text { on } \Gamma_{_D}\times (0,T], \label{diricon}\\
		\frac{\partial y}{\partial n} &=&q+g_{_N} \quad  \text { on } \Gamma_{_N}\times (0,T], \label{neumanncon}\\
		y(x,0) &=& y_0(x) \quad \;\;\; \text{ in } \Omega, \label{incon}
	\end{eqnarray}
\end{subequations}
where $\mathpzc{Q}_{Nad}$ is a closed convex set with control constraints defined by
\begin{equation}
\mathpzc{Q}_{Nad}:=\big\{q \in L^2(I;L^{2}(\Gamma_{_N})):~~ q_{a} \leq q(x,t) \leq q_{b} \;\; \text{ a.e. }\; (x,t) \in \Gamma_{_N} \times I \big\}. \label{contspace}
\end{equation}
Here $I=[0,T]$, $q_{a},\, q_{b} \in L^{\infty}(I;L^{\infty}\left(\Gamma_{_N}\right))$ with $q_{a} \leq q_{b}$ for almost all $(x,t) \in \Gamma_{_N} \times I$. We consider a regularization parameter ($\alpha$) as a positive constant. The desired control function $q_d$ serves as a reference for the control (see, \cite{hoppeiliash2008}). Keep in mind that the aforementioned method also works for the most prevalent and specific situation, $q_{d}=0$, this means that there is no prior knowledge about the ideal control. Furthermore, the coefficient $r_{_N}$ in \eqref{contfunc} is added explicitly to the cost functional which ensures that the boundary aspect of any given adjoint-state function are achieved (see, Eq. \eqref{2.17nuemanncod-adj}). With similar cost functionals, we are delighted to recommend \cite{casasdhamo2012,hinzematthes2009,kohlssiebert2014}, and indeed the citations thereof. We choose the given functions in the following spaces to illustrate the well-posedness  (cf., \cite{lions1971}, \cite{LM72}) of the minimization problem\eqref{contfunc}--\eqref{contspace}:
\begin{subequations}\label{assump1}
\begin{align}
f,\,& y_{d} \in L^2(I;L^{2}(\Omega)),\, q_{d},\, g_{_N},\, r_{_N} \in L^2(I;L^{2}\left(\Gamma_{_N}\right)),\\ 
\text{and} &~ y_0\in H^2(\Omega),\, g_{_D} \in L^2(I;H^{1/2}\left(\Gamma_{_D}\right)), \;a_0 \in L^{\infty}(\Omega), \end{align}
\end{subequations}
then there exists a positive constant $c$ such that 
\begin{equation}
a_0 \geq c > 0 \quad \text{a.e. in } \Omega. \label{assum2}
\end{equation}
Optimization problems governed by partial differential equations (PDEs) arise in various real-life scenarios, including the shape optimization of technological devices \cite{pironneau2001}, parameter identification in environmental processes, and addressing control flow problems \cite{fernando2001, skiba2006}. These problems are inherently complex, demanding meticulous attention to secure efficient numerical approximations. Among all the methods, the adaptive finite element method (AFEM) stands out as particularly effective. The AFEM operates through the successive iterations of a sequence comprising the following steps: 
$$\texttt{ Solve} \rightarrow \texttt{ Estimate } \rightarrow \texttt{ Mark } \rightarrow \texttt{ Refine}.$$
In the \texttt {Solve} stage, the optimization problem is numerically tackled within a finite-dimensional space defined by a given mesh. But the crucial component for the procedure can be determined in the \texttt{Estimate} stage. Here, local error indicators are computed based on discrete solutions without knowing for the exact solution. These indications are essential in designing mesh adaption algorithms that optimize computations and convey computational cost proportionately. Subsequently, the \texttt{Marking} step utilizes the information gleaned from these indicators to select a subset of elements for refinement. Afterwards, the \texttt{Refine} step executes the refinement process within the adaptive loop.

Although the AFEM, which made significant contributions to the pioneering work of Babuška and Rheinboldt \cite{rheinboldt1978}, has long been recognized as a popular approach for efficiently solving initial and boundary value problems governed by PDEs. It's application to constrained optimal control problems (OCPs) has only recently gained popularity due to the contribution of Liu and Yan \cite{liuyan2001} and Becker, Kapp, and Rannacher \cite{becker2000}. In their seminal work \cite{liuyan2001}, Liu and Yan proposed a residual-type a posteriori error estimator tailored for OCPs. Concurrently, Becker, Kapp, and Rannacher \cite{becker2000} introduced a dual-weighted goal-oriented adaptivity strategy. For a comprehensive understanding of these methodologies, we refer to \cite{hoppeiliash2008, zhou2009, hoppeiliash06,  manohar2024error, rmrk2021, manohar2022local,   yanzhou2009,  arasozen2015} 
for residual-type estimators, and to \cite{benedix2009,  mhoppe2010} for insights into the dual-weighted goal-oriented approach, along with references therein for further exploration of recent advancements. Moreover, ensuring the theoretical success of a posteriori error estimators has prompted endeavors such as those undertaken in \cite{gaevskaya2007, siebertrosch2014} to establish the convergence of the AFEM for OCPs. These theoretical analyses serve to underpin the practical applicability and reliability of the AFEM in this domain.

Adaptive mesh refinement presents an attractive avenue for tackling OCPs, especially those characterized by layers or singularities within certain regions of the mesh. This adaptivity enables local refinement around these layers as necessary, effectively achieving the desired residual bound with minimal degrees of freedom. The a posteriori error analysis of OCPs is covered in a large amount of literature, although distributed OCPs are the main emphasis. It is a domain that is well explored in works like \cite{becker2000,mhoppe2010,hoppeiliash2008, zhou2009,  Pratibha2024,shakya17, shakya19, yanzhou2009,arasozen2015}. 
However, numerical solutions for boundary OCPs have not received adequate consideration. Existing studies primarily investigate residual type error estimators \cite{hoppeiliash06,kohlssiebert2014,liuyan09,liuyan2001}, with some delving into hierarchical type estimators \cite{kohlssiebert2014}, all primarily utilizing continuous finite element discretizations, except for \cite{kohlssiebert2014}.
In the latter, discontinuous finite elements for control discretization are used for the first time by Kohls, Rösch, and Siebert. In particular, research by Leykekhman \cite{leykekhman2012} suggests that discontinuous Galerkin methods offer superior convergence behavior for OCPs featuring boundary layers. Optimal convergence orders are achievable when errors are computed away from the boundary or interior layers. Moreover, employing discontinuous finite elements for control discretization on the Neumann boundary facilitates more efficient projection operator computation as elucidated in \cite[Sec. 4.1]{kohlssiebert2014}.

Discontinuous Galerkin methods possess several inherent advantages over alternative finite element methods. 
They have the ability to handle curved boundary and inhomogeneous boundary conditions with convenience, create state and test spaces with simplicity, manage nonmatching grids with simplicity, and develop $hp$-adaptive grid improvements. Although these methods have long been used, interest in them has increased recently due to the availability of low-cost processors. For further insights into discontinuous Galerkin methods, the reader may refer to seminal works such as \cite{arnoldbrezzi2002,warburton2008,kanschat2009,kanschat2007,pascal2007,bevere2014}. Despite their applicability, discontinuous Galerkin methods have mostly been explored in the context of distributed OCPs \cite{leykekhman2012,arasozen2015}. 
To the best of our knowledge, their application to boundary OCPs has not yet been studied in any existing study, indicating a promising topic for future research and development.

The SIPG approach has various potential for optimal control problems. It handles complicated geometries and unstructured or locally refined meshes with increased flexibility, as well as allowing for local mesh adaptivity, local mass conservation, and robustness in handling rough coefficients.  Its locally conservative characteristics are critical in situations regulated by conservation laws. 
The SIPG method also makes it easier to generate high-order accurate schemes, which can increase the accuracy of state and adjoint variables in optimal control formulations. This approach is  suitable for problems involving discontinuities, state or control problems, and also interface conditions.  Furthermore, the symmetric interior penalty formulation provides strong stability, especially in diffusion-dominated problems. Nevertheless, the element-wise nature of DG techniques allows them to be highly parallelizable, which is useful for tackling large-scale control problems effectively.
These advantages make SIPG method a popular choice for solving optimal control problems, particularly those involving fluid dynamics, diffusion processes, or non-smooth data. By performing element-wise computations, the SIPG method enables efficient implementation, making it suitable for various type of problems. 
By taking these factors into account, we study  the error analysis for the parabolic boundary optimal control problems.

In this study, we focus on obtaining reliable and efficient a posteriori error estimators for mixed boundary control problems (BCPs) governed by a parabolic PDE with bilateral constraints on the control variable. We followed the first-discretize-then-optimize strategy. The discretization of the problem is followed by the use of the SIPG method. We provide the optimality conditions, articulating them in terms of the state, adjoint-state, control, and cocontrol variables, with reference to the Lagrangian multiplier associated with the controls. Further, we delineates the SIPG discretization methodology applied to the BCPs. We expound upon the fully-discretization process, elucidating the numerical techniques employed to address the problem effectively. 
Specifically, we used a residual-type error estimator to evaluate global discretization errors across all variables, encompassing edge and element residuals. In addition, we incorporate considerations of data oscillations into the error analysis. Moreover, we derive local upper and lower a posteriori error estimates for the BCPs.
Our a posteriori error analysis of the boundary control problems encompasses a comprehensive assessment of errors in the state, adjoint state, control, and cocontrol variables. Further, we demonstrate the accuracy and efficiency of our method considering two examples. First example shows that how well our derived control estimators capture the active-inactive sets and the mesh reflection near the prescribed controls on the boundary. 
However, the second example addresses the time-dependent singularity behavior, which emphasizes how much effectivily error estimators capture the trajectory of the singularity at the different time steps.
It is worth noting that while considerations of data oscillations have been addressed in previous works such as \cite{aniswarth2007,mnsiebert} for single-state equations and in \cite{hoppeiliash06} for OCPs, our paper extends this analysis to the context of boundary OCPs governed by parabolic PDEs.

The subsequent sections of this paper are arranged as follows. Section \ref{section222} presents some basic notation and weak formulation of the control problems \eqref{contfunc}-\eqref{incon}. Section \ref{section333} is dedicated to the space-time discretization of the control problems. Section \ref{sec4errest} provides a comprehensive discussion on the a posteriori error estimates derived in our analysis. Section \ref{section5555} presents numerical findings that illustrate the performance and efficacy of our proposed adaptive mesh refinement strategy. A concluding remark is presented at the last.
\section{Weak representation of the model problem} \label{section222}
We adopt the standard notation from Lebesgue and Sobolev space theory, as presented in \cite{adams1975}. For $ m \in \mathbb{N}$, we denote the inner product, seminorm and norm on $H^{m}(\Omega)$ by $(\cdot, \cdot)_{m, \Omega}$, $|\cdot|_{m, \Omega}\; \text{and}\; \|\cdot\|_{m, \Omega}$, respectively. Further, let
\begin{align}
(u,v)_{L^2(\Omega)}:=\int_{\Omega}u\,v\,dx,\quad \forall u,\,v \in L^2(\Omega),\nonumber\\
(\phi,\psi)_{L^2(\Gamma)}:=\int_{\Omega}\phi\,\psi\,dx,\quad \forall \phi\,\psi \in L^2(\Gamma).\nonumber
\end{align}
For simplicity, we use $(\cdot,\cdot)_{L^2(\Omega)}=(\cdot,\cdot)$.
At this point, we introduce the state, adjoint, and test function spaces as $\mathpzc{Y}:= L^2(I; \mathpzc{W})\cap H^1(I;\mathpzc{V}^*)$, $\mathscr{Z}:= L^2(I; \mathpzc{V})\cap H^1(I;\mathpzc{V}^*)$ and $ \mathpzc{V}$, respectively,
where 
$$
\mathpzc{W}:=\big\{ w \in  H^{1}(\Omega):\;\; w|_{\Gamma_{_D}}=g_{_D}\big\}, \quad 
\text{and}  \quad \mathpzc{V}:=\big\{ v \in H^{1}(\Omega):\;\;v|_{\Gamma_{_D}}=0\big\}.
$$
Note that the nonhomogeneous Dirichlet boundary state solution space $\mathpzc{Y}$ can be modeled over the Hilbert space $\mathscr{Z}$ such that  $\mathpzc{Y}:=\big\{y \in L^2(I; H^1(\Omega))|~~ y=w+\tilde{u},\, \tilde{u}\in \mathscr{Z}\big\}$. This can be rewritten as $\mathpzc{Y}:=w+\mathscr{Z}$, which is a affine space of a Hilbert space. Further, $\mathpzc{Y}$ inherits a natural metric and topological structure from $\mathscr{Z}$ via translation $w$ (known as a lifting function, i.e., $\gamma_0(w) = g_D \;\text{on } \Gamma_D \times (0,T)$).

Let us define the bilinear form  $a(\cdot,\cdot):~\mathpzc{W}\times \mathpzc{V} \mapsto \mathbb{R}$ given by, for $w\in \mathpzc{W}$,
$$
a(w, v)=\int_{\Omega}(\nabla w \cdot \nabla v+a_0 w v)\, dx, \;\;\forall \; v \in \mathpzc{V}.
$$
This leads to the following weak formulation of the problem \eqref{contfunc}--\eqref{incon}, reads as: To find $(y, q) \in \mathpzc{Y}\times \mathpzc{Q}_{Nad}$ such that
\begin{subequations}
\begin{align}
\underset{q \in \mathpzc{Q}_{Nad}}{\min} \int_0^T\Big\{\frac{1}{2}\int_{\O}&(y-y_{d})^2\, dx+\frac{\alpha}{2} \int_{\Gamma_{N}}(q-q_{d})^2\,ds+\int_{\Gamma_{N}} r_{_N}\,y\, ds\Big\}dt, \label{weakcontfunc} \\
\text{subject  to} \hspace{2cm} & \nonumber\\
\big(\frac{\partial y}{\partial t}, \phi \big)+a(y, \phi)&=(f, \phi)+\left(q+g_{N}, v\right)_{L^2(\Gamma_{_N})}, \quad t\in (0,T],~~ \forall \phi \in\mathpzc{V},\\
y(x,0)&=y_0(x). \label{initialcnd}
\end{align}
\end{subequations}
 It is well known \cite{lions1971, LM72} that the BCP \eqref{weakcontfunc}--\eqref{initialcnd} with the assumption \eqref{assump1}--\eqref{assum2} has a unique solution $(y, q) \in \mathpzc{Y} \times \mathpzc{Q}_{Nad}$ iff there is an adjoint-state variable $z \in \mathscr{Z}$, for $t\in [0,T]$, satisfying
 \begin{subequations}
\begin{eqnarray}
\big(\frac{\partial y}{\partial t}, \phi \big)+a(y, \phi)&=&(f, \phi)+\left(q+g_{_N}, \phi \right)_{L^2(\Gamma_{_N})}  \quad \forall \phi \in \mathpzc{V},\label{weakformstate} \\
y(x,0)&=&y_0(x),\\
-\big(\frac{\partial z}{\partial t}, \psi \big)+a(\psi, z)&=&\left(y-y_{d}, \psi\right)+\left(r_{_N}, \psi \right)_{L^2(\Gamma_{_N})} \quad \forall \psi \in \mathpzc{V}, \label{weakformadjoint-state}\\
z(x,T)&=&0,\\
\left(\alpha \left(q-q_{d}\right)+z, \varphi-q\right)_{L^2(\Gamma_{_N})} &\geq& 0 \quad \forall \varphi \in \mathpzc{Q}_{Nad}, \label{weakformcontrol}
\end{eqnarray}
\end{subequations}
where the adjoint-state variable $z$ is determined by the following system
\begin{subequations}
\begin{eqnarray}
-\frac{\partial z}{\partial t}-\Delta z+a_0 z &=&y-y_{d} \quad \text { in } \Omega \times [0,T), \label{2.14adjoint-state}\\
z(\cdot, T) &=&0 \quad \text { in } \Omega,\\
z &=&0 \quad \text { on } \Gamma_{_D}\times [0,T), \\
\frac{\partial z}{\partial n} &=&r_{_N} \quad \text { on } \Gamma_{_N}\times [0,T).\label{2.17nuemanncod-adj}
\end{eqnarray}
\end{subequations}
Introducing a co-control $\mu \in L^2(I; L^{2}\left(\Gamma_{_N}\right))$ associated with control constraints, for $\gamma>0$, the optimality system \eqref{weakformstate}--\eqref{weakformcontrol} is given by
 \begin{subequations}
\begin{align}
\big(\frac{\partial y}{\partial t}, \phi \big)+a(y, \phi)&=~(f, \phi)+\left(q+g_{_N}, \phi \right)_{L^2(\Gamma_{_N})}  \quad \forall \phi \in \mathpzc{V},\label{weakstate1} \\
y(x,0)&=~y_0(x),\\
-\big(\frac{\partial z}{\partial t}, \psi \big)+a(\psi, z)&=~\left(y-y_{d}, \psi\right)+\left(r_{_N}, \psi \right)_{L^2(\Gamma_{_N})} \quad \forall \psi \in \mathpzc{V}, \label{weakadjoint-state1}\\
z(x,T)&=~0,\\
 \mu+\alpha\left(q-q_{d}\right)+z&=~0 \quad \text { a.e. in }\; (x,t)\in \Gamma_{_N}\times I,\label{optmallitycon1} \\
\mu-\max \left\{0, \mu+\gamma\left(q-q_{b}\right)\right\}&+\min \left\{0, \mu+\gamma\left(q-q_{a}\right)\right\}~=~0 \quad \text { a.e. in }\; (x,t)\in \Gamma_{_N}\times I.\label{projminmax1}
\end{align}
\end{subequations}
It is  found  that \eqref{weakstate1}--\eqref{projminmax1} possesses Newton differentiability  at least for the $\gamma=\alpha$ (cf., \cite{itokunisch2002}). 
An equivalent representation of \eqref{projminmax1}--\eqref{projminmax1} analogous to the pointwise complementarity system with $\mu(x,t)=\mu_{b}(x,t)-\mu_{a}(x,t)$, for $(x,t)\in \Gamma_N\times I$, can be written as:
\begin{subequations}
\begin{align}
 & \mu_{a} \geq 0, \quad q_{a}-q \geq 0, \quad \mu_{a}\left(q_a-q\right)=0, \; (x,t)\in \Gamma_N\times I, \label{compcond2}\\
 & \mu_{b}\geq 0, ~ \quad q-q_{b} \geq 0, \quad \mu_{b}\left(q-q_b\right)=0, \; (x,t)\in \Gamma_N\times I.\label{compcond1} 
\end{align}
\end{subequations}
To tackle the nonsmoothness in \eqref{projminmax1}, we employ a semismooth Newton iteration, which can be effectively combined with an active set strategy. This approach enables us to identify the active sets at each Newton iteration step by
\begin{subequations}
\begin{align}
&\mathpzc{A}_{a}=\left\{(x,t) \in \Gamma_{_N}\times I:~ \mu +\gamma\left(q-q_{a} \right)<0\right\}, \label{activeseta}\\
&\mathpzc{A}_{b}=\left\{(x,t) \in \Gamma_{_N}\times I:~ \mu+\gamma\left(q-q_{b} \right)>0\right\}, \label{activesetb}
\end{align}
\end{subequations}
and inactive set is given by $\mathcal{I}=\Gamma_{_N}\times I \backslash\left\{\mathpzc{A}_{a} \cup \mathpzc{A}_{b}\right\}$. The complementarity conditions in \eqref{compcond1}-- \eqref{compcond2} can be equivalently expressed as
\begin{subequations}\label{2.28contcompcond}
\begin{align}
&q=q_{a}, \quad \mu_{b}=0, \quad \mu \leq 0 \quad \text { a.e. on }\; \mathpzc{A}_{a}, \label{compcon11}\\
&q=q_{b}, \quad \mu_{a}=0, \quad \mu \geq 0 \quad \text { a.e. on }\; \mathpzc{A}_{b}, \label{compcon12}\\
q_{a}< & q <q_{b}, \quad \mu_{a}=\mu_{b}=0, \quad \mu=0 \quad \text { a.e. on }\; \mathcal{I}.\label{compcon13}
\end{align}
\end{subequations}
\section{Finite dimensional setting of the PBCP} \label{section333}
We employ the SIPG method to discretize the minimization problem \eqref{contfunc}-\eqref{contspace}. \smallskip

\noindent
\textbf{\it  Computational-domain discretization.} Let $\mathscr{T}_{h}$ be a shape-regular simplicial triangulation of $\bar{\Omega}$, where the triangle boundaries align with the boundary
$\Gamma$. The triangulation satisfies, for the triangles $K_i,\, K_j \in \mathscr{T}_{h},\, i \neq j$, then $K_i \cap K_j$ is either empty or a vertex or an edge. Let us introduce the following notions which is crucial for the subsequent analysis.
\begin{enumerate}
\item We denote  the set of all interior  and the boundary edges by  $\mathcal{E}_{0,h}$ and 
 $\mathcal{E}_{B,h}$, respectively, while the boundary edges  decomposed  into
 the set of all  Dirichlet and  Neumann boundary edges and denoted by  $\mathcal{E}_{D,h}$ and $\mathcal{E}_{N,h}$, respectively.
We set  $\mathcal{E}_{h}=\mathcal{E}_{0, h} \cup \mathcal{E}_{B,h}.$
\item For each element $K\in \mathscr{T}_h$ and each edge $E\in \mathcal{E}_{h}$, we introduce  the diameter of element $K$ and the length of edge $E$ by 
 $h_K=diam(K)$ and 
 $h_E=length(E)$, respectively. Further,  the maximum diameter of element is denoted by
  $h=max\{h_K: K\in \mathscr{T}_h\}$.
\end{enumerate}
Consider an interior edge $E\in \mathcal{E}_{0, h}$ shared by two elements $K$ and $K^e$ in $\mathscr{T}_{h}$ such that $E=\partial K\cap \partial K^e$. Let $\mathbf{n}_{K}$ and $\mathbf{n}_{K^{e}}$ be the unit outward normals to  $\partial K$ and $\partial K^{e}$, respectively.  We  now define the traces for a piecewise continuous scalar function of $\varphi$ along $E$ denoted by $\varphi|_{E}$ from inside $K$ and $\varphi^{e}|_{E}$ from inside $K^{e}$, respectively. We define the mean and jump of $\varphi$ across the edge $E$, as follows
\begin{align}\label{jumpaveg}
\smean{\varphi} :=\frac{1}{2}(\varphi|_{E}+\varphi^{e}|_{E})  \quad \text{and} \quad 
\sjump{\varphi} :=\varphi|_{E} \mathbf{n}_{K}+\varphi^{e}|_{E} \mathbf{n}_{K^{e}},
\end{align}
respectively.  Analogously, let $\nabla \phi$ be a piecewise continuous vector field, we introduce the mean and jump for $\nabla \varphi$ across an edge $E$ by
\begin{align}\label{gradjumpaveg}
\smean{\nabla \varphi }:=\frac{1}{2}(\nabla \varphi|_{E}+\nabla \varphi^{e}|_{E}) \quad   \text{and} \quad  \sjump{\nabla\varphi}:=\nabla \varphi|_{E} \cdot \mathbf{n}_{K}+\nabla \varphi^{e}|_{E} \cdot \mathbf{n}_{K^{e}},
\end{align}
respectively. For a boundary edge 
$E \in K \cap \Gamma$, we define the mean and jump as $\smean{\nabla \varphi }=\nabla \varphi$ and 
$\sjump{\varphi}=\varphi \mathbf{n}$, respectively, where $\mathbf{n}$ is the  unit outward normal vector on $\Gamma$.

Next, we introduce the discrete spaces for the state, adjoint-state, and control variables, along with the corresponding test space, by
\begin{subequations}
\begin{align}
&\mathpzc{V}_{h}~=~\mathpzc{W}_{h}=\left\{v \in L^{2}(\Omega):\quad v|_{K} \in \mathbb{P}_{1}(K) \quad \forall K \in \mathscr{T}_{h} \right\}, \label{femspacestate}\\
&\mathpzc{Q}_{N,h}~=~\{q \in L^{2}\left(\Gamma_{_N}\right):\quad q|_{E} \in \mathbb{P}_{1}(E) \quad \forall E \in \mathcal{E}_{N,h},\},\label{femspacecontrol} 
\end{align}
\end{subequations}
respectively, where $\mathbb{P}_{1}(K)$ (resp., $\mathbb{P}_{1}(E)$) is the set of linear polynomials in $K$ (resp., on $E$).
One may observed that the  discrete space $\mathpzc{W}_{h}$ and the space of test functions $\mathpzc{V}_{h}$ are identical due to the weak treatment of boundary conditions in DGFEM. 
Further, we set $$
\mathpzc{Y}_h:=H^1(I; \mathpzc{W}_{h}),~~ \mathscr{Z}_h:=H^1(I; \mathpzc{V}_{h}), ~~ \text{and}~~\mathpzc{Q}_{Nad,h}= L^2(I;\mathpzc{Q}_{N,h})\cap \mathpzc{Q}_{Nad}.$$ 
For all $(w,q,v) \in \mathpzc{W}_{h} \times \mathpzc{Q}_{N,h}\times \mathpzc{V}_{h},$ we define the bilinear and linear forms as   
\begin{subequations}\label{3.4libili}
\begin{align}
&a_{h}(w, v)=\sum_{K \in \mathscr{T}_{h}} \int_{K}(\nabla w \cdot \nabla v+a_0 w v)\,dx -\sum_{E \in \mathcal{E}_{0,h} \cup \mathcal{E}_{D,h}} \int_{E}( \smean{\nabla w} \cdot \sjump{v} +\smean{\nabla v} \cdot \sjump{w})\,ds\nonumber \\
&\hspace{2.0cm}+\sum_{E \in \mathcal{E}_{0,h} \cup \mathcal{E}_{D,h}} \frac{\sigma_0}{h_{E}} \int_{E} \sjump{w} \cdot \sjump{v}\,ds, \label{weakform}\\
&b_{h}(q, v)=\sum_{E \in \mathcal{E}_{N,h}} \int_{E} q v\,ds, \label{weakformcon}\\
&l_{h}(v)=\sum_{K \in \mathscr{T}_{h}} \int_{K} f\, v\, dx+\sum_{E \in \mathcal{E}_{D,h}} \int_{E} g_{_D}\left(\frac{\sigma_0}{h_{E}} \mathbf{n}_{E} \cdot \sjump{v}-\smean{\nabla v}\right)\,ds  +\sum_{E \in \mathcal{E}_{N,h}} \int_{E} g_{_N} v\, ds,\label{linearfuc} 
\end{align}
\end{subequations}
where the penalty parameter $\sigma_0 \in \mathbb{R}_{0}^{+}$ which is independent of mesh parameter $h$. The penalty parameter $\sigma_0$ should be sufficiently large to guarantee the stability of the discontinuous Galerkin method. The discontinuous Galerkin approximation solution converges to the continuous Galerkin solution as the penalty parameter goes to infinity, for details (see,  \cite{chapman2014}).

Since, the bilinear form $a_{h}(\cdot, \cdot)$ is consistent with the state equation \eqref{contstate}-\eqref{incon} for a fixed given control $q$ in the weak sense. 
This leads to the following orthogonality relation
\begin{equation}
a_{h}\left(y-y_{h}, \phi\right)=0 \quad \forall \phi \in \mathpzc{V}_{h}.\label{orthogonality}
\end{equation}
The following lemma demonstrates the continuity and coercivity of the bilinear form (see, \cite[Lemma 3.1]{pascal2007}).  The proof can be easily followed by using the Cauchy-Schwarz inequality, trace inequality\eqref{tracegeninq} and the inverse inequality\eqref{inversinq}. We omit the detail of the proof.
\begin{lemma}\label{lemma3.1} Let $a_{h}(\cdot, \cdot)$ be the bilinear form as defined in\eqref{weakform}, Then, the following properties hold
\begin{itemize}
\item[\bf{(i)}]{\tt Continuity:} For all $y, \phi \in \mathpzc{W}_{h}$, for $t\in [0, T]$,
\begin{equation}
 \left|a_{h}(y, \phi)\right| \leq 2\,|\|y|\|_{\mathscr{E}} |\|\phi|\|_{\mathscr{E}}. \label{bddness} 
\end{equation} 
\item[\bf{(ii)}]{\tt Coercivity:} There exists a positive constant $c_{a}$ such that
\begin{equation}
a_{h}(\phi, \phi) \geq c_{a}\, |\|\phi|\|^{2}_{\mathscr{E}} , \quad \forall \phi \in \mathpzc{V}\cap\mathpzc{V}_{h}, \label{coercivity}
\end{equation}
with the following mesh-dependent energy norm
\begin{align}
|\| \phi |\|_{\mathscr{E}} &:=\Big[\sum_{K \in \mathscr{T}_{h}}\big(\|\nabla \phi\|_{L^2(K)}^{2}+a_0\|\phi\|_{L^2(K)}^{2}\big)\nonumber \\
&\hspace{0.5cm}+\sum_{E \in \mathcal{E}_{0,h} \cup \mathcal{E}_{D,h}}\big( h_{E} \|\smean{\nabla \phi}\|_{L^2(E)}^{2}+\frac{\sigma_0}{h_{E}}\| \sjump{\phi} \|_{L^2(E)}^{2}\big)\Big]^{1 / 2},\label{energynorm}
\end{align}
\end{itemize}
where $\sigma_0 \geq \tilde{\sigma}_0 > 0$, however, the constant  $\tilde{\sigma}_0$  may depend on the mesh and the degree of polynomial.
\end{lemma}
\noindent
Thus, we use the following notation for the time dependent energy-norm 
\begin{equation}
|\| \phi |\|_{L^2(I;~ \mathscr{E})}~=~\Big(\sum_{i=1}^{N_T}\int_{t_{i-1}}^{t_i} |\| \phi |\|_{\mathscr{E}} ^2\,dt\Big)^{\frac{1}{2}}.\hspace{4cm} \label{eng2.19}
\end{equation}
Further, we assume that the given functions $f$, $y_{d}$, $a_0$,  $q_{d}$, $g_{_N}, r_{_N}$,  $q_{a}$, and the upper bound $q_{b}$, respectively, are approximated by  $f_{h},\, y_{d, h},\, a_{0,h},\, q_{d,h},\, g_{_N,h},\, r_{_N,h},\, q_{a,h},\,\text{and}\; q_{b,h}$, where
\begin{equation}
f_{h},\,y_{h,0},\, y_{d, h},\, a_{0,h} \in L^2(I;\mathpzc{V}_{h}),\quad \text{and} \quad q_{d,h},\, g_{_N,h},\, r_{_N,h},\, q_{a,h},\, q_{b,h} \in L^2(I;\mathpzc{Q}_{Nad,h}). \label{givendata}
\end{equation}
While, the Dirichlet boundary data $g_{_D}$ is approximated by $g_{_D,h}\in L^2(I;\mathpzc{Q}_{D,h})$, where 
$$ \mathpzc{Q}_{D,h}=\{ q \in L^{2}(\Gamma_{_D}):\quad q|_{E} \in \mathbb{P}^{1}(E) \quad \forall E \in \mathcal{E}_{D,h}\}.$$
Then, the semidiscrete formulation of the problem  \eqref{weakcontfunc}--\eqref{initialcnd} can be written as 
\begin{subequations}
\begin{align}
&\underset{q_h \in \mathpzc{Q}_{Nad,h}}{\min} \int_0^T\Big\{\frac{1}{2}\int_{\O}(y_h-y_{d,h})^2 dx+\frac{\alpha}{2} \int_{\Gamma_{N}}(q_h-q_{d,h})^2ds+\int_{\Gamma_{N}} r_{_N,h}\, y_h ds\Big\}dt \label{discretfunc}\\
&\text{subject to}  \nonumber\\
& \hspace{2.2cm}\big(\frac{\partial y_h}{\partial t}, \phi_h)+a_{h}\left(y_{h}, \phi_h \right)=l_{h}\left(\phi_h \right)+b_{h}\left(q_{h}, \phi_h \right), \quad \phi_h \in \mathpzc{V}_{h},\label{discretformstate}\\
& \hspace{3.2cm} y_h(x,0)=y_{h,0},\label{disint}
\end{align}
\end{subequations}
where $y_{h,0}(x)$ is a suitable approximation or projection of $y_0(x)$. 

It follows that the control problem \eqref{discretfunc}--\eqref{disint} admits a unique solution $(y_h, q_h)\in \mathpzc{Y}_h\times \mathpzc{Q}_{Nad,h}$ if and only if the theere exist a discrete adjoint state $z_{h} \in \mathscr{Z}_{h}$ such that the following optimality conditions holds,
\begin{subequations}
\begin{align}
\big(\frac{\partial y_h}{\partial t}, \phi_h)+a_{h}\left(y_{h}, \phi_h \right)&= l_{h}\left(\phi_h \right)+b_{h}\left(q_{h}, \phi_h \right), \quad \phi_h \in \mathpzc{V}_{h},\label{disoptstate}\\
y_h(x,0)&=y_{h,0},\label{disint1}\\
-\big(\frac{\partial z_h}{\partial t}, \psi_h \big)+a_h(\psi_h, z_h)&=\left(y_h-y_{d,h}, \psi\right)+\left(r_{_N,h}, \psi_h \right)_{L^2(\Gamma_{_N})}, \quad \forall \psi_h \in \mathpzc{V}_h, \label{disoptadjoint-state}\\
z_h(x,T)&=0,\label{distfinal}\\
(\alpha \left(q_h-q_{d,h}\right)+z_h, \varphi_h-&q_h)_{L^2(\Gamma_{_N})} \geq 0, \quad \forall \varphi_h \in \mathpzc{Q}_{Nad,h} .\label{disfirstoptcond}
\end{align}
\end{subequations}
Analogous to the continuous case, introducing the discrete co-control $\mu_{h} \in \mathpzc{Q}_{Nad,h}$, allows us to restate the first-order optimality condition\eqref{disfirstoptcond} as follows
\begin{subequations}
\begin{align}
&\mu_{h}+\alpha \left(q_{h}-q_{d,h}\right)+z_{h}=0, \label{disopt11}\\
&\mu_{h}-\max \left\{0,  \mu_{h}+\gamma\left(q_{h}-q_{b,h}\right)\right\}+\min \left\{0, \mu_{h}+\gamma\left(q_{h}-q_{a, h}\right)\right\}=0.\label{disopt12}
\end{align}
\end{subequations}
By defining $\mu_{h}(x,t)=\mu_{b,h}(x,t)-\mu_{a,h}(x,t)$, for $(x,t)\in \mathcal{E}_{N,h} \times I$, we can equivalently express \eqref{disopt12} as the following discrete complementarity system, for $(x,t)\in \mathcal{E}_{N,h} \times I$,
\begin{subequations}
\begin{eqnarray}
&\mu_{a,h} \geq 0, \quad q_{a,h}- q_{h} \geq 0, \quad \mu_{a,h}\left(q_{a,h}-q_{h}\right)=0,\label{discompcondi1}\\
&\mu_{b,h} \geq 0, \quad q_{h}-q_{b,h} \geq 0, \quad \mu_{b,h}\left(q_{h}-q_{b,h}\right)=0.\label{discompcondi2}
\end{eqnarray}
\end{subequations}
We then define the discrete active and inactive sets as follows
\begin{subequations}
\begin{align}
&\mathpzc{A}_{a, h}=\bigcup\left\{(x,t)\in E\times I \mid \mu_{h}+\gamma\left(q_{h}-q_{a,h}\right)<0, \quad \forall E \in \mathcal{E}_{N,h}\right\}, \label{activeset1}\\
&\mathpzc{A}_{b, h}=\bigcup\left\{(x,t) \in E\times I \mid \mu_{h}+\gamma\left(q_{h}-q_{b,h}\right)>0, \quad \forall E \in \mathcal{E}_{N,h}\right\},\label{activeset2} \\
\text{and}& \nonumber\\
&\mathcal{I}_{h}=\mathcal{E}_{N,h}\times I \backslash\left\{\mathpzc{A}_{a, h} \cup \mathpzc{A}_{b, h}\right\},
\end{align}
\end{subequations}
respectively. Analogous to the continuous case, the complementarity conditions \eqref{discompcondi1}--\eqref{discompcondi2} can be restated as
\begin{subequations}
\begin{align}
&q_{h}=q_{a,h}, \quad \mu_{b,h}=0, \quad \mu_{h} \leq 0~~ \quad \text { on }~~ \mathpzc{A}_{a, h}, \\
&q_{h}=q_{b,h}, \quad \mu_{a,h}=0, \quad \mu_{h} \geq 0~~ \quad \text { on }~~ \mathpzc{A}_{b, h}, \\
q_{a,h} &< q_{h}<q_{b,h}, \quad \mu_{a,h}=\mu_{b,h}=0, \quad \mu_{h}=0~~  \text { on }~~ \mathcal{I}_{h}.
\end{align}
\end{subequations}
We proceed to develop the space-time approximations of the optimal BCP  \eqref{contfunc}-\eqref{contspace} utilizing the combination of backward Euler time-stepping scheme with the discontinuous Galerkin spatial discretization method.\smallskip

\noindent
\textbf{\it Space-time discretization.} 
We consider a partition of the time domain $I=[0,T]$ into $N$ subintervals, denoted as $I_i=(t_{i-1},t_i]$ such that $I=\cup_{i=1}^{N_T} \bar{I}_i$, where $0=t_0<t_1<t_2<\ldots<t_N=T$. The time step size is given by  $k_i=t_i-t_{i-1}$, $i=1,\,2,\ldots,\,{N_T}$. 

Furthermore, we construct $\mathscr{T}^i_h$, $i=1,\,2,\,\ldots,\,N_T,$ (similar to $\mathscr{T}_h$) as the triangulation of $\bar{\Omega}$ at the time level $t_i$. At each time level $t_i$, we introduce the set of all edges $\mathcal{E}^i_{h}$, which can be decomposed into interior edges $\mathcal{E}^i_{0, h}$ and boundary edges $\mathcal{E}^i_{B,h}$. The boundary edges are further divided into Dirichlet boundary edges $\mathcal{E}^i_{D,h}$ and Neumann boundary edges $\mathcal{E}^i_{N,h}$, respectively. 

Consistent with the semi-discrete framework, we introduce the finite-dimensional spaces $\mathpzc{Y}_h^i$, $\mathscr{Z}_h^i$, and $\mathpzc{Q}^i_{Nad,h}$ associated with the triangulation $\mathscr{T}^i_h$ at each time level $t_i$, for $i=[1:N_T]$, by
\begin{subequations}
\begin{align}
&\mathpzc{Y}^i_{h}=\mathscr{Z}^i_{h}=\left\{v \in L^{2}(\Omega):\quad v|_{K} \in \mathbb{P}_{1}(K) \quad \forall K \in \mathscr{T}_{h}^i \right\}, \label{femspacestatefully}\\
\text{and} & \nonumber\\
&\mathpzc{Q}^i_{N,h}=\big\{q \in L^2(\Gamma_{_N}):\quad q|_{E} \in \mathbb{P}_{1}(E) \quad \forall E \in \mathcal{E}^i_{N,h}\big\},\label{femspacecontrolfully}
\end{align}
\end{subequations}
respectively. We set $\mathpzc{Q}^i_{Nad,h}:=\mathpzc{Q}_{Nad}\cap \mathpzc{Q}^i_{N,h}, i= [1:N_T].$

Then, the space-time discretization of the OCP \eqref{discretfunc}--\eqref{disint} is given by 
\begin{subequations}
\begin{align}
&\underset{q^i_h \in \mathpzc{Q}^i_{Nad,h}}{\min}\sum_{i=1}^{N_T} k_i \Big\{\frac{1}{2}\int_\O (y^i_h-y^i_{d,h})^2dx+\frac{\alpha}{2} \int_{\Gamma_{N}}(q^i_h-q^i_{d,h})^{2}ds+\int_{\Gamma_{N}} r^i_{_N,h}\, y^i_h ds\Big\} \label{fullydiscretfunc}\\
\text {over} & \nonumber\\
&~~\big(\frac{y^i_h-y^{i-1}_h}{k_i}, \phi_h)+a_{h}\left(y^i_{h}, \phi_h \right)=l^i_{h}\left(\phi_h \right)+b_{h}\left(q^i_{h}, \phi_h \right), \quad \phi_h \in \mathpzc{V}^i_{h},\; i\in[1:N_T], \label{fullydiscretformstate}\\
&~~~~y_h(x,0)=y_{h}^0. \label{fullydisint}
\end{align}
\end{subequations}
The fully-discrete problem \eqref{fullydiscretfunc}--\eqref{fullydisint} has a unique solution $(y_h^i,q^i_h)\in \mathpzc{Y}_h^i\times \mathpzc{Q}^i_{Nad,h}$, $i\in [1:N_T]$ iff there exists a unique $z_h^{i-1}\in \mathscr{Z}_h^i,\, i\in [N_T:1]$ such that the following optimality conditions satisfies, $\forall \phi_h,\, \psi_h \in \mathpzc{V}_h$, 
\begin{subequations}
\begin{align}
\big(\frac{y^i_h-y^{i-1}_h}{k_i}, \phi_h)+a_{h}\left(y^i_{h}, \phi_h \right)&=l^i_{h}\left(\phi_h \right)+b_{h}\left(q^i_{h}, \phi_h \right), \quad   1\leq i \leq N_T,\label{disoptstate12}\\ 
y_h(x,0)&=y_{h,0},\label{disint12}\\
-\big(\frac{z^i_h-z^{i-1}_h}{k_i}, \psi_h \big)+a_h(\psi_h, z^{i-1}_h)&=\left(y^i_h-y^i_{d,h}, \psi\right)+\left(r^i_{_N,h}, \psi_h \right)_{L^2(\Gamma_{_N})}, \quad 
N_T\leq i \leq 1, \label{disoptadjoint-state12}\\ 
z_h(x,t_N)&=0,\label{distfinal12}\\
(\alpha \left(q^i_h-q^i_{d,h}\right)+z^{i-1}_h, \varphi_h&-q^i_h)_{L^2(\Gamma_{_N})} \geq 0, \quad \forall \varphi_h \in \mathpzc{Q}^i_{Nad,h}, \quad 1\leq i \leq N_T. 
 \label{disfirstoptcond12} 
 \end{align}
\end{subequations}
For each time subinterval $I_i,\; i=1,\,2\,\ldots,\,N_T$,
we define the piecewise linear and continuous time interpolants $Y_h(t)$ and $Z_h(t)$, valid for all   $t\in I_i,$ by
\begin{align*}
& Y_h(t)|_{t\in I_i}:=\frac{t-t_{i-1}}{k_i}y_h^i+\frac{t_i-t}{k_i}y_h^{i-1}, \quad \text{and}  \quad 
Z_h(t)|_{t\in I_i}:=\frac{t-t_{i-1}}{k_i}z_h^i+\frac{t_i-t}{k_i}z_h^{i-1},  
\end{align*}
respectively. In addition, we set $Q_h(t)|_{t\in I_i}:=q_h^i$. Similarly, for $t\in I_i,\; i=1,\,2\,\ldots,\,N_T$, we define
\begin{eqnarray}
&&\mathcal{G}_{D,h}(t)~:=~\frac{t-t_{i-1}}{k_i}g_{_D,h}^i+\frac{t_i-t}{k_i}g_{_D,h}^{i-1},\quad
\mathcal{G}_{N,h}(t)~:=~\frac{t-t_{i-1}}{k_i}g_{_N,h}^i+\frac{t_i-t}{k_i}g_{_N,h}^{i-1}, \nonumber \\
&& \text{and} \quad \mathcal{R}_{N,h}(t)~:=~\frac{t-t_{i-1}}{k_i}r_{_N,h}^i+\frac{t_i-t}{k_i}r_{_N,h}^{i-1},  \nonumber
\end{eqnarray}
respectively. For any function $\omega \in C(I,L^2(\Omega))$, let $\hat{\omega}(x,t)|_{t\in (t_{i-1},t_i]}=\omega(x,t_i)$ and $\tilde{\omega}(x,t)|_{t\in (t_{i-1},t_i]}=\omega(x,t_{i-1})$ denote the restrictions of $\omega(x,\cdot)$ to the interval $(t_{i-1},t_i]$. Then, for each time subinterval $I_i,\; i=1,\,2\,\ldots,\,N_T$, the optimality conditions can be restated as follows
\begin{subequations}
\begin{align}
\big(\frac{\partial Y_h}{\partial t}, \phi_h)+a_{h}(\hat{Y}_h, \phi_h)&=\hat{l}_{h}\left(\phi_h \right)+b_{h}\left(Q_h, \phi_h \right), \quad  \forall \phi_h \in \mathpzc{V}^i_{h},\, 1\leq i \leq N_T,\label{redisoptstate12}\\ 
Y^0_h&=y_{h,0},\label{redisint12}\\
-\big(\frac{\partial Z_h}{\partial t}, \psi_h \big)+a_h(\psi_h, \tilde{Z}_h)&=\big(\hat{Y}_h-\hat{y}_{d,h}, \psi_h\big)+\left(\hat{\mathcal{R}}_{N,h}, \psi_h \right)_{L^2(\Gamma_{_N})},\,\forall \underset{ N_T\leq i \leq 1,}{\psi_h \in \mathpzc{V}^i_h,} \label{redisoptadjoint-state12}\\ 
Z_h^N&=0,\label{redistfinal12}\\
\big(\alpha \left(Q_h-\hat{q}_{d,h}\right)+\tilde{Z}_h,& \varphi_h-Q_h\big)_{L^2(\Gamma_{_N})} \geq 0, \quad \forall \varphi_h \in \mathpzc{Q}^i_{Nad,h},\, 1\leq i \leq N_T. 
\label{redisfirstoptcond12}
\end{align}
\end{subequations}
By introducing the discrete co-control $\hat{\mu}_{h} \in \mathpzc{Q}^i_{Nad,h}$,  for each $t\in  I_i,\,i=1,\,2,\ldots, N_T$, we can restate the first-order optimality condition\eqref{redisfirstoptcond12}  as follows
\begin{subequations}
\begin{align}
&\hat{\mu}_{h}+\alpha \left(Q_{h}-\hat{q}_{d,h}\right)+\tilde{Z}_{h}=0, \label{redisopt11}\\
&\hat{\mu}_{h}-\max \left\{0, \hat{\mu}_{h}+\gamma\left(Q_{h}-\hat{q}_{b,h}\right)\right\}+\min \left\{0, \mu_{h}+\gamma\left(Q_h-\hat{q}_{a,h}\right)\right\}=0.\label{redisopt12}
\end{align}
\end{subequations}
In analogy to the continuous case, an equivalent formulation of \eqref{redisopt12}  with $\hat{\mu}_{h}=\hat{\mu}_{b,h}-\hat{\mu}_{a,h}$ can be expressed as follows, for $(x,t)\in E\times I_i, \forall E\in \mathcal{E}^i_{N,h}, i\in [1:N_T]$,
\begin{subequations}
\begin{align}
&\hat{\mu}_{a,h}\geq 0, \quad \hat{q}_{a,h}- Q_h \geq 0, \quad \hat{\mu}_{a,h}\left(\hat{q}_{a,h}-Q_{h}\right)=0,\label{discompcondi111}\\
&\hat{\mu}_{b,h} \geq 0, \quad Q_{h}-\hat{q}_{b,h} \geq 0, \quad \hat{\mu}_{b,h}\left(Q_{h}-\hat{q}_{b,h}\right)=0.\label{discompcondi222}
\end{align}
\end{subequations}
For $i\in [1:N_T],$ the active  and inactive sets are given by 
\begin{subequations}
\begin{align}
&\mathpzc{A}^i_{a, h}=\bigcup\left\{(x,t) \in E\times I_i \mid \hat{\mu}_{h}+\gamma\left(Q_{h}-\hat{q}_{a,h}\right)<0, \; \forall E \in \mathcal{E}^i_{N,h}\right\} \label{activeset111}\\
&\mathpzc{A}^i_{b, h}=\bigcup\left\{(x,t) \in E\times I_i \mid \hat{\mu}_{h}+\gamma\left(Q_{h}-\hat{q}_{b,h})\right)>0, \; \forall E \in \mathcal{E}^i_{N,h}\right\}\label{activeset222}\\
&\mathcal{I}^i_{h}=\mathcal{E}^i_{N,h}\times I_i \backslash\{\mathpzc{A}^i_{a, h} \cup \mathpzc{A}^i_{b, h}\}.
\end{align}
\end{subequations}
 Furthermore, for $i\in [1:N_T]$, the complementarity conditions \eqref{discompcondi111}--\eqref{discompcondi222} can be compactly expressed as
 \begin{subequations}
\begin{align}
&Q_{h}=\hat{q}_{a,h}, \quad \hat{\mu}_{b,h}=0, \quad \hat{\mu}_{h}\leq 0~~ \quad \text { on }~~ \mathpzc{A}^i_{a, h}, \\
&Q_{h}=\hat{q}_{b,h}, \quad \hat{\mu}_{a,h}=0, \quad \hat{\mu}_{h} \geq 0~~ \quad \text { on }~~ \mathpzc{A}^i_{b, h}, \\
&\hat{q}_{a,h}<Q_{h}<\hat{q}_{b,h}, \quad \hat{\mu}_{a,h}=\hat{\mu}_{b,h}=0, \quad \hat{\mu}_{h}=0~~ \text { on }~~ \mathcal{I}^i_{h}.
\end{align}
\end{subequations}
The subsequent error analysis will frequently rely on the following inverse and trace inequalities (cf., \cite{brenner2002, pascal2007}), which are necessary for the analysis
\begin{lemma}
Let $\mathpzc{D}$ be a bounded domain (or polygonal) with  sufficiently smooth  boundary $\Gamma_{_\mathpzc{D}}$. Then there exists the positive constant $c_{t r}$ such that the following estimates holds, for any $\chi\in H^2(\mathpzc{D})\cap H^1(\mathpzc{D})$, 
\begin{align}
&\|\chi\|_{0, \Gamma_{_\mathpzc{D}}} ~\leq~ c_{t r}\|\chi\|_{1, \mathpzc{D}}, \quad \forall \chi\in H^{1}(\mathpzc{D}), \label{traceinq}\\
&\|\chi\|_{0, \Gamma_{_\mathpzc{D}}} ~\leq~ c_{t r}\left(h_{\mathpzc{D}}^{-1}\|\chi\|_{0, \mathpzc{D}}^{2}+h_{\mathpzc{D}}\|\nabla \chi\|_{0, \mathpzc{D}}^{2}\right)^{1 / 2}, \quad \forall \chi \in H^{1}(\mathpzc{D}),\label{tracegeninq}
\end{align}
and, the inverse estimate
\begin{equation}
|\chi|_{j, \mathpzc{D}} \leq c_{inv}\, h_{\mathpzc{D}}^{s-j}|\chi|_{i, \mathpzc{D}} \quad \forall \chi \in \mathbb{P}_{k}(\mathpzc{D}), \quad 0 \leq s \leq j \leq 2.\hspace{0.3cm} \label{inversinq}
\end{equation}  
\end{lemma}
\noindent
Further, we exploit the  approximation results from \cite[Proposition $2.4$ and Theorem $4.3$]{baker1995}.
Define  $\psi=e-\phi_h$, and let  $\phi_h$ be the best piecewise constant approximation of $e$.  Then there exist a positive constant $C$ such that
\begin{eqnarray}
\|\psi\|_{L^2(K)}\leq C h_K\|\nabla e\|_{L^2(K)},\quad K\in \mathscr{T}_h.\nonumber 
\end{eqnarray} 
Here, $e$ represents either the state error $e_y(=y(Q_h)-Y_h)$ or the adjoint-state error $e_z(=z(Q_h)-Z_h)$.
An application of the trace inequality\eqref{tracegeninq} leads to
\begin{eqnarray}
\sum_{K\in \mathscr{T}_h}h_K^{-2}\|\psi\|^2_{L^2(K)}\leq C\sum_{K\in \mathscr{T}_h} \|\nabla e\|^2_{L^2(K)}.\label{appinq4.24} 
\end{eqnarray}
Additionally, exploiting the mesh property $h_E^{-1}h_K\leq 1$ to have
\begin{align}
\sum_{E\in \mathcal{E}_{0,h}}h_E^{-1}\big(\|\smean{\psi}\|^2_{L^2(E)}+\|\sjump{\psi}\|^2_{L^2(E)}\big)
&\leq C \sum_{K\in \mathscr{T}_h} \|\nabla e\|^2_{L^2(K)},\label{appinq4.25} \\
\text{and} \hspace{4.4cm}
\sum_{E\in \mathcal{E}_{B,h}}h_E^{-1}\|\psi\|^2_{L^2(E)} 
&\leq C\sum_{K\in \mathscr{T}_h} \|\nabla e\|^2_{L^2(K)}.\label{appinq4.26} 
\end{align}
We now proceed to the error analysis of the finite element approximation for the  BCP governed by equations\eqref{contfunc}-\eqref{contspace}.
\section{A posteriori Error Analysis}\label{sec4errest}
This section focuses on the development of residual-based a posteriori error estimators for quantifying errors in the state, adjoint-state, control, and co-control, respectively. Specifically, we measure the state and adjoint-state errors in the $|\|\cdot\||_{L^2(I;\mathscr{E})}$ as defined in \eqref{eng2.19}, while the control $(q)$ and co-control $(\mu)$ errors are measured in the $L^2(I;L^{2}(\Gamma_{_N}))$-norm.
To initiate the error analysis for the upper bounds, we first introduce the residual-type error estimators tailored to the PBCP governed by equations \eqref{contfunc}-\eqref{contspace}, as
\begin{enumerate}
\item {\tt Element-wise residuals:} 
We begin by introducing element-wise residuals for the state and adjoint variables, denoted by
$\eta_{y,K}^i,\, \text{and}\; \eta_{z,K}^i$ ,\;$i\in [1:N_T]$,  respectively. For $ K\in \mathscr{T}_h^i,\;i\in [1:N_T],$ these residuals are defined by, for  $i\in [1:N_T]$,  
\begin{subequations}\label{spacest4.4}
\begin{align}
(\eta_{y,K}^i)^2&:=\int^{t_i}_{t_{i-1}}h^2_K\|\hat{f}_h-Y_{h,t}+\Delta\hat{Y}_h-a_{0,h}\hat{Y}_h\|^2_{L^2(K)}\,dt,\\
(\eta_{z,K}^i)^2&:=\int^{t_i}_{t_{i-1}}h_K^2\|\hat{Y}_h-\hat{y}_{d,h}+Z_{h,t}+\Delta \tilde{Z}_h-a_{0,h}\tilde{Z}_h\|^2_{L^2(K)}\,dt.
\end{align}
\end{subequations}
\item {\tt Interior edge residuals:} 
Let $\eta^i_{y,E_0} $ and $\eta^i_{z,E_0}$ be the interior edge residuals corresponding to state and adjoint-state variables, respectively, associated with each interior edge $E_0\in \mathcal{E}_{0,h}^i$, are defined as, for $i\in [1:N_T],$ 
\begin{subequations}
\begin{align}
&(\eta^i_{y,E_0})^2:=\int^{t_i}_{t_{i-1}}\big[h_{E_0}\|\sjump{\nabla \hat{Y}_h}\|^2_{L^2(E_0)}+\frac{\sigma^2_0}{h_{E_0}}\|\sjump{\hat{Y}_h}\|^2_{L^2(E_0)}\big]\,dt, \\
&(\eta^i_{z,E_0})^2:=\int^{t_i}_{t_{i-1}}\big[h_{E_0} \|\sjump{\nabla \tilde{Z}_h}\|^2_{L^2(E_0)}+\frac{\sigma^2_0}{h_{E_0}} \| \sjump{\tilde{Z}_h}\|^2_{L^2(E_0)}\big]\,dt, \displaybreak[0]\\
&(\tilde{\eta}_{y,E_0}^i)^2:=\int^{t_i}_{t_{i-1}}\big[\frac{\sigma^2_0}{ h_{E_0}}\|\sjump{\hat{Y}_h}\|^2_{L^2(E_0)}+\frac{\sigma_0^2}{ h_{E_0}}\|\sjump{Y_h}\|^2_{L^2(E_0)}\big]\,dt ,\\ 
&(\tilde{\eta}_{z,E_0}^i)^2:=\int^{t_i}_{t_{i-1}}\big[\frac{\sigma_0^2}{h_{E_0}}\| \sjump{\tilde{Z}_h}\|^2_{L^2(E_0)}+\frac{\sigma_0^2}{h_{E_0}}\|\sjump{Z_h}\|^2_{L^2(E_0)}\big]\,dt. 
\end{align}
\end{subequations}
\item {\tt Dirichlet and Neumann boundary edge residuals:} We define the boundary edge residuals, which are associated with $E_D\in \mathcal{E}_{_D,h}^i$ and $E_N\in \mathcal{E}_{_N,h}^i $, $i\in [1:N_T]$, respectively, and given as
\begin{itemize}
\item[(i)] Dirichlet boundary edges residuals $\eta^i_{y,g_D}$, $\eta^i_{y,E_D}$ and $\eta^i_{z,E_D}$, for $i\in [1:N_T]$, by 
\begin{subequations}
\begin{align}
&(\eta^i_{y,g_D})^2:=\int^{t_i}_{t_{i-1}}\frac{\sigma_0^2}{ h_{E_D}}\|\hat{\mathcal{G}}_{D,h}-\hat{Y}_h\|^2_{L^2(E_D)}\,dt,\\
&(\eta^i_{y,E_D})^2:=\int^{t_i}_{t_{i-1}}\big[\frac{\sigma_0^2}{ \,h_{E_D}}\|\hat{Y}_h\|^2_{L^2(E_D)}+\frac{\sigma^2_0}{h_{E_D}}\|Y_h\|^2_{L^2(E_D)}\big]\,dt,\\
&(\eta^i_{z,E_D})^2:=\int^{t_i}_{t_{i-1}} \big[\frac{\sigma^2_0}{h_{E_D}}\| \tilde{Z}_h\|^2_{L^2({E_D})}+ \frac{\sigma_0^2}{h_{E_D}}\|Z_h\|^2_{L^2({E_D})}\big]\,dt,  
\end{align}
\item[(ii)] and, Neumann boundary edge residuals $\eta^i_{y,E_N} $,  $\eta^i_{z,E_N}$,  and $\eta^i_{q,E_N}$, for $i\in [1:N_T]$,  by 
\begin{align}
&(\eta^i_{y,E_N})^2:=\int^{t_i}_{t_{i-1}}h_{E_N}\|Q_h+\hat{\mathcal{G}}_{N,h}-\mathbf{n}_E\cdot \nabla \hat{Y}_h\|^2_{L^2(E_N)}\,dt,\label{4.12yen}\\
&(\eta^i_{z,E_N})^2:= \int^{t_i}_{t_{i-1}} h_{E_N} \|\hat{\mathcal{R}}_{N,h}- \mathbf{n}_{E_N}\cdot \nabla \tilde{Z}_h\|^2_{L^2(E_N)}\,dt, \label{4.13en}\\
&(\eta^i_{q,E_N})^2:= \int^{t_i}_{t_{i-1}} h_{E_N}^2 \|\mathbf{n}_{E_N}\cdot \nabla(\alpha \left(Q_h-\hat{q}_{d,h}\right)+\tilde{Z}_h)\|^2_{L^2(E_N)}\,dt, \label{4.14444qen}
\end{align}
\end{subequations}
\end{itemize}
\item Moreover, data oscillations in the error analysis are represented by
\begin{subequations}\label{dataest4.17}
\begin{align}
&\varTheta_y^2~:=~\sum_{i=1}^{N_T}\int_{t_{i-1}}^{t_i}\Big[\sum_{K\in \mathscr{T}_h^i}\underbrace{h_K^2\big(\|(a_{0,h}-a_0)\hat{Y}_h\|^2_{L^2(K)}+\|f-\hat{f}_h\|^2_{L^2(K)}\big)}_{(\varTheta_{y,K}^i)^2}\nonumber\\
&\hspace{1.5cm}+\sum_{E\in \mathcal{E}_{N,h}}\underbrace{h_E\|g_N-\hat{\mathcal{G}}_{N,h}\|^2_{L^2(E)}}_{(\varTheta_{y,E_N}^i)^2}+\sum_{E\in \mathcal{E}^i_{D,h}}\underbrace{\frac{\sigma_0}{h_E}\|g_D-\mathcal{G}_{D,h}\|^2_{L^2(E)}}_{(\varTheta_{y,E_D}^i)^2}\Big]dt, \\  
&\varTheta_z^2~:=~\sum_{i=1}^{N_T}\int_{t_{i-1}}^{t_i}\Big[\sum_{K\in \mathscr{T}_h^i}\underbrace{h_K^2\big(\|(a_{0,h}-a_0)\tilde{Z}_h\|^2_{L^2(K)}+\|\hat{y}_{d,h}-y_d\|^2_{L^2(K)}\big)}_{(\varTheta_{z,K}^i)^2}\nonumber\\
&\hspace{1.5cm}+\sum_{E\in \mathcal{E}_{_N,h}^i}\underbrace{h_E\|r_{_N}-\hat{\mathcal{R}}_{N,h}\|^2_{L^2(E)}}_{(\varTheta_{z,E_N}^i)^2}\Big]dt,  \displaybreak[0]\\
&\varTheta_q^2 ~:=~ \sum_{i=1}^{N_T}\int_{t_{i-1}}^{t_i}\sum_{E\in \mathcal{E}_{N,h}^i}\Big[\underbrace{\|\alpha (q_d-\hat{q}_{d,h})\|^2_{L^2(E)}+\|q_a-\hat{q}_{a,h}\|^2_{L^2(E)}+\|q_b-\hat{q}_{b,h}\|^2_{L^2(E)}}_{(\varTheta^i_{q,E_N})^2} \Big]dt.\label{4.17qen}
\end{align}
\end{subequations}
\item We define the temporal error estimators for the state and adjoint-state variables as $\varTheta_{y,T}$ and $\varTheta_{z,T}$ respectively, which are expressed as 
\begin{subequations}\label{tempest4.18}
\begin{align}
\varTheta_{y,T}^2 ~:=~\sum_{i=1}^{N_T}\int_{t_{i-1}}^{t_i}|\|Y_h-\hat{Y_h}|\|^2_{\mathscr{E}}\, dt, \quad \text{and} \quad
\varTheta_{z,T}^2~:=~\sum_{i=1}^{N_T}\int_{t_{i-1}}^{t_i}|\|Z_h-\tilde{Z_h}|\|^2_{\mathscr{E}}\, dt,
\end{align}
\end{subequations}
respectively. Further, we define $$\Xi_{T_{yz}}:=\big(\varTheta_{y,T}^2+\varTheta_{z,T}^2 \big)^{1/2}.$$
\item Ultimately, we introduce the global error indicators for the state, adjoint-state, and control variables, defined as
\begin{subequations}\label{sumofest4.19}
	\begin{align}
		&\eta_y^2~:=~\sum_{K\in \mathscr{T}_h^i}(\eta_{y_0,K})^2+ \sum_{i=1}^{N_T}\Big[\sum_{K\in\mathscr{T}_h^i} (\eta_{y,K}^i)^2
		+\sum_{E\in \mathcal{E}^i_{N,h}} (\eta^i_{y,E_N})^2+\sum_{E\in \mathcal{E}^i_{D,h}}(\eta^i_{y,g_D})^2\nonumber\\
        &\hspace{1cm}+\sum_{E\in \mathcal{E}^i_{0,h}}(\eta^i_{y,E_0})^2+\sum_{E\in \mathcal{E}^i_{0,h}}(\tilde{\eta}^i_{y,E_0})^2+\sum_{E\in \mathcal{E}^i_{D,h}}(\eta^i_{y,E_D})^2\Big],\\
		&\eta^2_z~:=~\sum_{i=1}^{N_T}\Big[\sum_{K\in \mathscr{T}_h^i} (\eta^i_{z,K})^2+\sum_{E\in \mathcal{E}_{N,h}^i }(\eta^i_{z,E})^2+\sum_{E\in \mathcal{E}_{0,h}^i}(\eta^i_{z,E})^2+\sum_{E\in \mathcal{E}_{D,h}^i }(\eta^i_{z,E})^2
		\nonumber\\
        &\hspace{1cm}+\sum_{E\in \mathcal{E}_{0,h}^i }(\tilde{\eta}_{z,E}^i)^2\Big],\\ 
		&\eta^2_q~:=~ \sum_{i=1}^{N_T}\sum_{E_N\in \mathcal{E}_{N,h}^i}(\eta^i_{q,E_N})^2,\label{4.19cq}
	\end{align}
\end{subequations}
where $\eta_{y_0,K}$ denotes the initial data error estimator and defined as
\begin{align}
\eta_{y_0,K}~:=~\|y_0-y_{0,h}\|_{L^2(K)}. \hspace{2cm}
\end{align}
\end{enumerate}
Moreover, the sum of the residual error estimators for the SIPG discretization of the PBCP is given by
\begin{eqnarray}
\Upsilon_{yzq}~:=~ (\eta_{y}^2+\eta_z^2+\eta_q^2)^{1/2}, \label{upsilon} \hspace{0.5cm}
\end{eqnarray}
and, $\varTheta$ related to data oscillations is defined by
\begin{eqnarray}
\varTheta_{yzq}~:=~ (\varTheta^2_y+\varTheta^2_z+\varTheta^2_q)^{1/2}.\label{vartheta} \hspace{0.3cm}
\end{eqnarray}
\subsection{Reliable-type A posteriori Error Estimation}\label{section4444}
This section focuses on deriving reliability-type a posteriori error estimates for the  OCP governed by equations \eqref{contfunc}--\eqref{contspace}. 

To achieve this, we begin by introducing intermediate problems for the state and adjoint-state variables.
For given control variable $q^* \in \mathpzc{Q}_{Nad}$, to determine a pair $(y(q^*), z(q^*))\in \mathpzc{Y}\times \mathscr{Z}$ that satisfies the following control system
\begin{eqnarray}
\big(\frac{\partial y(q^*)}{\partial t}, \phi \big)+a(y(q^*), \phi)&=&(f,  \phi)+\left(q^*+g_{_N}, \phi \right)_{L^2(\Gamma_{_N})}  \quad \forall \phi \in \mathpzc{V},\label{intstate} \\
y(q^*)(x,0)&=&y_0(x),\label{intinitial}\\
-\big(\frac{\partial z(q^*)}{\partial t}, \psi \big)+a(\psi, z(q^*))&=&\left(y(q^*)-y_{d}, \psi\right)+\left(r_{_N}, \psi \right)_{L^2(\Gamma_{_N})} \quad \forall \psi \in \mathpzc{V}, \label{intadjoint-state}\\
z(q^*)(x,T)&=&0.\label{intfinal}
\end{eqnarray}
To derive the error equations for the state and adjoint-state variables, we subtract the intermediate state and adjoint-state equations \eqref{intstate} and \eqref{intadjoint-state} with $q^*=Q_h$ from the continuous state and adjoint-state \eqref{weakformstate} and \eqref{weakformadjoint-state}, respectively. This yields
\begin{align}
\Big(\frac{\partial (y-y(Q_h))}{\partial t}, \phi \Big)+a(y-y(Q_h), \phi)&=\left(q-Q_h, \phi \right)_{L^2(\Gamma_{_N})},  \quad \forall \phi \in \mathpzc{V}, \label{erroreqstate}\\
-\Big(\frac{\partial (z-z(Q_h))}{\partial t}, \psi \Big)+a(\psi, z-z(Q_h))&=(y-y(Q_h), \psi), \quad \forall \psi \in \mathpzc{V}.\label{erroreqadjoint-state} 
\end{align}
We do the integration over $[0,T]$ by setting $\phi= y-y(Q_h)$ and $\psi=z-z(Q_h)$ in \eqref{erroreqstate} and \eqref{erroreqadjoint-state}. Then we use Lemma \ref{lemma3.1} and the trace inequality \eqref{traceinq} to get
\begin{align}
\frac{1}{2}\|(y-y(Q_h))(T)\|^2+c_a\int_0^T|\|y-y(Q_h)|\|^2_{\mathscr{E}}\, dt&\leq\|q-Q_h\|_{L^2(I;L^2(\Gamma_{_N}))} \|y-y(Q_h)\|_{L^2(I;L^2(\Gamma_{_N}))}\nonumber\\ 
&\leq c_0c_{tr}\|q-Q_h\|_{L^2(I;L^2(\Gamma_{_N}))} |\|y-y(Q_h)|\|_{L^2(I;\mathscr{E})}\nonumber\\ 
|\|y-y(Q_h)|\|_{L^2(I; \mathscr{E})}&\leq \frac{c_0c_{tr}}{c_a}\|q-Q_h\|_{L^2(I;L^2(\Gamma_{_N}))}, \label{statebound}
\end{align}
where $c_0$ is defined as the minimum of $a_0$ and its reciprocal $a_0^{-1}$. Likewise,
\begin{align}
\frac{1}{2}\|(z-z(Q_h))(0)\|^2+c_a\int_0^T|\|z-z(Q_h)|\|^2_{\mathscr{E}}\,dt&\leq\|y-y(Q_h)\|_{L^2(I;L^2(\Omega))}\|z-z(Q_h)\|_{L^2(I;L^2(\Omega))}\nonumber\\
|\|z-z(Q_h)|\|_{L^2(I; \mathscr{E} )}&\leq |\|y-y(Q_h)|\|_{L^2(I; \mathscr{E})}.
\label{adjoint-statebound} 
\end{align} 
The subsequent lemma establishes a reliable-type a posteriori error estimate.
\begin{lemma}\label{lm4.1int} 
Let $(y,z,q)$ and $(Y_h, Z_h,Q_h)$ be the solutions of \eqref{weakformstate}--\eqref{weakformcontrol} and \eqref{redisoptstate12}--\eqref{redisfirstoptcond12}, respectively, and the co-control  variable $\mu$ and the discrete co-control variable $\mu_h$ as defined in \eqref{optmallitycon1} and \eqref{redisopt11}, respectively. Further, we assume that, for all $i\in[1:N_T]$, $\mathpzc{Q}^i_{Nad,h} \subset \mathpzc{Q}_{Nad,h}$, $\big(\alpha \left(Q_h-\hat{q}_{d,h}\right)+\tilde{Z}_h\big)|_{E\in \mathcal{E}^i_{N,h}} \in H^1(E)$ and then there is a positive constant $C_{4,1}$ and $\varphi_h \in \mathpzc{Q}^i_{Nad,h},\, i \in [1:N_T]$ such that 
\begin{align}
&\int_0^T\big|\big(\alpha \left(Q_h-\hat{q}_{d,h}\right)+\tilde{Z}_h, \varphi_h-Q_h\big)_{L^2(\Gamma_{_N})}|\,dt\nonumber\\
&\hspace{1cm}\leq~C_{4,1} \int_0^T \sum_{E \in \mathcal{E}^i_{N,h}} h_E \|\mathbf{n}_E\cdot \nabla(\alpha \left(Q_h-\hat{q}_{d,h}\right)+\tilde{Z}_h)\|_{L^2(E)}\|q-Q_h\|_{L^2(E)}\,dt.\label{assmption11}
\end{align}
Then, we have
\begin{align}
&|\|y-Y_h|\|_{L^2(I;\mathscr{E})}+|\|z-Z_h|\|_{L^2(I;\mathscr{E})}+\|q-Q_h\|_{L^2(I;L^2(\Gamma_{_N}))}+\|\mu-\hat{\mu}_h\|_{L^2(I;L^2(\Gamma_{_N}))}\nonumber\\ & \hspace{1.3cm} \leq C_{4,4}\,\eta_q+C_{4,5}\, \|\alpha(q_d-\hat{q}_{d,h})\|_{L^2(I;L^2(\Gamma_{_N}))}+C_{4,6}\,\|\nabla(Z_h-\tilde{Z}_h)\|_{L^2(I;L^2(\Omega))}\nonumber\\
&\hspace{1.3cm}+C_{4,7}\,|\|z(Q_h)-Z_h|\|_{L^2(I; \mathscr{E})}
+C_{4,8}\,|\|y(Q_h)-Y_h|\|_{L^2(I;\mathscr{E})}.
\end{align}
where $C_{4,j},\, j=4,\,5,\ldots,\,8$ are positive constants, and the pair $(y(Q_h),z(Q_h))\in \mathpzc{Y}\times \mathscr{Z}$ is a solution of the problem \eqref{intstate}--\eqref{intfinal} with $q^*=Q_h$.
\end{lemma}
\begin{proof} The proof of this lemma proceeds in the following steps.	We first deduce the bounds for the state and adjoint errors. \smallskip

\noindent
{(i) \tt Bounds for the state and adjoint-state errors.}
Applying the energy norm definition \eqref{energynorm} and the triangle inequality, together with the bound \eqref{statebound}, yields
\begin{align}
|\|y-Y_h|\|_{L^2(I;\mathscr{E})}&\leq |\|y-y(Q_h)|\|_{L^2(I;\mathscr{E})} +|\|y(Q_h)-Y_h|\|_{L^2(I;\mathscr{E})},\nonumber\\
&\leq c_0c_{tr}c_a^{-1}\|q-Q_h\|_{L^2(I;L^2(\Gamma_{_N}))}+|\|y(Q_h)-Y_h|\|_{L^2(I;\mathscr{E})}, \label{stbound}
\end{align}
Employing the triangle inequality and the bound in Eq. \eqref{adjoint-statebound} with similar argument leads to the following adjoint-state error bound
\begin{align}
|\|z-Z_h|\|_{L^2(I; \mathscr{E})}&\leq |\|z-z(Q_h)|\|_{L^2(I; \mathscr{E})}+|\|z(Q_h)-Z_h|\|_{L^2(I; \mathscr{E})},\nonumber\\
&\leq c_0c_{tr}c_a^{-1}\|q-Q_h\|_{L^2(I;L^2(\Gamma_{_N}))}+|\|z(Q_h)-Z_h|\|_{L^2(I;\mathscr{E})}. \label{adjoint-statebound2}
\end{align}
Next we bound the co-control error. \smallskip

\noindent
{(ii) \tt Bound for the co-control error.}  From the inequalities \eqref{optmallitycon1} and \eqref{redisopt11} with \eqref{adjoint-statebound2}, we find that
\begin{align}
\|\mu -\hat{\mu}_h\|_{L^2(I;L^2{(\Gamma_{N})})}&\leq \alpha \|q-Q_h\|_{L^2(I;L^2{(\Gamma_{N})})}+ \alpha \|q_d-\hat{q}_{d,h}\|_{L^2(I;L^2{(\Gamma_{N})})}\nonumber\\&~~~+\|z-\tilde{Z}_h\|_{L^2(I;L^2(\Gamma_{_N}))}\nonumber \\
&\leq (\alpha+c_0c_{tr}c_a^{-1}) \|q-Q_h\|_{L^2(I;L^2{(\Gamma_{N})})}+ \alpha \|q_d-\hat{q}_{d,h}\|_{L^2(I;L^2{(\Gamma_{N})})}\nonumber\\&~~~+|\|z(Q_h)-Z_h|\|_{L^2(I; \mathscr{E})}+c_0c_{tr}\|\nabla(Z_h-\tilde{Z}_h)\|_{L^2(I; L^2(\Omega))}.\label{cocontbdd}
\end{align}
Next, we derive a bound for the control error.\smallskip

\noindent
{(iii) \tt Bound for the control error.} We observe that
\begin{align}
\alpha\int_0^T\|q-Q_h\|^2_{L^2(\Gamma_{_N})}dt&=\int^T_0(\alpha (q-Q_h), q-Q_h)_{L^2(\Gamma_{_N})}dt\nonumber\\&=\int^T_0(\alpha q, q-Q_h)_{L^2(\Gamma_{_N})}dt-\int^T_0(\alpha Q_h, q-Q_h)_{L^2(\Gamma_{_N})}dt.\nonumber
\end{align}
With the choice $\varphi=q-Q_h$, the first-order optimality condition \eqref{weakformcontrol} becomes
\begin{align}
&\alpha\int_0^T\|q-Q_h\|^2_{L^2(\Gamma_{_N})}dt~\leq ~\int^T_0(\alpha q_d-z, q-Q_h)_{L^2(\Gamma_{_N})}dt-\int^T_0(\alpha Q_h, q-Q_h)_{L^2(\Gamma_{_N})}dt\nonumber\\
&~~~\leq\int_0^T(\alpha(q_d-\hat{q}_{d,h}),q-Q_h)_{L^2(\Gamma_{_N})}dt+\int_0^T(\alpha(Q_h-\hat{q}_{d,h})+\tilde{Z}_h,\varphi_h-q)_{L^2(\Gamma_{_N})}dt\nonumber\\
&~~~-\int_0^T(\alpha(Q_h-\hat{q}_{d,h})+\tilde{Z}_h,\varphi_h-Q_h)_{L^2(\Gamma_{_N})}dt-\int_0^T(z-\tilde{Z}_h,q-Q_h)_{L^2(\Gamma_{_N})}dt.\label{4.14}
\end{align}
By virtue of inequality \eqref{redisfirstoptcond12}, we have $-\int_0^T(\alpha(Q_h-\hat{q}_{d,h})+\tilde{Z}_h,\varphi_h-Q_h)_{L^2(\Gamma_{_N})}dt\leq 0$. Consequently, the Eq. \eqref{4.14} becomes
\begin{align}
&\alpha\int_0^T\|q-Q_h\|^2_{L^2(\Gamma_{_N})}dt \leq \int_0^T(\alpha(q_d-\hat{q}_{d,h}),q-Q_h)_{L^2(\Gamma_{_N})}dt\nonumber\\ &~+\int_0^T(\alpha(Q_h-\hat{q}_{d,h})+\tilde{Z}_h,\varphi_h-q)_{L^2(\Gamma_{_N})}dt+\int_0^T(\tilde{Z}_h-z,q-Q_h)_{L^2(\Gamma_{_N})}dt\nonumber\\
&:=~\mathcal{I}_1+\mathcal{I}_2+\mathcal{I}_3.\label{4.15}
\end{align}
Our next step is to estimate the terms $\mathcal{I}_1,\, \mathcal{I}_2$ and $\mathcal{I}_3$. Invoking the well-known inequality $ab\leq \frac{1}{2\epsilon}a^2+\frac{\epsilon}{2}b^2$ for all $a,\,b\in \mathbb{R}^{+}$, we can estimate $\mathcal{I}_1$ as follows
\begin{align}
\mathcal{I}_1
&\leq \frac{1}{2\epsilon}\|\alpha(q_d-\hat{q}_{d,h})\|^2_{L^2(I;L^2(\Gamma_{_N}))}+\frac{\epsilon}{2}\|q-Q_{h}\|^2_{L^2(I;L^2(\Gamma_{_N}))}.\label{4.16} 
\end{align}
The term $\mathcal{I}_2$ is bounded by utilizing the inequality \eqref{assmption11}
\begin{align}
&\mathcal{I}_2
\leq \frac{1}{2 \epsilon}C^2_{4,1}\int_0^T \sum_{E_N\in \mathcal{E}_{N,h}^i}h_{E_N}^2 \|\mathbf{n}_E\cdot \nabla(\alpha \left(Q_h-\hat{q}_{d,h}\right)+\tilde{Z}_h)\|^2_{L^2(E)}dt\nonumber\\
&\hspace{0.9cm}+\frac{\epsilon }{2}\|q-Q_{h}\|^2_{L^2(I;L^2(\Gamma_{_N}))}.\label{4.17}
\end{align}
Lastly, we estimate the final term
\begin{eqnarray}
&&\mathcal{I}_3
=\int_0^T(\tilde{Z}_h-z(Q_h),q-Q_h)_{L^2(\Gamma_{_N})}dt
+\int_0^T(z(Q_h)-z,q-Q_h)_{L^2(\Gamma_{_N})}dt\nonumber\\&&\hspace{1.0cm}:=~\mathpzc{J}_1+\mathpzc{J}_2. \label{4.18}
\end{eqnarray}
We proceed to estimate the terms $\mathpzc{J}_1$ and $\mathpzc{J}_2$. To bound the term $\mathpzc{J}_1$, we employ the $\epsilon$-form of Young's inequality, which yields
\begin{align}
 \mathpzc{J}_1
\leq \frac{c_{tr}^2c_0^2}{2\epsilon}\|z(Q_h)-Z_h\|^2_{L^2(I; \mathscr{E})}+\frac{c_{tr}^2c_0^2}{2\epsilon}\|\nabla(Z_h-\tilde{Z}_h)\|^2_{L^2(I; L^2(\Omega))}+ \frac{\epsilon}{2}\|q-Q_h\|^2_{L^2(I;L^2(\Gamma_{_N}))}.\label{J_24.22}
\end{align}
Setting $\phi=z(Q_h)-z \in \mathpzc{V}$ in \eqref{erroreqstate}, and integrate  from $0$ to $T$, then integration by parts formula with $(y-y(Q_h))(0)=0=(z-z(Q_h))(T)$  yields the bound of the term $\mathpzc{J}_2$
\begin{eqnarray}
\mathpzc{J}_2
&=&\int_0^T\Big(-\frac{\partial(z(Q_h)-z)}{\partial t}, y-y(Q_h) \Big)dt+\int_0^Ta( z(Q_h)-z,y-y(Q_h))dt\nonumber\\&=&\int_0^T\left(y(Q_h)-y, y-y(Q_h)\right)dt~=~-\int_0^T\|y-y(Q_h)\|^2_{L^2(\Omega)}dt~\leq~0.\label{J1bdd}
\end{eqnarray}
Combine \eqref{4.15}--\eqref{J1bdd}, and set $C^2_{4,2}=\max\{1, C_{4,1}^2, c_{tr}^2 c_0^2, c_{tr}^2\}$ with $\epsilon=\alpha/3$, we find that
\begin{align}
&\int_0^T\|q-Q_h\|^2_{L^2(\Gamma_{_N})}dt~\leq~  \frac{9C^2_{4,2}}{2\alpha^2}\Big[\int_0^T \sum_{E_N\in \mathcal{E}_{N,h}^i}h_{E_N}^2 \|\mathbf{n}_E\cdot \nabla(\alpha \left(Q_h-\hat{q}_{d,h}\right)+\tilde{Z}_h)\|^2_{L^2(E_N)}dt\nonumber\\
&~~~~~~~+\|\alpha(q_d-\hat{q}_{d,h})\|^2_{L^2(I;L^2(\Gamma_{_N}))}+\|\nabla(Z_h-\tilde{Z}_h)\|^2_{L^2(I;L^2(\Omega))}+|\|z(Q_h)-Z_h|\|^2_{L^2(I; \mathscr{E})}\Big],\nonumber
\end{align}
and hence
\begin{align}
&\|q-Q_h\|_{L^2(I;L^2(\Gamma_{_N}))}\leq~  \frac{3C_{4,2}}{\sqrt{2} \alpha}\Big[\Big(\sum_{i=1}^{N_T}\int_{t_{i-1}}^{t_i}\sum_{E_N\in \mathcal{E}_{N,h}^i}(\eta^i_{q,E_N})^2dt\Big)^{1/2}\nonumber\\
&~+\|\alpha(q_d-\hat{q}_{d,h})\|_{L^2(I;L^2(\Gamma_{_N}))}+\|\nabla(Z_h-\tilde{Z}_h)\|_{L^2(I;L^2(\Omega))}+|\|z(Q_h)-Z_h|\|_{L^2(I; \mathscr{E})}\Big],\label{4.23bdd}
\end{align}
Combining the bounds \eqref{stbound}--\eqref{cocontbdd} with \eqref{4.23bdd}, and set $C_{4,3}=3c_0c_{tr}c^{-1}_a+\alpha+1,\,C_{4,4}=\frac{3C_{4,2}}{\sqrt{2} \alpha}C_{4,3},\;C_{4,5}=C_{4,4}+1,\;C_{4,6}=c_0c_{tr}+C_{4,4},\;C_{4,7}=C_{4,5}+1,\;C_{4,8}=1$,  this completes the rest of the proof.
\end{proof}
\begin{remark}
A crucial assumption underlying our reliable error estimation is that $\mathpzc{Q}^i_{Nad,h} \subset \mathpzc{Q}_{Nad}$ and \eqref{assmption11} hold, (cf., \cite[Lemma 3.1]{liuyan2001}).  If this inclusion property is not satisfied, some additional work is needed to establish the validity of our error estimate (cf., \cite[Remark 3.2]{liuyan2001}).
\end{remark}  
We now proceed to establish the reliable error bounds for the discrepancies between the continuous and discrete solutions of the state and adjoint-state variables, expressed in terms of estimators and data oscillations.
\begin{lemma}[\bf{Intermediate error bounds}] \label{interrbdlm4.3}
Let $(Y_h,Z_h,Q_h)$ and $(y(Q_h),z(Q_h))$ be the solutions of \eqref{redisoptstate12}--\eqref{redisfirstoptcond12} and \eqref{intstate}-- \eqref{intfinal} with $q^*=Q_h$, respectively. Then there exists two positive constants $c_{1,10}$ and $c_{1,12}$ such that the following estimates hold:
\begin{align}
&|\|z(Q_h)-Z_h|\|^2_{L^2(I; \mathscr{E})}\leq c_{1,10}\big[\varTheta_z^2+\varTheta_{z,T}^2+\eta_{z}^2 +\|y(Q_h)-\hat{Y}_h\|^2_{L^2(I;L^2(\Omega))}\big], \label{intadbound4.37}\\
&|\|y(Q_h)-Y_h|\|^2_{L^2(I; \mathscr{E})}\leq c_{1,12}\big[\varTheta_y^2 +\varTheta_{y,T}^2 +\eta_{y}^2+  \|q-Q_h\|^2_{L^2(I;L^2(\Gamma_{N}))}\big].\label{intadstatebound4.49}
\end{align}
\end{lemma}
\begin{proof} We first proceed with the intermediate adjoint-state error bounds. \smallskip

\noindent
\textbf{(i)}{ \tt Bound for intermediate adjoint-state error.} 
Letting $e_z=z(Q_h)-Z_h$ and choosing $\phi \in \mathpzc{V}$, we then compute
\begin{align}
-(e_{z,t}, \phi)+a_h(e_z, \phi)&=-(z_t(Q_h), \phi)+a_h(z(Q_h), \phi)-[-(Z_{h,t},\phi)+a_h(\tilde{Z_h}, \phi)+a_h(Z_h-\tilde{Z_h}, \phi)]\nonumber\\
&=~(y(Q_h)-\hat{Y}_h,\phi)+(\hat{y}_{d,h}-y_d,\phi)+(r_{_N}-\hat{\mathcal{R}}_{N,h}, \phi)+ a_h(\tilde{Z_h}-Z_h, \phi)\nonumber\\&~~~~+[(\hat{Y}_h-\hat{y}_{d,h}+Z_{h,t},\phi-\phi_h)-a_h(\tilde{Z_h}, \phi-\phi_h)+(\hat{\mathcal{R}}_{N,h}, \phi-\phi_h)_{L^2(\Gamma_{_N})}].\nonumber
\end{align}
Setting $\phi=e_z$ and $\psi=\phi-\phi_h$, we obtain
\begin{align}
&-(e_{z,t}, e_z)+a_h(e_z, e_z)~=~(y(Q_h)-\hat{Y}_h,e_z)+(\hat{y}_{d,h}-y_d,e_z)+(r_{_N}-\hat{\mathcal{R}}_{N,h}, e_z)_{L^2(\Gamma_{_N})}\nonumber\\&~~~~~~~~~~~~~~+ a_h(\tilde{Z_h}-Z_h, e_z)+[(\hat{Y}_h-\hat{y}_{d,h}+Z_{h,t},\psi)-a_h(\tilde{Z_h}, \psi)+(\hat{\mathcal{R}}_{N,h}, \psi)_{L^2(\Gamma_{_N})}],\nonumber
\end{align}
and hence 
\begin{align}
&-\frac{1}{2}\frac{d}{dt}\|e_z\|^2+c_a\||e_z\||^2_{\mathscr{E}}\leq(y(Q_h)-\hat{Y}_h,e_z)+(\hat{y}_{d,h}-y_d,e_z)+(r_{_N}-\hat{\mathcal{R}}_{N,h}, e_z)_{L^2(\Gamma_{_N})}\nonumber\\&~~~~~~~+ a_h(\tilde{Z_h}-Z_h, e_z)+[(\hat{Y}_h-\hat{y}_{d,h}+Z_{h,t},\psi)+(\hat{\mathcal{R}}_{N,h}, \psi)_{L^2(\Gamma_{_N})}-\sum_{K\in \mathscr{T}_h^i}\{(\nabla \tilde{Z}_h, \nabla \psi)_K\nonumber\\&~~~~~~~+(a_0\tilde{Z}_h,\psi)_K\}+\sum_{E\in \mathcal{E}_{0,h}^i \cup \mathcal{E}_{D,h}^i }\{(\smean{\nabla \psi},\sjump{\tilde{Z}_h})_E+(\smean{\nabla \tilde{Z}_h}, \sjump{\psi})_E-\frac{\sigma_0}{h_E}(\sjump{\psi},\sjump{\tilde{Z}_h})_E\}].\label{4.26eq}
\end{align}
We now simplify the term $\sum_{K\in \mathscr{T}_h^i}(\nabla \tilde{Z}_h, \nabla \psi)_K$ using integration by parts formula, which gives
\begin{eqnarray}
\sum_{K\in \mathscr{T}_h^i}(\nabla \tilde{Z}_h, \nabla \psi)_K&=&\sum_{K\in \mathscr{T}_h^i}(-\Delta \tilde{Z}_h, \psi)_K+\sum_{E\in \mathcal{E}_{0,h}^i}\{(\smean{\nabla \tilde{Z}_h},\sjump{\psi})_E+(\sjump{ \nabla \tilde{Z}_h},\smean{\psi})_E\}\nonumber\\&&+\sum_{E\in \mathcal{E}_{D,h}^i}(\mathbf{n}_{E} \cdot \nabla \tilde{Z}_h,\psi)_E+\sum_{E\in \mathcal{E}_{N,h}^i}(\mathbf{n}_{E} \cdot \nabla \tilde{Z}_h,\psi)_E.\label{4.27eq}
\end{eqnarray}
Substitution of \eqref{4.27eq} into \eqref{4.26eq} reveals that
\begin{align}
&-\frac{1}{2}\frac{d}{dt}\|e_z\|^2+c_a|\|e_z|\|^2_{\mathscr{E}}\leq~(y(Q_h)-\hat{Y}_h,e_z)+(\hat{y}_{d,h}-y_d,e_z)+(r_{_N}-\hat{\mathcal{R}}_{N,h}, e_z)_{L^2(\Gamma_{_N})}\nonumber\\&\hspace{1cm}+ a_h(\tilde{Z_h}-Z_h, e_z)+\sum_{K\in \mathscr{T}_h^i}[(\hat{Y}_h-\hat{y}_{d,h}+Z_{h,t}+\Delta \tilde{Z}_h-a_{0,h}\tilde{Z}_h,\psi)_K+((a_{0,h}-a_0)\tilde{Z}_h,\psi)_K]
\nonumber\\&\hspace{1cm}+\sum_{E\in \mathcal{E}_{N,h}^i }(\hat{\mathcal{R}}_{N,h}-\mathbf{n}_{E} \cdot \nabla \tilde{Z}_h, \psi)_{E}+\sum_{E\in \mathcal{E}_{D,h}^i }(\mathbf{n}_{E} \cdot \nabla \psi-\frac{\sigma_0}{h_E}\psi, \tilde{Z}_h)_E\nonumber\\&\hspace{1cm}+\sum_{E\in \mathcal{E}_{0,h}^i}[(\smean{\nabla \psi},\sjump{\tilde{Z}_h})_E-(\sjump{\nabla \tilde{Z}_h}, \smean{\psi})_E-\frac{\sigma_0}{h_E}(\sjump{\psi},\sjump{\tilde{Z}_h})_E].\label{4.28eq}
\end{align}
We know that for any $\upsilon \in \mathpzc{V}_h \cap H^1(\Omega)$ that vanishes on $\Gamma_D$, then the orthogonality leads to
\begin{align}
&0=a_h(e_z, \tilde{Z}_h-\upsilon)
~=~\sum_{K\in \mathscr{T}_h^i}[(\nabla e_z,\nabla (\tilde{Z}_h-\upsilon))_K+(a_0e_z,\tilde{Z}_h-\upsilon)_K]\nonumber\\&\hspace{0.5cm}-\sum_{E\in \mathcal{E}_{0,h}^i}[(\smean{\nabla (\tilde{Z}_h-\upsilon)}, \sjump{ e_z})_E+(\smean{\nabla e_z}, \sjump{\tilde{Z}_h})_E-\frac{\sigma_0}{h_E}(\sjump{ e_z},\sjump{\tilde{Z}_h})_E]\nonumber\\&\hspace{0.5cm}-\sum_{E\in \mathcal{E}_{D,h}^i}[(\mathbf{n}_E\cdot\nabla(\tilde{Z}_h-\upsilon), e_z)_E+(\mathbf{n}_E\cdot\nabla e_z, \tilde{Z}_h)_E+\frac{\sigma_0}{h_E}(Z_h, \tilde{Z}_h)_E].\label{4.29eq}
\end{align}
Incorporating \eqref{4.29eq} in \eqref{4.28eq}, and substituting $\psi=e_z-\phi_h$, where $\phi_h$ is constant on $\mathscr{T}_h^i$, we obtain
\begin{align}
&-\frac{1}{2}\frac{d}{dz}\|e_z\|^2+c_a\|e_z\|^2_{\mathscr{E}}~\leq~(y(Q_h)-\hat{Y}_h,e_z)+(\hat{y}_{d,h}-y_d,e_z)+(r_{_N}-\hat{\mathcal{R}}_{N,h}, e_z)_{L^2(\Gamma_{_N})}\nonumber\\&\hspace{0.7cm}+ a_h(\tilde{Z_h}-Z_h, e_z)+\sum_{K\in \mathscr{T}_h^i}[(\hat{Y}_h-\hat{y}_{d,h}+Z_{h,t}+\Delta \tilde{Z}_h-a_{0,h}\tilde{Z}_h,\psi)_K+((a_{0,h}-a_0)\tilde{Z}_h,\psi)_K]
\nonumber\\
&\hspace{0.7cm}+\sum_{E\in \mathcal{E}_{N,h}^i }(\hat{\mathcal{R}}_{N,h}-\mathbf{n}_{E} \cdot \nabla \tilde{Z}_h, \psi)_{E}-\sum_{E\in \mathcal{E}_{D,h}^i }\frac{\sigma_0}{h_E}(\psi, \tilde{Z}_h)_E-\sum_{E\in \mathcal{E}_{0,h}^i}[\frac{\sigma_0}{h_E}(\sjump{\psi},\sjump{\tilde{Z}_h})_E\nonumber\\&\hspace{0.7cm}+(\sjump{\nabla \tilde{Z}_h}, \smean{\psi})_E]+\sum_{K\in \mathscr{T}_h^i}[(\nabla e_z,\nabla (\tilde{Z}_h-\upsilon))_K+(a_0e_z,\tilde{Z}_h-\upsilon)_K]\nonumber\\
&\hspace{0.7cm}-\sum_{E\in \mathcal{E}_{0,h}^i}[(\smean{\nabla (\tilde{Z}_h-\upsilon)}, \sjump{ e_z})_E+\frac{\sigma_0}{h_E}(\sjump{Z_h},\sjump{\tilde{Z}_h})_E]-\sum_{E\in \mathcal{E}_{D,h}^i}[(\mathbf{n}_E\cdot\nabla(\tilde{Z}_h-\upsilon), e_z)_E\nonumber\\
&\hspace{0.7cm}+\frac{\sigma_0}{h_E}(Z_h, \tilde{Z}_h)_E].\label{4.30eq}
\end{align}
After integrating \eqref{4.30eq} from $t_{i-1}$ to $t_i$ with respect to time, and summing over all time steps from $1$ to $N_T$, with $e_z(t_N)=0$, we arrive at
\begin{align}
&\frac{1}{2}\|e_z(0)\|^2+c_a\sum_{i=1}^{N_T} \int_{t_{i-1}}^{t_i}\|e_z\|^2_{\mathscr{E}}dt~\leq\sum_{i=1}^{N_T}\int_{t_{i-1}}^{t_i}(y(Q_h)-\hat{Y}_h,e_z)dt+\sum_{i=1}^{N_T}\int_{t_{i-1}}^{t_i}(\hat{y}_{d,h}-y_d,e_z)dt\nonumber\\&~~~~+\sum_{i=1}^{N_T}\int_{t_{i-1}}^{t_i}(r_{_N}-\hat{\mathcal{R}}_{N,h}, e_z)_{L^2(\Gamma_{_N})}dt+ \sum_{i=1}^{N_T}\int_{t_{i-1}}^{t_i}a_h(\tilde{Z_h}-Z_h, e_z)dt\nonumber \displaybreak[0]\\
&~~~~+\sum_{i=1}^{N_T}\int_{t_{i-1}}^{t_i}\sum_{K\in \mathscr{T}_h^i}[(\hat{Y}_h-\hat{y}_{d,h}+Z_{h,t}+\Delta \tilde{Z}_h-a_{0,h}\tilde{Z}_h,\psi)_K+((a_{0,h}-a_0)\tilde{Z}_h,\psi)_K]dt
\nonumber   \\
&~~~~+\sum_{i=1}^{N_T}\int_{t_{i-1}}^{t_i}\sum_{E\in \mathcal{E}_{N,h}^i }(\hat{\mathcal{R}}_{N,h}-\mathbf{n}_{E} \cdot \nabla \tilde{Z}_h, \psi)_{E}dt-\sum_{i=1}^{N_T}\int_{t_{i-1}}^{t_i}\sum_{E\in \mathcal{E}_{D,h}^i }\frac{\sigma_0}{h_E}(\psi, \tilde{Z}_h)_Edt\nonumber\\&~~~~-\sum_{i=1}^{N_T}\int_{t_{i-1}}^{t_i}\sum_{E\in \mathcal{E}_{0,h}^i}[\frac{\sigma_0}{h_E}(\sjump{\psi},\sjump{\tilde{Z}_h})_E+(\sjump{\nabla \tilde{Z}_h}, \smean{\psi})_E]dt+\sum_{i=1}^{N_T}\int_{t_{i-1}}^{t_i}\sum_{K\in \mathscr{T}_h^i}[(\nabla e_z,\nabla (\tilde{Z}_h-\upsilon))_K\nonumber\\&~~~~+(a_0e_z,\tilde{Z}_h-\upsilon)_K]dt-\sum_{i=1}^{N_T}\int_{t_{i-1}}^{t_i}\sum_{E\in \mathcal{E}_{0,h}^i}(\smean{\nabla (\tilde{Z}_h-\upsilon)}, \sjump{ e_z})_Edt
\nonumber\\&~~~~-\sum_{i=1}^{N_T}\int_{t_{i-1}}^{t_i}\sum_{E\in \mathcal{E}_{D,h}^i}(\mathbf{n}_E\cdot\nabla(\tilde{Z}_h-\upsilon), e_z)_Edt-\sum_{i=1}^{N_T}\int_{t_{i-1}}^{t_i}\sum_{E\in \mathcal{E}_{0,h}^i}\frac{\sigma_0}{h_E}(\sjump{Z_h},\sjump{\tilde{Z}_h})_Edt
\nonumber\\&~~~~-\sum_{i=1}^{N_T}\int_{t_{i-1}}^{t_i}\sum_{E\in \mathcal{E}_{D,h}^i}\frac{\sigma_0}{h_E}(Z_h, \tilde{Z}_h)_Edt~
:=~\sum_{j=1}^{8}\mathpzc{B}_j.\label{4.31eq}
\end{align}
We now proceed to estimate each term $\mathpzc{B}_j,\,j=1,\,\ldots,8$, separately. For any $\epsilon>0$, then application of $\epsilon$-form of Young's inequality leads to the bound of  $\mathpzc{B}_1$ as
\begin{eqnarray}
\mathpzc{B}_1~\leq~
\frac{1}{2\epsilon_1}\|y(Q_h)-\hat{Y}_h\|^2_{L^2(I;L^2(\Omega))}+\frac{\epsilon_1}{2} |\|e_z|\|^2_{L^2(I;\mathscr{E})}. \nonumber
\end{eqnarray}
Analogously, we can bound the term $\mathpzc{B}_2$ as
\begin{eqnarray}
\mathpzc{B}_2~\leq~
\frac{1}{2\epsilon_2}\sum_{i=1}^{N_T}\int_{t_{i-1}}^{t_i}\|\hat{y}_{d,h}-y_d\|^2_{L^2(\Omega)}dt+\frac{\epsilon_2}{2}|\|e_z|\|^2_{L^2(I; \mathscr{E})}.\nonumber
\end{eqnarray}
Next we consider the boundary term $\mathpzc{B}_3$ and use of trace inequality leads to
\begin{eqnarray}
\mathpzc{B}_3~ 
\leq~ \frac{c^2_{tr}c_0^2}{2\epsilon_3}\sum_{i=1}^{N_T}\int_{t_{i-1}}^{t_i}\|r_{_N}-\hat{\mathcal{R}}_{N,h}\|^2_{L^2(\Gamma_{_N})}dt+\frac{\epsilon_3}{2} |\|e_z|\|^2_{L^2(I; \mathscr{E})}. \nonumber
\end{eqnarray}
Applying the continuity of the bi-linear form on the term $\mathpzc{B}_4$, we find that
\begin{eqnarray}
\mathpzc{B}_4 
~\leq~ \frac{2}{\epsilon_4}  \sum_{i=1}^{N_T}\int_{t_{i-1}}^{t_i}\|\tilde{Z_h}-Z_h\|_{\mathscr{E}}^2\,dt+\frac{\epsilon_4}{2} |\|e_z|\|^2_{L^2(I; \mathscr{E})}.\nonumber
\end{eqnarray}
Next, we focus on estimating the term $\mathpzc{B}_5$ involving $\psi$. Utilizing the $\epsilon$-form of Young's inequality and rearranging the terms, we derive
\begin{align}
\mathpzc{B}_5
\leq &~ \frac{1}{2\epsilon_5} \sum_{i=1}^{N_T}\int_{t_{i-1}}^{t_i}\Big[\sum_{K\in \mathscr{T}_h^i}h_K^2\big\{\|\hat{Y}_h-\hat{y}_{d,h}+Z_{h,t}+\Delta \tilde{Z}_h-a_{0,h}\tilde{Z}_h\|^2_{L^2(K)}+\|(a_{0,h}-a_0)\tilde{Z}_h\|^2_{L^2(K)}\big\}\Big]dt\nonumber\\
&~~+\frac{1}{2\epsilon_6}\sum_{i=1}^{N_T}\int_{t_{i-1}}^{t_i}\sum_{E\in \mathcal{E}_{N,h}^i }h_E\|\hat{\mathcal{R}}_{N,h}-\mathbf{n}_{E} \cdot \nabla \tilde{Z}_h\|^2_{L^2(E)}dt+\frac{1}{2\epsilon_7}\sum_{i=1}^{N_T}\int_{t_{i-1}}^{t_i}\sum_{E\in \mathcal{E}_{7D,h}^i }\frac{\sigma^2_0}{h_E}\| \tilde{Z}_h\|^2_{L^2(E)}dt\nonumber \displaybreak[0]\\
&~~+\frac{1}{2\epsilon_8}\sum_{i=1}^{N_T}\int_{t_{i-1}}^{t_i}\sum_{E\in \mathcal{E}_{0,h}^i}\frac{\sigma_0^2}{h_E}\|\sjump{\tilde{Z}_h}\|^2_{L^2(E)}dt+\frac{1}{2\epsilon_9}\sum_{i=1}^{N_T}\int_{t_{i-1}}^{t_i}\sum_{E\in \mathcal{E}_{0,h}^i}h_E\|\sjump{\nabla \tilde{Z}_h}\|^2_{L^2(E)}dt\nonumber\\&~~+\frac{\epsilon_5}{2} \sum_{i=1}^{N_T}\int_{t_{i-1}}^{t_i}\sum_{K\in \mathscr{T}_h^i}h_K^{-2}\|\psi\|^2_{L^2(K)} dt+\frac{\epsilon_6}{2} \sum_{i=1}^{N_T}\int_{t_{i-1}}^{t_i} \sum_{E\in \mathcal{E}_{N,h}^i }h_E^{-1}\|\psi\|^2_{L^2(E)}dt\nonumber\\&~~+\frac{\epsilon_7}{2} \sum_{i=1}^{N_T}\int_{t_{i-1}}^{t_i}\sum_{E\in \mathcal{E}_{D,h}^i }h_E^{-1}\|\psi\|^2_{L^2(E)}dt+\frac{\epsilon_8}{2} \sum_{i=1}^{N_T}\int_{t_{i-1}}^{t_i}\sum_{E\in \mathcal{E}_{0,h}^i}h_E^{-1}\|\sjump{\psi}\|^2_{L^2(E)}dt\nonumber\\&~~+ \frac{\epsilon_9}{2} \sum_{i=1}^{N_T}\int_{t_{i-1}}^{t_i}\sum_{E\in \mathcal{E}_{0,h}^i}  h^{-1}_E\|\smean{\psi}\|^2_{L^2(E)}dt.\nonumber
\end{align}
Applying the approximation results \eqref{appinq4.24}--\eqref{appinq4.26} with $e=e_z$, we obtain
\begin{align}
\mathpzc{B}_5&\leq~ \frac{1}{2 \epsilon_5} \sum_{i=1}^{N_T}\int_{t_{i-1}}^{t_i}\Big[\sum_{K\in \mathscr{T}_h^i}h_K^2\big\{\|\hat{Y}_h-\hat{y}_{d,h}+Z_{h,t}+\Delta \tilde{Z}_h-a_{0,h}\tilde{Z}_h\|^2_{L^2(K)}\nonumber\\
&~~+\|(a_{0,h}-a_0)\tilde{Z}_h\|^2_{L^2(K)}\big\}\Big]dt+\frac{1}{2 \epsilon_6}\sum_{i=1}^{N_T}\int_{t_{i-1}}^{t_i}\sum_{E\in \mathcal{E}_{N,h}^i}h_E\|\hat{\mathcal{R}}_{N,h}-\mathbf{n}_{E} \cdot \nabla \tilde{Z}_h\|^2_{L^2(E)}dt\nonumber\\&~~+\frac{1}{2\epsilon_7}\sum_{i=1}^{N_T}\int_{t_{i-1}}^{t_i}\sum_{E\in \mathcal{E}_{D,h}^i }\frac{\sigma_0^2}{h_E}\| \tilde{Z}_h\|^2_{L^2(E)}dt+\frac{1}{2\epsilon_8}\sum_{i=1}^{N_T}\int_{t_{i-1}}^{t_i}\sum_{E\in \mathcal{E}_{0,h}^i}\frac{\sigma_0^2}{h_E}\|\sjump{\tilde{Z}_h}\|^2_{L^2(E)}dt\nonumber\\&~~+\frac{1}{2\epsilon_9}\sum_{i=1}^{N_T}\int_{t_{i-1}}^{t_i}\sum_{E\in \mathcal{E}_{0,h}^i}h_E\|\sjump{\nabla \tilde{Z}_h}\|^2_{L^2(E)}dt+\frac{\epsilon}{2}C \sum_{i=1}^{N_T}\int_{t_{i-1}}^{t_i}\sum_{K\in \mathscr{T}_h^i}\|\nabla e_z\|^2_{L^2(K)} dt,\nonumber
\end{align}
with $\epsilon=\epsilon_5=\epsilon_6=\epsilon_7=\epsilon_8=\epsilon_9$. 
We proceed to examine the terms involving  $\tilde{Z}_h-\upsilon$, denoted as 
\begin{align}
\mathpzc{B}_6&\leq  \frac{\sigma_0^2}{2\epsilon_{10}}\sum_{i=1}^{N_T}\int_{t_{i-1}}^{t_i}\sum_{K\in \mathscr{T}_h^i}\|\nabla (\tilde{Z}_h-\upsilon)\|^2_{L^2(K)}dt+\frac{1}{2\epsilon_{11}} \sum_{i=1}^{N_T}\int_{t_{i-1}}^{t_i}\sum_{K\in \mathscr{T}_h^i}\|\tilde{Z}_h-\upsilon\|^2_{L^2(K)}]dt\nonumber\\&\hspace{0.5cm}+\frac{\sigma_0^2}{2\epsilon_{12}} \sum_{i=1}^{N_T}\int_{t_{i-1}}^{t_i}\sum_{E\in \mathcal{E}_{0,h}^i} h_E\|\smean{\nabla (\tilde{Z}_h-\upsilon)}\|^2_{L^2(E)}dt\nonumber \displaybreak[0]\\
&\hspace{0.5cm}+\frac{\sigma_0^2}{2\epsilon_{13}} \sum_{i=1}^{N_T}\int_{t_{i-1}}^{t_i}\sum_{E\in \mathcal{E}_{D,h}^i}h_E\|\mathbf{n}_E\cdot\nabla(\tilde{Z}_h-\upsilon)\|^2_{L^2(E)}dt+\frac{\epsilon_{10}}{2\sigma_0^2} \sum_{i=1}^{N_T}\int_{t_{i-1}}^{t_i}\sum_{K\in \mathscr{T}_h^i}\|\nabla e_z\|^2 dt\nonumber 
\\&\hspace{0.5cm}+\frac{\epsilon_{11}}{2}\sum_{i=1}^{N_T}\int_{t_{i-1}}^{t_i}\sum_{K\in \mathscr{T}_h^i}a_0\|e_z\|^2_{L^2(K)}dt+\frac{\epsilon_{12}}{2\sigma_0^2} \sum_{i=1}^{N_T}\int_{t_{i-1}}^{t_i}\sum_{E\in \mathcal{E}_{0,h}^i}h_E^{-1}\|\sjump{e_z}\|^2_{L^2(E)}dt\nonumber\\&\hspace{0.5cm}+\frac{\epsilon_{13}}{2\sigma_0^2}\sum_{i=1}^{N_T}\int_{t_{i-1}}^{t_i}\sum_{E\in \mathcal{E}_{D,h}^i}h_E^{-1}\|e_z\|^2_{L^2(E)}dt.\label{4.3355eq}
\end{align}
To bound the terms involving $\nabla(\tilde{Z}_h-\upsilon)$ in \eqref{4.3355eq}, we utilize the trace inequality and the inverse inequality. Furthermore,  invoking \cite[Theorem 2.1]{pascal2007} we have
\begin{align}
\mathpzc{B}_6 &\leq~ \frac{\tilde{C}}{2}\Big(\frac{1}{\epsilon_{10}}+\frac{1}{\epsilon_{12}}+\frac{1}{\epsilon_{13}}\Big) \sum_{i=1}^{N_T}\int_{t_{i-1}}^{t_i}\Big[\sum_{E\in \mathcal{E}_{0,h}^i}\frac{\sigma_0^2}{h_E}\|\sjump{\tilde{Z}_h}\|^2_{L^2(E)}+\sum_{E\in \mathcal{E}_{D,h}^i}\frac{\sigma_0^2}{h_E}\|\tilde{Z}_h\|^2_{L^2(E)}\Big]dt\nonumber\\
&\hspace{0.5cm}+ \frac{1}{2\epsilon_{11}}\tilde{C} \sum_{i=1}^{N_T}\int_{t_{i-1}}^{t_i}\Big[\sum_{E\in \mathcal{E}_{0,h}^i}h_E\|\sjump{\tilde{Z}_h}\|^2_{L^2(E)}+\sum_{E\in \mathcal{E}_{D,h}^i}h_E\|\tilde{Z}_h\|^2_{L^2(E)}\Big]dt\nonumber\\
&\hspace{0.5cm}+\frac{\epsilon_{10}}{2} \sum_{i=1}^{N_T}\int_{t_{i-1}}^{t_i}\sum_{K\in \mathscr{T}_h^i}\|\nabla e_z\|^2 dt+\frac{\epsilon_{11}}{2}\sum_{i=1}^{N_T}\int_{t_{i-1}}^{t_i}\sum_{K\in \mathscr{T}_h^i}a_0\|e_z\|^2_{L^2(K)}dt\nonumber\\
&\hspace{0.5cm}+\frac{\epsilon_{12}}{2\sigma_0^2}C \sum_{i=1}^{N_T}\int_{t_{i-1}}^{t_i}\sum_{K\in \mathscr{T}_h^i}\|\nabla e_z\|^2dt+\frac{\epsilon_{13}}{2\sigma_0^2}C\sum_{i=1}^{N_T}\int_{t_{i-1}}^{t_i}\sum_{K\in \mathscr{T}_h^i}\|\nabla e_z\|^2dt.\nonumber
\end{align}
An application of the Cauchy-Schwarz inequality yields the estimate for $\mathpzc{B}_7$ as
\begin{align}
\mathpzc{B}_7
&\leq~ \frac{1}{2\sigma_0} \sum_{i=1}^{N_T}\int_{t_{i-1}}^{t_i}\Big[\sum_{E\in \mathcal{E}_{0,h}^i}\frac{\sigma_0^2}{h_E}\|\sjump{Z_h}\|^2_{L^2(E)}+\sum_{E\in \mathcal{E}_{0,h}^i}\frac{\sigma_0^2}{h_E}\|\sjump{\tilde{Z}_h}\|^2_{L^2(E)}\Big]dt.\nonumber 
\end{align}
Likewise, we bound the term $\mathpzc{B}_8$ by
\begin{eqnarray}
\mathpzc{B}_8\leq\frac{1}{2\sigma_0} \sum_{i=1}^{N_T}\int_{t_{i-1}}^{t_i}\Big[\sum_{E\in \mathcal{E}_{D,h}^i}\frac{\sigma_0^2}{h_E}\|Z_h\|^2_{L^2(E)}+\sum_{E\in \mathcal{E}_{D,h}^i}\frac{\sigma_0^2}{h_E}\|\tilde{Z}_h\|^2_{L^2(E)}\Big]dt.\label{4.44eq} \nonumber
\end{eqnarray}
By consolidating the estimates for $\mathpzc{B}_1$ through $\mathpzc{B}_8$ with \eqref{4.31eq} and incorporating $\|\nabla e_z\|_{L^2(K)}\leq |\|e_z|\|_\mathscr{E}$ and $a_0\|e_z\|_{L^2(K)}\leq |\|e_z|\|_\mathscr{E}$, we arrive at
\begin{align}
&\frac{1}{2}\|e_z(0)\|^2+c_a\sum_{i=1}^{N_T}\int_{t_{i-1}}^{t_i}\|e_z\|^2_{\mathscr{E}}dt~\leq~ \frac{1}{2\epsilon_1}\|y(Q_h)-\hat{Y}_h\|^2_{L^2(I;L^2(\Omega))}+ \frac{1}{2}\max\{\frac{1}{\epsilon_2}, \frac{c_0^2c_{tr}^2}{\epsilon_3},\frac{1}{\epsilon_5}\} \nonumber\\&~~\times\sum_{i=1}^{N_T}\int_{t_{i-1}}^{t_i}\Big[\sum_{K\in \mathscr{T}_h^i}\big(h_K^2\|(a_{0,h}-a_0)\tilde{Z}_h\|^2_{L^2(K)}+\|\hat{y}_{d,h}-y_d\|^2_{L^2(K)}\big)+\sum_{E\in \mathcal{E}_{N,h}^i}\|r_{_N}-\hat{\mathcal{R}}_{N,h}\|^2_{L^2(E)}\Big]dt\nonumber  \\
&~~+\frac{2}{\epsilon_4}\, \sum_{i=1}^{N_T}\int_{t_{i-1}}^{t_i}\|\tilde{Z_h}-Z_h\|^2_{\mathscr{E}}dt+ \frac{1}{2\epsilon_5} \sum_{i=1}^{N_T}\int_{t_{i-1}}^{t_i}\Big[\sum_{K\in \mathscr{T}_h^i}h_K^2\big\{\|\hat{Y}_h-\hat{y}_{d,h}+Z_{h,t}+\Delta \tilde{Z}_h-a_{0,h}\tilde{Z}_h\|^2_{L^2(K)}\big\}\Big]dt\nonumber  \displaybreak[0]\\
&~~+\frac{1}{2\epsilon_6}\sum_{i=1}^{N_T}\int_{t_{i-1}}^{t_i}\sum_{E\in \mathcal{E}_{N,h}^i }h_E\|\hat{\mathcal{R}}_{N,h}-\mathbf{n}_{E} \cdot \nabla \tilde{Z}_h\|^2_{L^2(E)}dt+\frac{1}{2}\max\{\frac{1}{\epsilon_8},\frac{1}{\epsilon_9}\}\sum_{i=1}^{N_T}\int_{t_{i-1}}^{t_i}\sum_{E\in \mathcal{E}_{0,h}^i}\Big[\frac{\sigma_0^2}{h_E}\|\sjump{\tilde{Z}_h}\|^2_{L^2(E)}\nonumber \\
&~~+h_E\|\sjump{\nabla \tilde{Z}_h}^2_{L^2(E)}\Big]dt+c_{1,8}\sum_{i=1}^{N_T}\int_{t_{i-1}}^{t_i}\sum_{E\in \mathcal{E}_{D,h}^i }\Big[\frac{(\sigma_0+1)}{h_E}\| \tilde{Z}_h\|^2_{L^2(E)}+\frac{(\sigma_0+1)}{h_E}\|Z_h\|^2_{L^2(E)}\Big]dt\nonumber\\&~~
 +c_{1,8}\sum_{i=1}^{N_T}\int_{t_{i-1}}^{t_i}\sum_{E\in \mathcal{E}_{0,h}^i }\Big[\frac{(\sigma_0+1)}{h_E}\| \sjump{\tilde{Z}_h}\|^2_{L^2(E)}+\frac{(\sigma_0+1)}{h_E}\|\sjump{Z_h}\|^2_{L^2(E)}\Big]dt \nonumber \displaybreak[0]\\&~~+ c_{1,9} \sum_{i=1}^{N_T}\int_{t_{i-1}}^{t_i}\Big[\sum_{E\in \mathcal{E}_{0,h}^i}h_E\|\sjump{\tilde{Z}_h}\|^2_{L^2(E)}+\sum_{E\in \mathcal{E}_{D,h}^i}h_E\|\tilde{Z}_h\|^2_{L^2(E)}\Big]dt
 +\frac{(5C+6)\epsilon}{2}\; \sum_{i=1}^{N_T}\int_{t_{i-1}}^{t_i}|\|e_z|\|_{\mathscr{E}}^2dt,\nonumber
\end{align}
where $c_{1,7}=c(\epsilon)\max\{1,c^2_{tr}c_0^2\},\, c_{1,8}=\frac{1}{2}\max\{3c(\epsilon ), 3c(\epsilon )c_{1,6},1\}$ and  $c_{1,9}=c_{1,6}c(\epsilon)$, respectively. Setting $c_{1,10}=\frac{2}{c_a}\max\{4c(\epsilon),c_{1,7},c_{1,8},c_{1,9}\}$ with $\epsilon=\frac{c_a}{(5C+6)}$, and then kick-back arguments yields the inequality \eqref{intadbound4.37}. \smallskip

\noindent
\textbf{(ii) \tt  Bound for the intermediate state error}. 
Analogous to the adjoint-state analysis, we define $e_y=y(Q_h)-Y_h$, $\phi \in \mathpzc{Y}$ and $\phi_h\in  \mathpzc{V}_h$. By invoking \eqref{redisoptstate12} and rearranging the terms, we find that 
\begin{align}
&(e_{y,t}, \phi)+a_h(e_y, \phi)~=~(y_t(Q_h), \phi)+a_h(y(Q_h), \phi)-[(Y_{h,t},\phi)+a_h(\hat{Y_h}, \phi)+a_h(Y_h-\hat{Y_h}, \phi)]
\nonumber\\&\hspace{1.2cm}=(f,\phi)+(q+g_N, \phi)_{L^2(\Gamma_{_N})}+a_h(\hat{Y_h}-Y_h, \phi)-[(Y_{h,t},\phi-\phi_h)+a_h(\hat{Y}_h,\phi-\phi_h)]\nonumber\\&\hspace{1.5cm}-[(Y_{h,t},\phi_h)+a_h(\hat{Y}_h,\phi_h)]\nonumber\\
&\hspace{1.2cm}=~(f-\hat{f}_h,\phi)+(g_N-\hat{\mathcal{G}}_{N,h}, \phi)_{L^2(\Gamma_{_N})}+(q-Q_h, \phi)_{L^2(\Gamma_{_N})}+a_h(\hat{Y_h}-Y_h, \phi)\nonumber\\
&\hspace{1.5cm}+\big[\sum_{K\in\mathscr{T}_h^i}(\hat{f}_h-Y_{h,t},\phi-\phi_h)_K-a_h(\hat{Y}_h,\phi-\phi_h)+\sum_{E\in \mathcal{E}^i_{N,h}}(Q_h+\hat{\mathcal{G}}_{N,h},\phi-\phi_h)_E\nonumber\\
&\hspace{1.5cm}-\sum_{E\in \mathcal{E}^i_{D,h}}(\hat{\mathcal{G}}_{D,h},\frac{\sigma_0}{h_E}\mathbf{n}_E \cdot \sjump{\phi_h}- \smean{\nabla \phi_h})_E\big]. \hspace{2cm}
\end{align}
Set $\phi=e_y$, $\psi=e_y-\phi_h$, then use of bilinear form \eqref{weakform} on the term $(\nabla \hat{Y}_h, \nabla \psi)_K$, and similar to the treatment of the adjoint-state in \eqref{4.27eq}. This yields 
\begin{align}
(e_{y,t}, \phi)&+a_h(e_y, \phi)
=(f-\hat{f}_h,\phi)+(g_N-\hat{\mathcal{G}}_{N,h}, \phi)_{L^2(\Gamma_{_N})}+(q-Q_h, \phi)_{L^2(\Gamma_{_N})}+a_h(\hat{Y_h}-Y_h, \phi)\nonumber\\&\hspace{0.6cm}+\sum_{K\in\mathscr{T}_h^i}\Big[(\hat{f}_h-Y_{h,t}+\Delta\hat{Y}_h-a_{0,h}\hat{Y}_h,\psi)_K+((a_{0,h}-a_0)\hat{Y}_h,\psi)_K\Big]\nonumber\\&\hspace{0.6cm}+\sum_{E\in \mathcal{E}^i_{N,h}}(Q_h+\hat{\mathcal{G}}_{N,h}-\mathbf{n}_E\cdot \nabla \hat{Y}_h,\psi)_E
+\sum_{E\in \mathcal{E}^i_{D,h}}\Big[(\smean{\nabla \psi},\sjump{\hat{Y}_h})_E -\frac{\sigma_0}{h_E}(\sjump{\hat{Y}_h}, \sjump{\psi})_E\Big]\nonumber\\&\hspace{0.6cm}+\sum_{E\in \mathcal{E}^i_{0,h}}\Big[(\smean{\nabla \psi}, \sjump{\hat{Y}_h})_E-\frac{\sigma_0}{h_E}(\sjump{\hat{Y}_h}, \sjump{\psi})_E-(\sjump{\nabla \hat{Y}_h},\smean{\psi})_E\Big]
\nonumber\\&\hspace{0.6cm}-\sum_{E\in \mathcal{E}^i_{D,h}}(\hat{\mathcal{G}}_{D,h},\frac{\sigma_0}{h_E}\mathbf{n}_E \cdot \sjump{\phi_h}- \smean{\nabla \phi_h})_E. \label{4.49eqb}
\end{align}
Following a similar approach as in the adjoint-state analysis \eqref{4.29eq} and rearrange the terms, we obtain
\begin{align}
&(e_{y,t}, \phi)+a_h(e_y, \phi)~=~ (f-\hat{f}_h,\phi)+(g_N-\hat{\mathcal{G}}_{N,h}, \phi)_{L^2(\Gamma_{_N})}+(q-Q_h, \phi)_{L^2(\Gamma_{_N})}+a_h(\hat{Y_h}-Y_h, \phi)\nonumber\\&\hspace{0.4cm}+\sum_{K\in\mathscr{T}_h^i}\Big[(\hat{f}_h-Y_{h,t}+\Delta\hat{Y}_h-a_{0,h}\hat{Y}_h,\psi)_K+((a_{0,h}-a_0)\hat{Y}_h,\psi)_K\Big]+\sum_{E\in \mathcal{E}^i_{N,h}}(Q_h+\hat{\mathcal{G}}_{N,h}-\mathbf{n}_E\cdot \nabla \hat{Y}_h,\psi)_E
\nonumber\\&\hspace{0.4cm}-\sum_{E\in \mathcal{E}^i_{0,h}}\Big[\frac{\sigma_0}{h_E}(\sjump{\hat{Y}_h}, \sjump{\psi})_E+(\sjump{\nabla \hat{Y}_h},\smean{\psi})_E\Big]
+\sum_{K\in\mathscr{T}_h^i}\Big[(\nabla e_y,\nabla (\hat{Y}_h-\upsilon))_K+(a_0e_y,\hat{Y}_h-\upsilon)_K\Big]\nonumber\\
&\hspace{0.4cm}-\sum_{E\in \mathcal{E}^i_{0,h}}(\smean{\nabla (\hat{Y}_h-\upsilon)},[[e_y]])_E-\sum_{E\in \mathcal{E}^i_{D,h}}(\mathbf{n}_E \cdot\nabla (\hat{Y}_h-\upsilon),e_y)_E+\sum_{E\in \mathcal{E}^i_{0,h}}\frac{\sigma_0}{h_E}(\sjump{e_y},\sjump{\hat{Y}_h})_E\nonumber \displaybreak[0]\\
&\hspace{0.4cm}-\sum_{E\in \mathcal{E}^i_{D,h}}\frac{\sigma_0}{h_E}(g_D-\mathcal{G}_{D,h},\hat{Y}_h)_E-\sum_{E\in \mathcal{E}^i_{D,h}}(\hat{\mathcal{G}}_{D,h}-\hat{Y}_h,\frac{\sigma_0}{h_E}\mathbf{n}_E \cdot \sjump{\phi_h}- \smean{\nabla \phi_h})_E\nonumber\\
&\hspace{0.4cm}-\sum_{E\in \mathcal{E}^i_{D,h}}\frac{\sigma_0}{h_E}(\sjump{\hat{Y}_h},\sjump{\phi})_E+\sum_{E\in \mathcal{E}^i_{D,h}}(\smean{\nabla \phi},\sjump{\hat{Y}_h})_E.\nonumber 
\end{align}
Using similar approach as \eqref{4.29eq}, and choosing $\phi=e_y$, we leverage the coercivity of the bilinear form. Integrating the resulting expression from $t_{i-1}$ to $t_i$ and summing over all time steps from $1$ to $N_T$, we arrive at
\begin{align}
&\frac{1}{2}\|e_y(T)\|^2_{L^2(\Omega)}+c_a\sum_{i=1}^{N_T}\int_{t_{i-1}}^{t_i}|\|e_y|\|_{\mathscr{E}}^2\, dt~\leq~\frac{1}{2}\|e_y(0)\|^2_{L^2(\Omega)}+\sum_{i=1}^{N_T}\int_{t_{i-1}}^{t_i}(f-\hat{f}_h,e_y)dt\nonumber\\&\hspace{0.6cm}+\sum_{i=1}^{N_T}\int_{t_{i-1}}^{t_i}(g_N-\hat{\mathcal{G}}_{N,h}, e_y)_{L^2(\Gamma_{_N})}dt+\sum_{i=1}^{N_T}\int_{t_{i-1}}^{t_i}(q-Q_h, e_y)_{L^2(\Gamma_{_N})}dt+\sum_{i=1}^{N_T}\int_{t_{i-1}}^{t_i}a_h(\hat{Y_h}-Y_h, e_y)dt
\nonumber\\&\hspace{0.6cm}+\sum_{i=1}^{N_T}\int_{t_{i-1}}^{t_i}\sum_{K\in\mathscr{T}_h^i}\Big[(\hat{f}_h-Y_{h,t}+\Delta\hat{Y}_h-a_{0,h}\hat{Y}_h,\psi)_K+((a_{0,h}-a_0)\hat{Y}_h,\psi)_K\Big]dt\nonumber \\&\hspace{0.6cm}+\sum_{i=1}^{N_T}\int_{t_{i-1}}^{t_i}\sum_{E\in \mathcal{E}^i_{N,h}}(Q_h+\hat{\mathcal{G}}_{N,h}-\mathbf{n}_E\cdot \nabla \hat{Y}_h,\psi)_Edt-\sum_{i=1}^{N_T}\int_{t_{i-1}}^{t_i}\sum_{E\in \mathcal{E}^i_{0,h}}\Big[\frac{\sigma_0}{h_E}(\sjump{\hat{Y}_h},\sjump{\psi})_E\nonumber \\
&\hspace{0.6cm}+(\sjump{\nabla \hat{Y}_h},\smean{\psi})_E\Big]dt
+\sum_{i=1}^{N_T}\int_{t_{i-1}}^{t_i}\sum_{K\in\mathscr{T}_h^i}\Big[(\nabla e_y,\nabla (\hat{Y}_h-\upsilon))_K+(a_0e_y,\hat{Y}_h-\upsilon)_K\Big]dt\nonumber\\
&\hspace{0.6cm}-\sum_{i=1}^{N_T}\int_{t_{i-1}}^{t_i}\sum_{E\in \mathcal{E}^i_{0,h}}(\smean{\nabla (\hat{Y}_h-\upsilon)},[[e_y]])_Edt-\sum_{i=1}^{N_T}\int_{t_{i-1}}^{t_i}\sum_{E\in \mathcal{E}^i_{D,h}}(\mathbf{n}_E \cdot\nabla (\hat{Y}_h-\upsilon),e_y)_Edt\nonumber \\
&\hspace{0.6cm}+\sum_{i=1}^{N_T}\int_{t_{i-1}}^{t_i}\sum_{E\in \mathcal{E}^i_{0,h}}\Big[\frac{\sigma_0}{h_E}(\sjump{Y_h},\sjump{\hat{Y}_h})_E+(\smean{\nabla \psi}, \sjump{\hat{Y}_h})_E\Big]dt\nonumber\\&\hspace{0.6cm}-\sum_{i=1}^{N_T}\int_{t_{i-1}}^{t_i}\sum_{E\in \mathcal{E}^i_{D,h}}\frac{\sigma_0}{h_E}(g_D-g_{D,h},\hat{Y}_h)_Edt-\sum_{i=1}^{N_T}\int_{t_{i-1}}^{t_i}\sum_{E\in \mathcal{E}^i_{D,h}} \frac{\sigma_0}{h_E}(\sjump{\hat{Y}_h},[[e_y]])_Edt\nonumber\displaybreak[0]\\
&\hspace{0.6cm}+\sum_{i=1}^{N_T}\int_{t_{i-1}}^{t_i}\sum_{E\in \mathcal{E}^i_{D,h}} (\smean{\nabla e_y},\sjump{\hat{Y}_h})_Edt+\sum_{i=1}^{N_T}\int_{t_{i-1}}^{t_i}\sum_{E\in \mathcal{E}^i_{D,h}}(\hat{\mathcal{G}}_{D,h}-\hat{Y}_h,\frac{\sigma_0}{h_E}\mathbf{n}_E \cdot \sjump{\phi_h}- \smean{\nabla \phi_h})_Edt.\label{4.51bbdb}
\end{align}
Utilizing the similar idea as above adjoint error,  and then aplying the inequalities \eqref{appinq4.24}--\eqref{appinq4.26} with $a_0\|e_p\|_{L^2(K)} \leq |\|e_p|\|_{\mathscr{E}}$ and $\|\nabla e_p\|_{L^2(K)} \leq |\|e_p|\|_{\mathscr{E}}$, it follows that  
\begin{align}
&\frac{1}{2}\|e_y(T)\|^2_{L^2(\Omega)}+c_a\sum_{i=1}^{N_T}\int_{t_{i-1}}^{t_i}\|e_y\|_{\mathscr{E}}^2\,dt~\leq~c_{1,11}\Big\{\sum_{K\in \mathscr{T}_h^i}\|y_0-Y_{0,h}\|^2_{L^2(K)}\nonumber\\&\hspace{0.6cm}+  \sum_{i=1}^{N_T}\int_{t_{i-1}}^{t_i}\Big[\sum_{K\in \mathscr{T}_h^i}\big(h_K^2\|(a_{0,h}-a_0)\hat{Y}_h\|^2_{L^2(K)}+\|f-\hat{f}_h\|^2_{L^2(K)}\big)+\sum_{E\in \mathcal{E}_{N,h} }\|g_N-\hat{\mathcal{G}}_{N,h}\|^2_{L^2(E)}\nonumber \displaybreak[0]\\
&\hspace{0.6cm}+\sum_{E\in \mathcal{E}^i_{D,h}}\frac{\sigma_0}{h_E}\|g_D-g_{D,h}\|^2_{L^2(E)}\Big]dt+ \sum_{i=1}^{N_T}\int_{t_{i-1}}^{t_i}\sum_{E\in \mathcal{E}_{N,h} }\|q-Q_h\|^2_{L^2(E)}dt+\sum_{i=1}^{N_T}\int_{t_{i-1}}^{t_i}\|\hat{Y_h}-Y_h\|_{\mathscr{E}}^2\,dt\nonumber\\
&\hspace{0.6cm}+ \sum_{i=1}^{N_T}\int_{t_{i-1}}^{t_i}\Big[ \sum_{K\in\mathscr{T}_h^i}h_K^2\|\hat{f}_h-Y_{h,t}+\Delta\hat{Y}_h-a_{0,h}\hat{Y}_h\|^2_{L^2(K)}+\sum_{E\in \mathcal{E}^i_{N,h}}h_E\|Q_h+\hat{\mathcal{G}}_{N,h}-\mathbf{n}_E\cdot \nabla \hat{Y}_h\|^2_{L^2(E)}\nonumber\\&\hspace{0.6cm}+\sum_{E\in \mathcal{E}^i_{D,h}}\frac{\sigma_0}{h_E}\|\hat{\mathcal{G}}_{D,h}-\hat{Y}_h\|^2_{L^2(E)} \Big]dt+\sum_{i=1}^{N_T}\int_{t_{i-1}}^{t_i}\sum_{E\in \mathcal{E}^i_{0,h}}\Big[\frac{\sigma^2_0}{h_E}\|\sjump{\hat{Y}_h}\|^2_{L^2(E)}+h_E\|\sjump{\nabla \hat{Y}_h}\|^2_{L^2(E)}\Big]dt\nonumber\\&\hspace{0.6cm}+\sum_{i=1}^{N_T}\int_{t_{i-1}}^{t_i}\sum_{E\in \mathcal{E}^i_{0,h}}\Big[\frac{(\sigma_0+1)}{h_E}\|\sjump{\hat{Y}_h}\|^2_{L^2(E)}+\frac{(\sigma_0+1)}{h_E}\|\sjump{Y_h}\|^2_{L^2(E)} \Big]dt\nonumber\\&\hspace{0.6cm}+\sum_{i=1}^{N_T}\int_{t_{i-1}}^{t_i}\sum_{E\in \mathcal{E}^i_{D,h}}\Big[\frac{(\sigma_0+1)}{h_E}\|\hat{Y}_h\|^2_{L^2(E)}+\frac{(\sigma_0+1)}{h_E}\|Y_h\|^2_{L^2(E)} \Big]dt\nonumber\\&\hspace{0.6cm}+\sum_{i=1}^{N_T}\int_{t_{i-1}}^{t_i}\Big[\sum_{E\in \mathcal{E}_{0,h}^i}h_E\|\sjump{\hat{Y}_h}\|^2_{L^2(E)}+\sum_{E\in \mathcal{E}_{D,h}^i}h_E\|\hat{Y}_h\|^2_{L^2(E)}\Big]dt\Big\}+4 \epsilon (c+2)|\|e_y|\|^2_{L^2(I; \mathscr{E})}.
\end{align}
By choosing $c_{1,12}$=$\frac{2}{c_a}c_{1,11}$ with $\epsilon=\frac{c_a}{8(c+2)}$, where $c_{1,11}=\max\{3c_{1,6}c(\epsilon),c(\epsilon),2\epsilon,1/2\}$, and applying kick-back arguments, we ultimately establish the desired inequality \eqref{intadstatebound4.49}, this completes the proof of the theorem.
\end{proof} 
Combining the results from Lemmas \ref{lm4.1int}--\ref{interrbdlm4.3}, we can readily deduce the following result. The proof is straightforward and thus we omitted.
\begin{theorem}\label{4.3theorem}
Let $(y,z,q)$ and $(Y_h, Z_h,Q_h)$ be the solutions of \eqref{weakformstate}--\eqref{weakformcontrol} and \eqref{redisoptstate12}--\eqref{redisfirstoptcond12}, respectively, and the co-control  variable $\mu$ and the discrete co-control variable $\mu_h$ as defined in  \eqref{optmallitycon1} and \eqref{redisopt11}, respectively.  Assume that all the conditions of Lemma \ref{lm4.1int} is fulfilled. Then there exists a positive constant $C_{4,9}$ such that the following reliability type error estimation holds, for $t\in I$:
\begin{align}
&|\|y-Y_h|\|_{L^2(I; \mathscr{E})}+|\|z-Z_h|\|_{L^2(I; \mathscr{E})}+\|q-Q_h\|_{L^2(I;L^2(\Gamma_{_N}))}+\|\mu-\hat{\mu}_h\|_{L^2(I;L^2(\Gamma_{_N}))}\nonumber\\&~~~~~~~~~\leq~ C_{4,9}\big(\varTheta_{yzq}+\Upsilon_{yzq}+\Xi_{T_{yz}}\big).
\end{align}
\end{theorem}
\subsection{Efficient-type A posteriori Error Estimation}
This section examines the efficiency of error estimators, taking into account data oscillations. We show that local error estimators are bounded above by the corresponding local errors, data oscillations, and active control contributions. To this end, we employ bubble functions, as in \cite{pascal2007}, \cite{verfurth1996}, and introduce element bubble functions $\mathpzc{b}_K$ based on the barycentric coordinates $\lambda_j, j = 1, 2, 3,$ of each triangle $K$,
\begin{equation}
\mathpzc{b}_K=27\lambda_1\,\lambda_2\,\lambda_3. \label{eltbubblefunc}
\end{equation}
In contrast, edge bubble functions, denoted by $\mathpzc{b}_E$, are given by
\begin{equation}
\mathpzc{b}_E|_K=4\lambda_1\,\lambda_2, \quad \text{and} \quad \mathpzc{b}_E|_{K^E}=4\lambda^e_1\,\lambda^e_2,
\end{equation}
respectively, where $\lambda_1,\,\lambda_2 \;(\text{or}~ \lambda^e_1,\,\lambda^e_2)$ are the barycentric functions of the triangle $K$ (or $K^e$) on the edge $E\in K\cap K^e$. 
Additionally, the bubble functions possess the following properties
\begin{align}
\|\mathpzc{b}_K\|_{L^\infty(K)}=1,\;  \mathpzc{b}_K\in H_0^1(K) \quad 
\text{and} \quad 
\|\mathpzc{b}_E\|_{L^\infty(E)}=1,\;  \mathpzc{b}_E\in H_0^1(\omega_E).
\end{align}
Here, $\omega_E$ denotes the patch of two elements sharing the edge $E$. We
recall from [49] that there exist constants $c_{5,i},\,i=1,\,2,\ldots,\,7$, which depend on the shape regularity of the triangulations  $\mathscr{T}^i_h,\,i=1,\,2,\ldots,\,N,$ such that the polynomials $\varPhi$ and $\varPsi$, defined on the elements and the patch $\omega_E$, satisfy for any element $K\in \mathscr{T}^i_h$, edge $E\in \mathcal{E}_h^i$,
\begin{subequations}
\begin{align}
&\|\varPhi\|^2_{L^2(K)}~\leq~c_{5,1}\,(\varPhi,\varPhi \mathpzc{b}_K)_K, \quad  K\in \mathscr{T}^i_h,\label{4.66al} \vspace*{2cm} \\ 
& \|\varPhi \mathpzc{b}_K\|_{L^2(K)}~\leq~c_{5,2}\, \|\varPhi\|_{L^2(K)}, \quad  K\in \mathscr{T}^i_h,\label{4.66bl} \vspace*{2cm} \\ 
& \|\nabla (\varPhi \mathpzc{b}_K)\|_{L^2(K)}~\leq~c_{5,3}\,h_K^{-1} \|\varPhi\|_{L^2(K)}, \quad  K\in \mathscr{T}^i_h,\label{4.66cl} \vspace*{2cm}\\  
&\|\varPsi\|^2_{L^2(E)}~\leq~c_{5,4}\,(\varPsi,\varPsi \mathpzc{b}_E)_E, \quad  E\in \mathcal{E}^i_h,\label{4.66dl} \vspace*{2cm}\\ 
& \|\varPsi \mathpzc{b}_E\|_{L^2(E)}~\leq~c_{5,5}\, \|\varPsi\|_{L^2(E)}, \quad  E\in \mathcal{E}^i_h,\label{4.66el} \vspace*{2cm}\\ 
& \|\varPsi \mathpzc{b}_E\|_{L^2(\omega_E)}~\leq~c_{5,6}\,h^{1/2}_E \|\varPsi\|_{L^2(E)}, \quad  \omega_E=K\cup K_E,\quad E=K\cap K_E,\label{4.66fl} \vspace*{2cm}\\ 
& \|\nabla (\varPsi \mathpzc{b}_E)\|_{L^2(\omega_E)}~\leq~c_{5,7}\,h_E^{-1/2} \|\varPsi\|_{L^2(E)}, \quad  \omega_E=K\cup K_E,\quad E=K\cap K_E, \label{4.66gl} 
\end{align}
\end{subequations}
respectively. We now introduce the local energy norm $\||\cdot\||_{\mathcal{S}}$ on a set of elements $\mathcal{S}$ by 
\begin{equation}
|\| \phi |\|_{\mathcal{S}}:=\Big(\sum_{K \in \mathcal{S}}(\|\nabla \phi\|_{L^2(K)}^{2}+a_0\|\phi\|_{L^2(K)}^{2} )
+\sum_{\overset{E \in \mathcal{E}_{0,h} \cup \mathcal{E}_{D,h};}{E\in \partial K,\,K\subset \mathcal{S} }} (h_{E} \|\smean{\nabla \phi}\|_{L^2(E)}^{2}+\frac{\sigma_0}{h_{E}}\| \sjump{\phi} \|_{L^2(E)}^{2})\Big)^{1 / 2},\label{localenergynorm}\nonumber
\end{equation}
and the time-dependent energy-norm is denoted by
$$|\|\phi|\|_{L^2(I;\mathcal{S})}~=~\Big(\sum_{i=1}^{N_T}\int_{t_{i-1}}^{t_i} |\|\phi|\|_{\mathcal{S}}^2\,dt\Big)^{\frac{1}{2}}.$$
The following lemma provides lower bounds on the residual-type error estimators $\eta_{y,K}^i$ and $\eta_{z,K}^i$.
\begin{lemma}\label{lm44eltestyz} 
Let $(y,z,q)$ and $(Y_h, Z_h,Q_h)$ be the solutions of \eqref{weakformstate}--\eqref{weakformcontrol} and \eqref{redisoptstate12}--\eqref{redisfirstoptcond12}, respectively. Further, let the error estimators $\eta_{y,K}^i$, $\eta_{z,K}^i$, $i\in[1:N_T]$, and the data oscillations $\varTheta_{y,K}$, $\varTheta_{z,K}$ be defined as in \eqref{spacest4.4} and \eqref{dataest4.17}, respectively. Then there are  positive constant $c_{5,8}$ such that the following estimates hold:
\begin{subequations}
\begin{align}
&(\eta_{y,K}^i)^2~\leq~c_{5,8}\big[|\|y-Y_h|\|_{L^2(I_i;K)}^2+\varTheta_{y,K}^2+h_K^2|\|\partial_t(y-Y_{h})|\|^2_{L^2(I_i;K)}\big],\label{4.67aelt}\\
&(\eta_{z,K}^i)^2~\leq~c_{5,8}\big[|\|z-Z_h|\|_{L^2(I_i;K)}^2+\varTheta_{z,K}^2+h_K^2|\|\partial_t(z-Z_{h})|\|^2_{L^2(I_i;K)}+h_K^2\|y-\hat{Y}_h\|_{L^2(I_i;L^2(K))}^2\big], \label{4.67belt}
\end{align}
\end{subequations}
where $c_{5,8}=c_{5,1}^2\max\{c_{5,2}^2,c_{5,3}^2\}$.
\end{lemma}
\begin{proof}
Set $\xi= \varPhi \mathpzc{b}_K$ with $\varPhi=\hat{f}_h-Y_{h,t}+\Delta\hat{Y}_h-a_{0,h}\hat{Y}_h$ in \eqref{4.66al}, where $\mathpzc{b}_K$ is the bubble function defined as in \eqref{eltbubblefunc}. Then, using the fact that $(f-y_t+\Delta y-a_0y)|_K=0$ and $\xi|_{\Gamma_K}=0$, where $\Gamma_K$ refer the boundary of $K$, we have, for $t\in I_i,\,i\in [1:N_T]$, 
\begin{align}
&h_K^2\|\varPhi\|^2_{L^2(K)}~\leq~c_{5,1}h_K^2(\hat{f}_h-Y_{h,t}+\Delta\hat{Y}_h-a_{0,h}\hat{Y}_h, \xi)_K\nonumber\\
&~~~~~\leq~c_{5,1}h_K^2\big\{(\hat{f}_h-f, \xi)_K+(y_t-Y_{h,t}, \xi)_K+(\nabla (y-\hat{Y}_h), \nabla \xi)_K+((a_0-a_{0,h})\hat{Y}_h, \xi)_K\nonumber\\&~~~~~~~~+(a_0(y-\hat{Y}_h), \xi)_K\big\}\nonumber\\
&~~~~~\leq~\frac{c^2_{5,1}}{2}\max\{c_{5,2}^2,c_{5,3}^2\}\big\{h_K^2\|\hat{f}_h-f\|^2_{L^2(K)}+h_K^2\|\partial_t(y-Y_{h})\|^2_{L^2(K)}+\|\nabla (y-\hat{Y}_h)\|^2_{L^2(K)}\nonumber\\&~~~~~~~~+h_K^2\|(a_0-a_{0,h})\hat{Y}_h\|^2_{L^2(K)}+h_K^2\|a_0(y-\hat{Y}_h)\|^2_{L^2(K)}\big\}+\frac{1}{2}h_K^2\|\varPhi\|^2_{L^2(K)}.\nonumber
\end{align} 
By invoking the kick-back argument and integrating over $t_{i-1}$ to $t_i$ with respect to time, we derive the inequality stated in \eqref{4.67aelt}.\smallskip

We proceed in a similar way by setting $\xi= \varPsi \mathpzc{b}_K$ with $\varPsi=\hat{Y}_h-\hat{y}_{d,h}+Z_{h,t}+\Delta \tilde{Z}_h-a_{0,h}\tilde{Z}_h$ in \eqref{4.66al}. Using the facts that $(y-y_d+z_t+\Delta z-a_0z)$ and $\xi$ vanish on $K$ and the boundary $\Gamma_K$, respectively. Following analogous steps, we obtain
\begin{align}
&h_K^2\|\varPsi\|^2_{L^2(K)} 
\leq~c_{5,1} h_K^2\big\{(\hat{Y}_h-y, \xi)_K+(y_d-\hat{y}_{d,h}, \xi)_K+(Z_{h,t}-z_t, \xi)_K+(\nabla (\tilde{z}_h-z), \nabla \xi)_K\nonumber\\&~~~~~~~~+((a_0-a_{0,h})\tilde{Z}_h, \xi)_K+(a_0(z-\tilde{Z}_h), \xi)_K\big\} \nonumber\\
&~~~~~~\leq~\frac{c^2_{5,1}}{2}\max\{c_{5,2}^2,c_{5,3}^2\}\big\{h_K^2\|\hat{Y}_h-y\|^2_{L^2(K)}+h_K^2\|y_d-\hat{y}_{d,h}\|^2_{L^2(K)}+h_K^2\|\partial_t(z-Z_{h})\|^2_{L^2(K)}\nonumber\\&~~~~~~~~+\|\nabla (z-\tilde{Z}_h)\|^2_{L^2(K)}+h_K^2\|(a_0-a_{0,h})\tilde{Z}_h\|^2_{L^2(K)}+h_K^2\|a_0(z-\tilde{Z}_h)\|^2_{L^2(K)}\big\}+\frac{1}{2}(\eta_{z,K}^i)^2.\nonumber
\end{align} 
The kick-back argument  and integrating over $t_{i-1}$ to $t_i$ with respect to time, ultimately yields the desired inequality \eqref{4.67belt}, which concludes the proof of the lemma.
\end{proof}
\begin{lemma}
Let $(y,z,q)$ and $(Y_h, Z_h,Q_h)$ be the solutions of \eqref{weakformstate}--\eqref{weakformcontrol} and \eqref{redisoptstate12}--\eqref{redisfirstoptcond12}, respectively. Further, let the error estimators $\eta_{y,K}^i$, $\eta_{z,K}^i$, \,$i\in [1:N_T]$ and the data oscillations $\varTheta^i_{y,K}$ $\varTheta^i_{z,K},\,i\in[1:N_T],$ be defined as in \eqref{spacest4.4} and \eqref{dataest4.17}. In addition, let $\omega_E=K \cup K_E$ be the union of any two elements, $(i.e., K,\;K_E$ such that $E=K\cap K_E)$, and the mesh shape regularity $(i.e., h_E/h_K\leq \beta_0$ with  $\beta_0>1)$.  Then there exist  positive constants $a_{5,1}$ and $a_{5,2}$ such that the following estimates hold, for each $t\in (t_{i-1},t_i],\; i\in[1:N_T]$,
\begin{subequations}
\begin{align}
h_E \| \sjump{ \nabla \hat{Y}_h }\|^2_{L^2(E)}&\leq~a_{5,1}\big[|\|y-Y_h|\|_{L^2(I_i;\omega_E)}^2+\sum_{K \in \omega_E}(\eta_{y,K}^i)^2+\sum_{K \in \omega_E} (\varTheta^i_{y,K})^2\nonumber \\
&\hspace{1.5cm}+|\|\partial_t(y-Y_{h})|\|^2_{L^2(I_i;L^2(\omega_E))}\big], \label{4.68al} \displaybreak[0]\\
h_E \| \sjump{ \nabla \tilde{Z}_h }\|^2_{L^2(E)}&\leq~a_{5,2} \big[|\|z-Z_h|\|_{L^2(I_i;\omega_E)}^2+\sum_{K \in \omega_E}(\eta_{z,K}^i)^2+\sum_{K \in \omega_E}(\varTheta^i_{z,K})^2\nonumber  \\
&\hspace{1.5cm}+\|\partial_t(z-Z_{h})\|^2_{L^2(I_i;L^2(\omega_E))}+\|y-Y_h\|_{L^2(I_i;L^2(\omega_E))}^2\big], \label{4.68b} 
\end{align}
\end{subequations}
where $a_{5,1}=a_{5,2}=2c^2_{5,4} \beta_1^2$ with $\beta_1=\max\{\beta_0c_{5,6},c_{5,7}\}$.
\end{lemma}
\begin{proof} Setting  $\xi= \varPsi \mathpzc{b}_E $ with $ \varPsi =\sjump{\nabla \hat{Y}_h}$ in \eqref{4.66dl}, and recalling that  $\sjump{\nabla y}=0$ on the interior edges, we integrate over the two elements comprising $\omega_E$ to obtain, for $t\in (t_{i-1},t_i],\; i\in[1:N_T]$,
\begin{align}
h_E \| \sjump{ \nabla \hat{Y}_h }\|^2_{L^2(E)}&\leq~	c_{5,4} h_E(\sjump{ \nabla \hat{Y}_h },\xi)_{L^2(E)}=c_{5,4} h_E(\sjump{ \nabla \hat{Y}_h }-\sjump{\nabla y},\xi)_{L^2(E)}\nonumber\\
&=~ c_{5,4} h_E(\hat{f}_h-Y_{h,t}+\Delta \hat{Y}_h-a_{0,h}\hat{Y}_{h},\xi)_{L^2(\omega_E)}+c_{5,4} h_E(f-y_t-a_0y,\xi)_{L^2(\omega_E)}\nonumber\\&\hspace{0.5cm}+~c_{5,4} h_E(\nabla (\hat{Y}_h- y),\nabla \xi)_{L^2(\omega_E)}-c_{5,4} h_E(\hat{f}_h-Y_{h,t}-a_{0,h}\hat{Y}_{h},\xi)_{L^2(\omega_E)}\nonumber\\
&=~ c_{5,4} h_E(\hat{f}_h-Y_{h,t}+\Delta \hat{Y}_h-a_{0,h}\hat{Y}_{h},\xi)_{L^2(\omega_E)}+c_{5,4} h_E(f-\hat{f}_h,\xi)_{L^2(\omega_E)}\nonumber\\&\hspace{0.5cm}+~c_{5,4} h_E(Y_{h,t}-y_t,\xi)_{L^2(\omega_E)}+c_{5,4} h_E((a_{0,h}-a_0)\hat{Y}_{h},\xi)_{L^2(\omega_E)}\nonumber\\&\hspace{0.5cm}+~c_{5,4} h_E(a_0(\hat{Y}_{h}-y),\xi)_{L^2(\omega_E)}+c_{5,4} h_E(\nabla (\hat{Y}_h- y),\nabla \xi)_{L^2(\omega_E)}.\nonumber
\end{align}
By invoking the PDE \eqref{contstate} and utilizing the inequalities \eqref{4.66fl}-\eqref{4.66gl} with the mesh shape regularity condition $(h_E/h_K\leq \beta_0$ with $\beta_0>1$), we derive
\begin{align}
h_E \| \sjump{ \nabla \hat{Y}_h }\|^2_{L^2(E)}&\leq~ c_{5,4}  \max\{\beta_0c_{5,6},c_{5,7}\} h^{1/2}_E\|\varPsi\|_{L^2(E)} \times \big[h_K^2 \big( \|\hat{f}_h-Y_{h,t}+\Delta \hat{Y}_h-a_{0,h}\hat{Y}_{h}\|^2_{L^2(\omega_E)}\nonumber\\&\hspace{0.5cm}+ \|f-\hat{f}_h\|^2_{L^2(\omega_E)}+ \|\partial_t(y-Y_h)\|^2_{L^2(\omega_E)} + \|(a_{0,h}-a_0)\hat{Y}_{h}\|_{L^2(\omega_E)}\nonumber\\&\hspace{0.5cm}+ \|a_0(\hat{Y}_{h}-y)\|^2_{L^2(\omega_E)} \big)+\| \nabla(\hat{Y}_h- y)\|^2_{L^2(\omega_E)}\big]^{1/2}.\nonumber
\end{align}
The Young's inequality followed by kick-back arguments, and then integrating over $t_{i-1}$ to $t_i$ with respect to time, directly yields the desired estimate \eqref{4.68al}.\smallskip

To establish the inequality \eqref{4.68b}, we proceed 
by choosing $\zeta= \sjump{\nabla \tilde{Z}_h} \mathpzc{b}_E$ and utilize \eqref{4.66dl}, inequalities \eqref{4.66fl} and \eqref{4.66gl}. Then, arguing similar to above, for $t\in (t_{i-1},t_i],\; i\in[1:N_T]$, we arrive at
\begin{align}
&h_E \| \sjump{ \nabla \tilde{Z}_h }\|^2_{L^2(E)}~\leq~ c_{5,4}  \max\{\beta_0c_{5,6},c_{5,7}\} h^{1/2}_E\|\sjump{ \nabla \tilde{Z}_h }\|_{L^2(E)}\times \big[h_K^2 \big\{ \|\hat{Y}_h-\hat{y}_{d,h}+Z_{h,t}\nonumber\\&\hspace{1cm}+\Delta \tilde{Z}_h-a_{0,h}\tilde{Z}_{h}\|^2_{L^2(\omega_E)}+ \|y-\hat{Y}_h\|^2_{L^2(\omega_E)}+ \|\hat{y}_{d,h}-y_d\|^2_{L^2(\omega_E)}+ \|\partial_t(z-Z_{h})\|^2_{L^2(\omega_E)} \nonumber\\&\hspace{1cm}+ \|(a_{0,h}-a_0)\tilde{Z}_{h}\|_{L^2(\omega_E)}+ \|a_0(\tilde{Z}_{h}-z)\|^2_{L^2(\omega_E)} \big\}+\| \nabla(\tilde{Z}_h- z)\|^2_{L^2(\omega_E)}\big]^{1/2}.\nonumber
\end{align}
An application of the Young's inequality, and then integrating over $t_{i-1}$ to $t_i$ with respect to time, yields the desired result, this concludes the proof of\eqref{4.68b}.
\end{proof} 
\begin{lemma}
Let $(y,z,q)$ and $(Y_h, Z_h,Q_h)$ be the solutions of \eqref{weakformstate}--\eqref{weakformcontrol} and \eqref{redisoptstate12}-- \eqref{redisfirstoptcond12}, respectively. Further, for any $E_N\in \partial K,\, K\in \mathscr{T}^i_h, i\in [1:N_T]$, let the error estimators $\eta_{y,E_N}^i$, $\eta_{z,E_N}^i$ and the data oscillations $\varTheta^i_{y,K}$, $\varTheta^i_{y,E_N}$, $\varTheta^i_{z,K}, \;\varTheta^i_{z,E_N}$ be defined as in \eqref{4.12yen}-\eqref{4.13en} and\eqref{dataest4.17}, respectively.  Then there exist  positive constants $a_{5,3}$ and $a_{5,4}$ such that the following estimates hold, for each $t\in (t_{i-1},t_i],\; i\in[1:N_T]$,   
\begin{subequations}
\begin{align}
(\eta_{y,E_N}^i)^2&\leq~a_{5,3}\Big(|\|y-Y_h|\|_{L^2(I_i;K)}^2+(\eta^i_{y,K})^2+(\varTheta^i_{y,K})^2+(\varTheta^i_{y,E_N})^2+\|\partial_t(y-Y_{h})\|^2_{L^2(I_i;L^2(K))}\nonumber\\&\hspace{1.5cm}+\|q-Q_h\|^2_{L^2(I_i;L^2(E_N))} \Big),\label{4.69ayen} \displaybreak[0] \\
(\eta_{z,E_N}^i)^2&\leq~ a_{5,4}\Big(|\|z-Z_{h}|\|^2_{L^2(I_i;K)}+(\eta^i_{y,K})^2+(\varTheta^i_{z,K})^2+(\varTheta^i_{z,E_N})^2+\|\partial_t(z-Z_{h})\|^2_{L^2(I_i;L^2(K))}\nonumber\\&\hspace{1.5cm}+\|y-\hat{Y}_h\|^2_{L^2(I_i;L^2(K))}\Big), \label{4.69bzen}
\end{align}
\end{subequations}
where $a_{5,3}=a_{5,4}= 2c^2_{5,4}\beta^2_2$ with $\beta_2=\max\{c_{5,5},\beta_0c_{5,6},c_{5,7} \}$.
\end{lemma}
\begin{proof}
First, we prove the estimate \eqref{4.69ayen}. To begin, we use the inequality \eqref{4.66dl} and  set $\xi=\varPsi  \mathpzc{b}_E$ with $\varPsi=\mathbf{n}_E\cdot \nabla \hat{Y}_h-Q_h-\hat{\mathcal{G}}_{N,h}$, to obtain,  for $t\in (t_{i-1},t_i],\; i\in[1:N_T]$, 
 \begin{align}
 &h_{E_N}\|Q_h+\hat{\mathcal{G}}_{N,h}-\mathbf{n}_E\cdot \nabla \hat{Y}_h\|^2_{L^2(E_N)}\leq~ c_{5,4} h_{E_N} (\mathbf{n}_E\cdot \nabla \hat{Y}_h-Q_h-\hat{\mathcal{G}}_{N,h}, \xi)_{L^2(E_N)} \nonumber\\&\hspace{0.6cm}=~ c_{5,4}h_{E_N}\big[ (\mathbf{n}_E\cdot \nabla (\hat{Y}_h-y), \xi)_{L^2(E_N)}+(q-Q_h,\xi)_{L^2(E_N)}\nonumber\\&\hspace{0.6cm}+(g_N-\hat{\mathcal{G}}_{N,h},\xi)_{L^2(E_N)}\big].
\end{align}
Since $\mathbf{n}_E\cdot \nabla y=q+g_N$ on the edge of the Neumann boundary.
Invoking Green's inequality on the element  $K$ that contains $E_N\subset K$, and incorporating equation \eqref{contstate}, we find
\begin{align}
&h_{E_N}\|Q_h+\hat{\mathcal{G}}_{N,h}-\mathbf{n}_E\cdot \nabla \hat{Y}_h\|^2_{L^2(E_N)}\leq~ c_{5,4}h_{E_N}\big[ (\hat{f}-Y_{h,t}+\Delta \hat{Y}_h-a_{0,h}\hat{Y}_h, \xi)_{L^2(K)}\nonumber\\&\hspace{0.6cm}+(f-\hat{f}_h,\xi)_{L^2(K)}+(Y_{h,t}-y_t,\xi)_{L^2(K)}+((a_{0,h}-a_0)\hat{Y}_{h},\xi)_{L^2(K)}+(a_0(\hat{Y}_{h}-y),\xi)_{L^2(K)}\nonumber\\&\hspace{0.6cm}+c_{5,4} h_E(\nabla (\hat{Y}_h- y),\nabla \xi)_{L^2(K)}+(q-Q_h,\xi)_{L^2(E_N)}+(g_N-\hat{\mathcal{G}}_{N,h},\xi)_{L^2(E_N)}\big].
\end{align}
By leveraging the bounds provided by inequalities \eqref{4.66el}-\eqref{4.66gl}, we deduce
\begin{align}
&h_{E_N}\|Q_h+\hat{\mathcal{G}}_{N,h}-\mathbf{n}_E\cdot \nabla \hat{Y}_h\|^2_{L^2(E_N)}\leq~ c_{5,4}h^{1/2}_{E_N}\max\{c_{5,5},\beta_0c_{5,6},c_{5,7} \}\|\varPsi\|_{L^2(E_N)}\big[|\|y-Y_h|\|^2_{K}\nonumber\\	&\hspace{0.4cm}+ (\eta^i_{y,K})^2+(\varTheta_{y,K}^i)^2+(\varTheta^i_{y,E_N})^2+\|\partial_t(y-Y_{h})\|^2_{L^2(K)}+\|q-Q_h\|^2_{L^2(E_N)}\big].
\end{align}
By invoking Young's inequality and employing a kickback argument, then integrating from $t_{i-1}$ to $t_i$ with respect to $t$, we establish the desired inequality \eqref{4.69ayen}.\smallskip

We now turn our attention to establishing inequality \eqref{4.69bzen}. From the inequality \eqref{4.66dl}, and  by setting $\zeta= (\mathbf{n}_{E_N}\cdot \nabla \tilde{Z}_h-\hat{\mathcal{R}}_{N,h})  \mathpzc{b}_E$, analogously, we have,  for $t\in (t_{i-1},t_i],\; i\in[1:N_T]$, 
\begin{align}
&h_{E_N}\|\hat{\mathcal{R}}_{N,h}-\mathbf{n}_{E_N}\cdot \nabla \tilde{Z}_h\|^2_{L^2(E_N)}\leq~ c_{5,4} h_{E_N} (\mathbf{n}_{E_N} \cdot \nabla \tilde{Z}_h-\hat{\mathcal{R}}_{N,h}, \zeta)_{L^2(E_N)} \nonumber\\&\hspace{0.6cm}=~ c_{5,4}h_{E_N}\big[ (\mathbf{n}_E\cdot \nabla (\tilde{Z}_h-z), \zeta)_{L^2(E_N)}+(r_{N}-\hat{\mathcal{R}}_{N,h},\zeta)_{L^2(E_N)}\big].\nonumber
\end{align}  
Since $\mathbf{n}_{E_N} \cdot \nabla \tilde{Z}_h=r_{_N}$. Use of the Green's inequality on element $K$ with $E_N\subset K$, and \eqref{2.14adjoint-state} yields
\begin{align}
&h_{E_N}\|\hat{\mathcal{R}}_{N,h}-\mathbf{n}_{E_N}\cdot \nabla \tilde{Z}_h\|^2_{L^2(E_N)}\leq~ c_{5,4} h_E\big[(\hat{Y}_h-\hat{y}_{d,h}+Z_{h,t}+\Delta \tilde{Z}_h-a_{0,h}\tilde{Z}_{h},\zeta)_{L^2(K)}\nonumber\\&\hspace{0.7cm}+(y-\hat{Y}_h,\zeta)_{L^2(K)}+(\hat{y}_{d,h}-y_d,\zeta)_{L^2(K)}+(z_t-Z_{h,t},\zeta)_{L^2(K)}+((a_{0,h}-a_0)\tilde{Z}_{h},\zeta)_{L^2(K)}\nonumber\\&\hspace{0.7cm}+(a_0(\tilde{Z}_{h}-z),\zeta)_{L^2(K)}+(\nabla (\tilde{Z}_h- z),\nabla \zeta)_{L^2(K)}+(r_{N}-\hat{\mathcal{R}}_{N,h},\zeta)_{L^2(E_N)}\big].\nonumber
\end{align} 
Applying the inequalities \eqref{4.66el}-\eqref{4.66gl}, we get
\begin{align}
&h_{E_N}\|\hat{\mathcal{R}}_{N,h}-\mathbf{n}_{E_N}\cdot \nabla \tilde{Z}_h\|^2_{L^2(E_N)}\leq~ c_{5,4} \max\{c_{5,5},\beta_0c_{5,6},c_{5,7}\}\, h^{1/2}_E\,\|\hat{\mathcal{R}}_{N,h}-\mathbf{n}_{E_N}\cdot \nabla \tilde{Z}_h\|_{L^2(E_N)} \nonumber\\~&\hspace{0.7cm}\times \big[|\|z-Z_{h}|\|^2_{K}+(\eta^i_{z,K})^2+(\varTheta^i_{z,K})^2+(\varTheta^i_{z,E_N})^2+\|\partial_t(z-Z_{h})\|^2_{L^2(K)}+\|y-\hat{Y}_h\|^2_{L^2(K)}\big]^{\frac{1}{2}}.\nonumber
\end{align} 
An application of the Young's inequality, and then integrating from $t_{i-1}$ to $t_i$ with respect to $t$, this yields the proof of the inequality \eqref{4.69bzen}.
\end{proof}
Our next step is to obtain estimates for the control estimator, with a particular emphasis on controlling the discretization error associated with the control variables.
\begin{lemma}\label{lm47etqelt} 
Let $(y,z,q)$ and $(Y_h, Z_h,Q_h)$ be the solution of \eqref{weakformstate}--\eqref{weakformcontrol} and \eqref{redisoptstate12}--\eqref{redisfirstoptcond12}, respectively. Assume that all the conditions of the lemma \ref{lm4.1int}. Further, for any $E_N\in \partial K,\, K\in \mathscr{T}^i_h, i\in [1:N_T]$, let the error estimators $\eta_{q,E_N}^i$ and $\varTheta^i_{q}$ be defined as in \eqref{4.14444qen}  and \eqref{4.17qen}, respectively.  Then there exist a positive constant $a_{5,5}$ such that the following estimates hold, for each $t\in (t_{i-1},t_i],\; i\in[1:N_T]$,
\begin{subequations}
\begin{align}
&(\eta_{q,E_N}^i)^2~\leq~a_{5,5}\Big(\|q-Q_h\|^2_{L^2(I_i;L^2(\Gamma_N))}+|\|z-Z_h|\|_{L^2(I_i;K)}^2+(\varTheta^i_{q})^2\nonumber\\& \hspace{2cm}+h^2_{E_N} \|(\mathbf{n}_{E_N}\cdot \nabla(\alpha \left(Q_h-\hat{q}_{d,h}\right)+\tilde{Z}_h))\chi_{_{\mathpzc{A}_h}}\|^2_{L^2(I_i;L^2(E_N))}\Big),\label{4.74qen}
\end{align}
\end{subequations}
where $a_{5,5}=\max\{1,c_{inv}^2\}$, and $\mathpzc{A}_h$ denotes the active set  such that $\mathpzc{A}_h=\mathpzc{A}_{a,h}\cup \mathpzc{A}_{b,h}$.
\end{lemma}
\begin{proof}
From \eqref{4.14444qen}, we have, for $t\in I_i,\; i\in[1:N_T]$,
\begin{align}
(\eta^i_{q,E_N})^2 = \int_{t_{i-1}}^{t_i}h^2_{E_N} \|\mathbf{n}_{E_N}\cdot \nabla(\alpha \left(Q_h-\hat{q}_{d,h}\right)+\tilde{Z}_h)\|^2_{L^2(E_N)}\,dt.\nonumber
\end{align}
Using \eqref{2.28contcompcond} with indicator function $\chi$, we have $(\alpha \left(q-q_d\right)+z)\chi_{_{\mathcal{I}_h}}=0$. Combining this with the inverse inequality \eqref{inversinq}, we deduce
\begin{align}
(\eta^i_{q,E_N})^2 &\leq\int_{t_{i-1}}^{t_i}\big[ h^2_{E_N} \|(\mathbf{n}_{E_N}\cdot \nabla(\alpha \left(Q_h-\hat{q}_{d,h}\right)+\tilde{Z}_h))\chi_{_{\mathcal{I}_h}}\|^2_{L^2(E_N)}\nonumber\\& \hspace{0.5cm}+h^2_{E_N} \|(\mathbf{n}_{E_N}\cdot \nabla(\alpha \left(Q_h-\hat{q}_{d,h}\right)
+\tilde{Z}_h))\chi_{_{\mathpzc{A}_h}}\|^2_{L^2(E_N)}\big] \, dt\nonumber\\
&\leq c_{inv}^2\int_{t_{i-1}}^{t_i}\big[  \|(\alpha \left(Q_h-q\right)\|^2_{L^2(E_N)}+\|z-\tilde{Z}_h\|^2_{L^2(E_N)}+\|(\alpha \left(q_d-\hat{q}_{d,h}\right)\|^2_{L^2(E_N)}\nonumber\\
&\hspace{0.5cm}+h^2_{E_N} \|(\mathbf{n}_{E_N}\cdot \nabla(\alpha \left(Q_h-\hat{q}_{d,h}\right)+\tilde{Z}_h))\chi_{_{\mathpzc{A}_h}}\|^2_{L^2(E_N)}]\, dt\nonumber\\
&\leq \max\{1,c_{inv}^2\} \int_{t_{i-1}}^{t_i}\big[\|q-Q_h|^2_{L^2(E_N)}+\|z-\tilde{Z}_h\|^2_{L^2(E_N)}+(\varTheta^i_{q,E_N})^2\nonumber\\
&\hspace{0.5cm}+h^2_{E_N} \|(\mathbf{n}_{E_N}\cdot \nabla(\alpha \left(Q_h-\hat{q}_{d,h}\right)+\tilde{Z}_h))\chi_{_{\mathpzc{A}_h}}\|^2_{L^2(E_N)}\big]\,dt. 
\end{align}
Thus, the desired result follows.
\end{proof}
We are now well placed to establish a reliable estimate for $\Upsilon_{yzq}$, which includes local error components, data oscillations, and active control contributions inherent to local error estimators. The forthcoming result underscores the efficiency of these estimators.
\begin{theorem}\label{4.8theorem}
Let $(y,z,q)$ and $(Y_h, Z_h,Q_h)$ be the solutions of \eqref{weakformstate}--\eqref{weakformcontrol} and \eqref{redisoptstate12}--\eqref{redisfirstoptcond12}, respectively. In addition, let the estimator $\Upsilon_{yzq}$ and the data oscillation $\varTheta_{yzq}$ be defined as in \eqref{upsilon} and \eqref{vartheta}, respectively. In addition, we assume that all the conditions of Lemma \ref{lm4.1int} are satisfied. Then there exists a positive constant $a_{5,6}$ such that the following inequality holds,
\begin{align}
&\Upsilon_{yzq}\leq~ a_{5,6}\big[\|q-Q_h\|_{L^2(I; L^2(\Gamma_{_N}))}+|\|y-Y_h|\|_{L^2(I; \mathscr{E})}+|\|z-Z_h|\|_{L^2(I; \mathscr{E})}\nonumber\\&\hspace{1cm}+\|\partial_t(y-Y_{h})\|_{L^2(I; L^2(\Omega))}+\|\partial_t(z-Z_{h})\|_{L^2(I; L^2(\Omega))}+\varTheta_{yzq}+ \Xi_{T_{yz}}\nonumber\\&\hspace{1cm}+\sum_{i=1}^{N_T}\sum_{E\in \mathcal{E}^i_{N,h}} h_{E_N} \|(\mathbf{n}_{E_N}\cdot \nabla(\alpha \left(Q_h-\hat{q}_{d,h}\right)+\tilde{Z}_h))\chi_{_{\mathpzc{A}_h}}\|_{L^2(I_i;L^2(E_N))} \big]. \label{mainine475}
\end{align}
\end{theorem}
\begin{proof}
Taking into account the regularity of the exact solution $y$, we observe that $\sjump{y}=0$. Further, utilizing the fact that $y=g_D$ on $\Gamma_D$ and manipulating the terms, one can easily deduce that  by using the definition of the energy norm
\begin{align}
\sum_{E \in \mathcal{E}^i_{0,h}} \frac{\sigma^2_0}{ h_{E}}\|\hat{Y}_h\|_{L^2(E)} &+ \sum_{E \in \mathcal{E}^i_{D,h}}\frac{\sigma^2_0}{ h_{E}}\|\hat{\mathcal{G}}_{D,h}-\hat{Y}_h\|_{L^2(E)}\leq~ a_{5,7}\big[|\|y-Y_h|\|_{\mathscr{E}}^2 +\varTheta^2_{y,T}\nonumber\\
&+\sum_{E \in \mathcal{E}^i_{D,h}} \frac{\sigma^2_0}{ h_{E}}  \|\hat{\mathcal{G}}_{D,h}-g_D\|^2_{L^2(E)}\big], \label{yhbd476}
\end{align}
where $a_{5,7}$ denotes the positive constant. Similarly, use of the regularity of $z$ and energy norm leads to the following error estimates 
\begin{align}
\sum_{E \in \mathcal{E}^i_{0,h}} \frac{\sigma^2_0}{ h_{E}}\|\tilde{Z}_h\|_{L^2(E)} &+ \sum_{E \in \mathcal{E}^i_{D,h}}\frac{\sigma^2_0}{ h_{E}}\|\tilde{Z}_h\|_{L^2(E)}\leq~ a_{5,8}\big[|\|z-Z_h|\|^2_{\mathscr{E}} +\varTheta^2_{z,T}\big]. \label{zhbd477}
\end{align}
 Combining all estimates from the Lemmas \ref{lm44eltestyz}-\ref{lm47etqelt} together with the estimates \eqref{yhbd476}-\eqref{zhbd477}. Then, summing over $K\in\mathscr{T}_h^i$, $i\in [1:N_T]$, and integrating from $t_{i-1}$ to $t_i$ with respect to time $t$, summing over $i=1,\,2,\ldots, N_T$, and setting the maximum $a_{5,6}$ on all the constants involved in the estimates, we obtain the desired inequality \eqref{mainine475}.
\end{proof}
\begin{remark}
 In Theorem \ref{4.8theorem}, we analyzed the efficiency part only for the space-time error estimators, while the efficiency part for the temporal terms are ignored. 
\end{remark}
\section{Numerical assessment} \label{section5555}
This section presents two numerical experiments designed to validate the efficacy of the error estimators developed in the previous section. Note that, in the first example, the results presented in Theorems \eqref{4.3theorem} and \eqref{4.8theorem} emphasize the reliability and efficiency of the estimator $\Upsilon_{yzq}$. However, the estimator $\eta_q$ reflects the approximation error of the control and may not effectively guide the localization of refinement in certain significant scenarios. In the second example, we analyze the time-dependent singularity behavior that emphasizes how effectively the estimators indicate the trajectory of the singularity at various time steps.

For the purpose of marking strategy, we define the following conditions, for $i\in [1:N]$:
\begin{subequations}\label{5.2Numsbulk}
\begin{eqnarray}
\theta \sum_{K\in \mathscr{T}_h^i} \big\{(\eta^i_{y,K})^2+(\eta^i_{z,K})^2\big\}&\leq& \sum_{K\in \mathcal{M}_K^i} \big\{(\eta^i_{y,K})^2+(\eta^i_{z,K})^2\big\},\\
\theta \sum_{E\in \mathcal{E}_h^i} \big\{(\eta^i_{y,E})^2+(\eta^i_{z,E})^2+(\eta^i_{q,E})^2\big\}&\leq& \sum_{E\in \mathcal{M}_E^i} \big\{(\eta^i_{y,E})^2+(\eta^i_{z,E})^2+(\eta^i_{q,E})^2\big\},
\end{eqnarray}
\end{subequations}
where $\theta$ be a given parameter with $0 < \theta < 1$. Moreover,   $\mathcal{M}^i_K$ and $\mathcal{M}^i_E$ are the initial subsets chosen from $\mathscr{T}^i_h$ and $\mathcal{E}^i_h$, respectively, for refinement.
In the adaptive procedure, we use the following space-time adaptive algorithm for the optimization problem \eqref{contfunc}–\eqref{contspace}. 

Consequently, enhancing this significant step numerically becomes imperative. Hence, in our numerical tests, we adopt
\begin{equation}
\bar{\eta}_q=\sum_{i=1}^{N_T}\sum_{E_N\in \mathcal{E}^i_{N,h}} h_{E_N} \|(\mathbf{n}_{E_N}\cdot \nabla(\alpha \left(Q_h-\hat{q}_{d,h}\right)+\tilde{Z}_h))\chi_{_{\mathcal{I}_h}}\|_{L^2(I_i;L^2(E_N))}\nonumber
\end{equation}
as a control indicator instead of employing $\eta_q$ as defined in equation \eqref{4.19cq}. Nevertheless, the numerical approximation of the characteristic function $\chi_{_{\mathcal{I}_h}}$ is given by:
\begin{equation}
\chi_{_{\mathcal{I}_h}} \approx \frac{(Q_h-\hat{q}_{a,h})\times (\hat{q}_{b,h}-Q_h)}{h^\mu+(Q_h-\hat{q}_{a,h})\times (\hat{q}_{b,h}-Q_h)}
\end{equation}
with $\mu>0$ \cite{liliuyan2000comp}.  However, the projection formula, is used for the control, is defined as
\begin{equation}\label{5.4proj}
P_{[q_a,q_b]}(\phi):= \max\{q_a,\min\{q_b,\phi\}\},.
\end{equation}
\begin{algorithm}[H] 
\caption{\bf{: \tt Space-time Finite Element Adaptive Algorithm }}\label{algo5.1} 
\begin{small}
\begin{flushleft}
 Given parameters $\delta_1\in (0,1), \; \delta_2>1$, $\Lambda_1 >0$, $\Lambda_2\in (0, \Lambda_1)$, the space and time tolerances $\epsilon_{space}$ and  $\epsilon_{time}$, respectively. Suppose that $(y_h^{i-1}, z_h^i, q_h^{i-1})$ is computed on the mesh $\mathscr{T}_h^{i-1}$ at the time level $t_{i-1}$ with time step size $k_{i-1}$. \vspace{.2cm}\\ 

\noindent
\textbf{Step-1}~~{\texttt{Set}} \hspace{.05cm} $\mathscr{T}_h^i:=\mathscr{T}_h^{i-1}$, $\mathcal{E}^i_h:=\mathcal{E}^{i-1}_h$,   $k_i:=k_{i-1}$ and $t_{i}:=t_{i-1}+k_i$. \vspace{.2cm}\\

 \hspace{3.3cm} Evaluate $(y_h^{i}, z_h^{i-1}, q_h^{i})$ using the discrete problem \eqref{disoptstate12}--\eqref{disfirstoptcond12} on $\mathscr{T}_h^i$ \vspace{.1cm}\\
\hspace{3.3cm} with boundary mesh $\mathcal{E}^i_h$ and the data $(y_h^{i-1}, z_h^i, q_h^{i-1})$\\
\hspace{3.3cm} Evaluate all estimators. \vspace{.2cm}\\

\noindent
\textbf{Step-2}~~{\bf\texttt{While}} \hspace{.05cm}$\Xi_{T_{yz}}\geq \Lambda_1 \frac{\epsilon_{time}}{T}$, ~~ {\bf \texttt{do}} \vspace{.1cm}\\
\hspace{3.2cm} $k_i:=\delta_1 k_{i-1},\;\; t_i:=t_{i-1}+k_{i}$ \vspace{.1cm}\\
\hspace{3.2cm} Evaluate $(y_h^{i}, z_h^{i-1}, q_h^{i})$ using the discrete problem \eqref{disoptstate12}--\eqref{disfirstoptcond12} on $\mathscr{T}_h^i$ \vspace{.1cm}\\
\hspace{3.2cm} Evaluate all estimators. \vspace{.1cm}\\
\hspace{1.35cm} {\bf \texttt{End While}} \vspace{.2cm} \\

\noindent
\textbf{Step-3}~~{\bf \texttt{While}} $\Upsilon_{yzq}>\frac{\epsilon_{space}}{T}$, or 
conditions \eqref{5.2Numsbulk} are not satisfied, {\bf \texttt{do}}\vspace{.1cm}\\
\hspace{3.5cm} Mark the subsets  $\mathcal{M}^i_K\subset \mathscr{T}^i_h$ and $\mathcal{M}^i_E\subset \mathcal{E}^i_h$ such that  $\mathcal{M}^i_K$ and $\mathcal{M}^i_E$ \vspace{.1cm}\\
\hspace{3.5cm} contain at least one element $\mathscr{T}^i_h$ and $\mathcal{E}^i_h$, respectively, \vspace{.1cm}\\
\hspace{3.5cm}  with the largest error indicators.\vspace{.2cm}\\
\hspace{3.5cm} Refine the mesh  $\mathscr{T}^i_h$ and  $\mathcal{E}^i_h$, and generate new mesh $\mathscr{T}^{i}_{\hat{h}_i}$ and  $\mathcal{E}^{i}_{\hat{h}_i}$ (say). \vspace{.1cm}\\
 \hspace{3.5cm} Evaluate $(y_h^{i}, z_h^{i-1}, q_h^{i})$ using the discrete problem \eqref{disoptstate12}--\eqref{disfirstoptcond12} on $\mathscr{T}_{\hat{h}_i}^i$  \vspace{.1cm}\\
\hspace{2.5cm} {\bf \texttt{While}}\hspace{.05cm}  $\Xi_{T_{yz}}\geq \Lambda_1 \frac{\epsilon_{time}}{T}$, {\bf \texttt{do}} \vspace{.1cm}\\
\hspace{4cm} $ k'_i:=\delta_1 k_{i-1},\;\; t_i:=t_{i-1}+k'_{i}$ \vspace{.2cm}\\
\hspace{4cm} Evaluate $(y_h^{i}, z_h^{i-1}, q_h^{i})$ using the discrete problem \eqref{disoptstate12}--\eqref{disfirstoptcond12} on $\mathscr{T}_{\hat{h}_i}^{i,k'_i}$. \vspace{.1cm}\\
\hspace{4cm} Evaluate all estimators. \vspace{.1cm}\\
\hspace{2.5cm} {\bf \texttt{End While}}\hspace{.05cm}\\
\hspace{1.8cm} {\bf \texttt{End While}}\vspace{.2cm}\\

\noindent
\textbf{Step-4}~~  {\bf \texttt{If}}\hspace{.05cm}  $\Xi_{T_{yz}}\leq \Lambda_2 \frac{\epsilon_{time}}{T}$, {\bf \texttt{do}} \vspace{0.1cm}\\
\hspace{3cm} $ k'_i:=\delta_2 k_{i-1},\;\; t_i:=t_{i-1}+k'_{i}$ \vspace{.1cm}\\
\hspace{1.7cm} {\bf \texttt{End If}}\hspace{.05cm} \vspace{.2cm}\\
\noindent
\textbf{Step-5}~~ {\tt Set} $i:=i+1$ and {\sc Go To Step 1} \vspace{.15cm}\\
\hspace{3cm} Repeat the process
\end{flushleft}
\end{small}
\end{algorithm}
\noindent
Moreover, the effectivity index (Eff. Index) is computed using the following formula:
$\text{Eff. Index} := \text{Est~Err}/\text{True~Err},$
where the estimated  and the true errors are defined by $\text{Est~Err}:=\Upsilon_{yzq} ~ \text{and}~\text{True~Err}:=|\|y-Y_h|\|_{L^2(I;\mathscr{E})}+|\|z-Z_h|\|_{L^2(I;\mathscr{E})}+\|q-Q_h\|_{L^2(I;L^2(\Gamma_{_N}))}+\|\mu-\tilde{\mu}_h\|_{L^2(I;L^2( \Gamma_{_N} ))}$, respectively.

In the computation, the regularization value $\alpha$ is assumed equal to the parameter $\gamma$ employed in defining the active and inactive sets. Additionally, equal tolerances are applied to both time and space considerations in numerical computations, i.e., $\epsilon_{time}=\epsilon_{space}=\epsilon$ (say).  
The following modified exam is taken from \cite{siebertrosch2014}.
\begin{exam}\label{ex5.1}
We consider a convex domain $\Omega=[0,3]\times[0,3]$ with boundary $\Gamma$, where the Neumann boundary $\Gamma_{_N}=\Gamma$ and the time domain $I=[0,1]$. The control $q$ is acting only on the part of the boundary $\Gamma$, say, $\bar{\Gamma}_{N}=\{0\}\times [1,2]$. However, we assume that the reaction term $a_0=1$ and the regularization parameter $\alpha =1$. Further, we choose the following  analytical solutions
\begin{align*}
&y(x, t)~=~ t\times e^{-10(x_1^2+x_2^2)}, \quad \text{where} \quad  x=(x_1,x_2),\\
&z(x, t)~=~ \frac{(1-t)\mathpzc{K}}{2n}\times\Big[(2n+1)\times\big(\frac{2x_2}{3}-1\big)-\big(\frac{2x_2}{3}-1\big)^{2n+1} \Big],\\
&q(x, t)~=~ P_{[q_a, q_b]} \big(z(x, t)\big), \quad
\mu(x, t)~=~ z(x,t)-P_{[q_a, q_b]} \big(z(x, t)\big), 
\end{align*}
and the given data as follows
\begin{align*}
&f(x,t)~=~\big\{[41-400\times (x_1^2+x_2^2)]\times t +1\big\}\times e^{-10(x_1^2+x_2^2)}, \\
&y_d(x,t)~=~ t\times e^{-10(x_1^2+x_2^2)}+\frac{(t-2)\mathpzc{K}}{2n}\times\Big[(2n+1)\times\big(\frac{2x_2}{3}-1\big)-\big(\frac{2x_2}{3}-1\big)^{2n+1} \Big]\\&\hspace{1.8cm}-\big(\frac{8n+4}{9}\big)\times(1-t)\mathpzc{K}\times\big(\frac{2x_2}{3}-1\big)^{2n-1},  \\
&g_{_N}(x,t)~=~
\begin{cases}
-P_{[q_a, q_b]} \big(z(x, t)\big), \quad (x,t)\in\bar{\Gamma}_N \times I,\\
-60\times t \times e^{-10(9+x_2^2)}, \quad x_1=3,\quad \text{and} \quad t \in I,\\
-60\times t \times e^{-10(x_1^2+9)}, \quad x_2=3,\quad \text{and} \quad t \in I,\\
\;\;\;0, \quad \text{Otherwise},
\end{cases}
\end{align*}
$\quad q_d(x,t)~=~0$ and $r_{_N}(x,t)~=~0,$ respectively, with $\mathpzc{K}=10$ and $n=20$.
\end{exam}
We provide a comprehensive analysis of errors concerning the state, adjoint-state, control and co-control variables observed on both uniform and adaptive meshes with fixing $\theta =0.40$ and the tolerances $(\epsilon=)\; 0.001$. The errors for the state and adjoint-state are evaluated in the energy norm, represented as $L^2(I;\mathscr{E})$-norm, while for the control and co-control are assessed in the $L^2(I; L^2(\Gamma))$-norm. 
Adaptive meshes (Figure \ref{meshuniadpt} (right)) are generated by employing error indicators based on an initial uniform mesh (Figure \ref{meshuniadpt} (left)) with mesh size $\tt N_x\times N_y=80\times 80$. Notably, Figure \ref{meshuniadpt} 
indicates a significant improvement in the number of elements and nodes within the adapted meshes compared to the uniform mesh. The distinct characteristics of the state, adjoint-state, and control variables require localized refinements across different regions of the domain, as illustrated in Figure \ref{meshuniadpt} (right). Due to the small exponential peak in the state variable $(y)$, adaptive meshes are particularly concentrated around the origin, as depicted in the right portion of Figure \ref{meshuniadpt}. Figure \ref{meshuniadpt} (left and right) demonstrates that the utilized error estimators effectively adapt the meshes to the pertinent areas. Figure \ref{appmeshuex} displays the profiles of the approximated state (left) and adjoint-state (right) variables on adaptive meshes at penalty 
\begin{figure}[H]
\begin{center}
\includegraphics[width=0.49\textwidth]{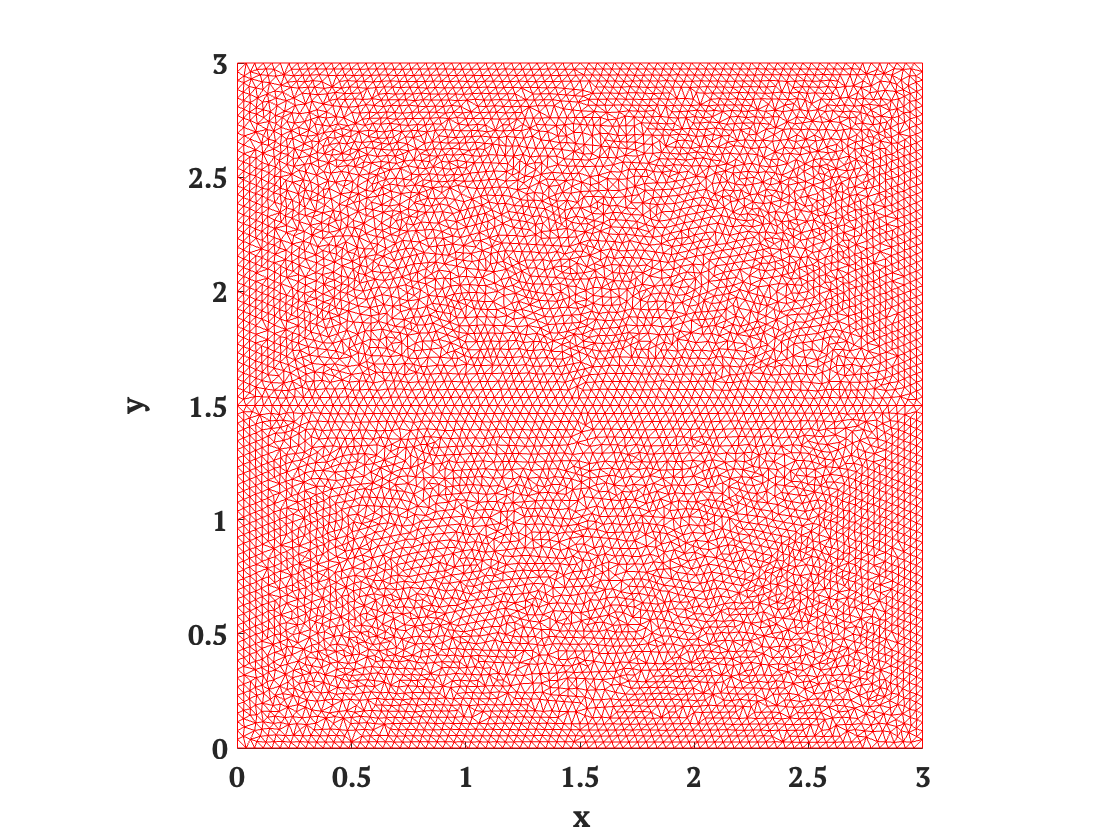}  \hfill
\includegraphics[width=0.49\textwidth]{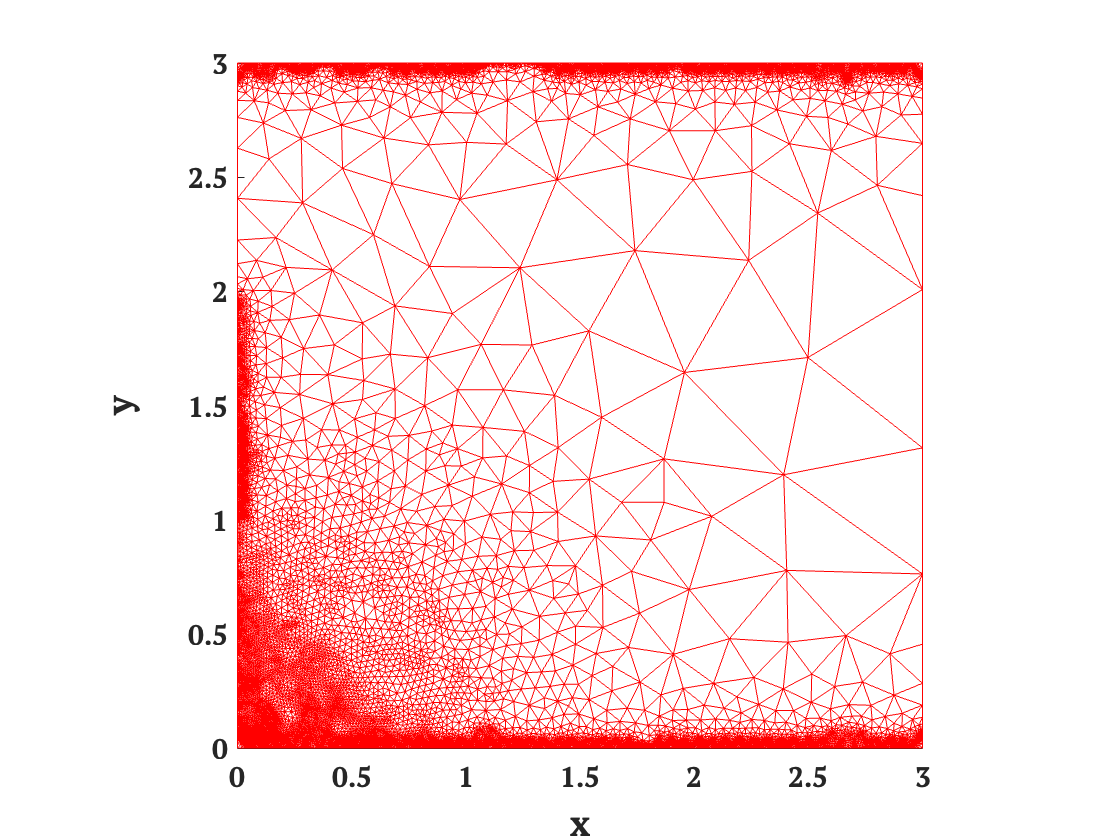}
\caption{Uniform meshes with size $80\times 80$ (NDOF=$91164$) and adaptively generated meshes at $3^{rd}$-iteration (NDOF=$333804$), respectively.}\label{meshuniadpt} \vspace{0.5cm}
\end{center}
\begin{center}
\includegraphics[width=0.49\textwidth]{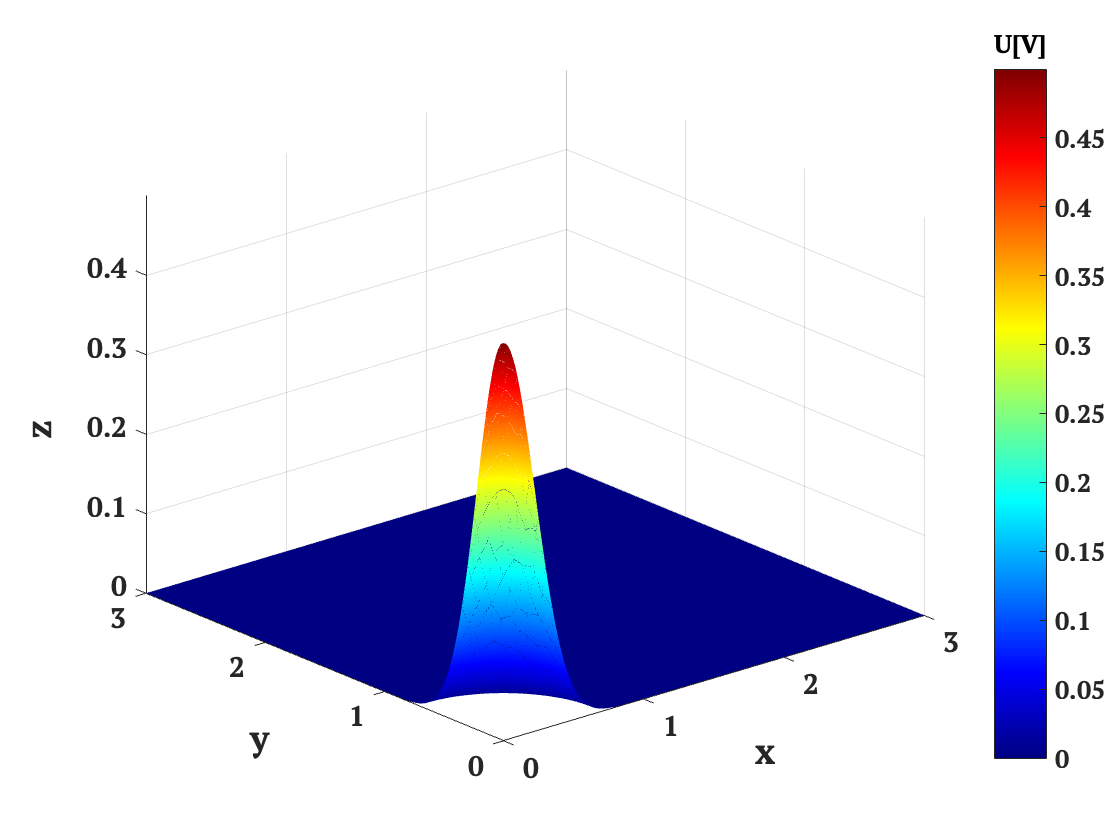}  \hfill 
\includegraphics[width=0.49\textwidth]{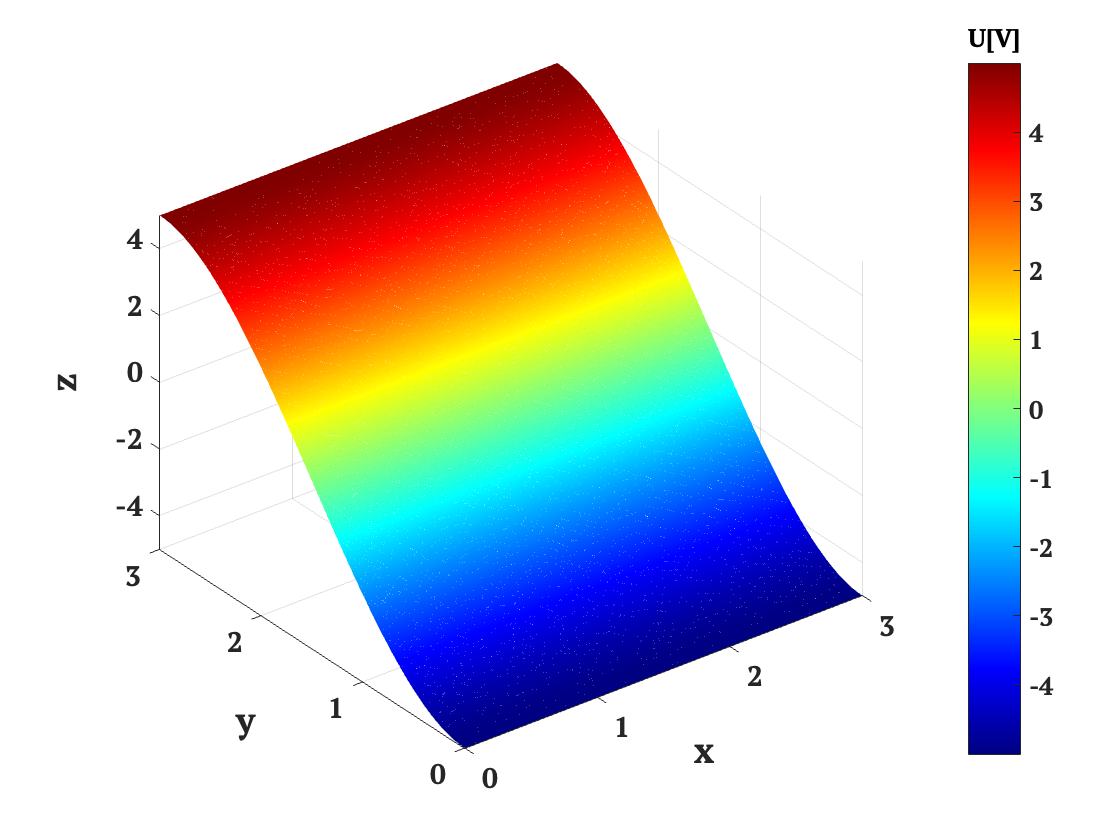} 
\caption{Profile of the approximated DG state and adjoint-state  on adaptive meshes with   $\sigma_0=100$.}\label{appmeshuex}  \vspace{1cm}
\end{center}
\begin{center}
\includegraphics[width=0.49\textwidth]{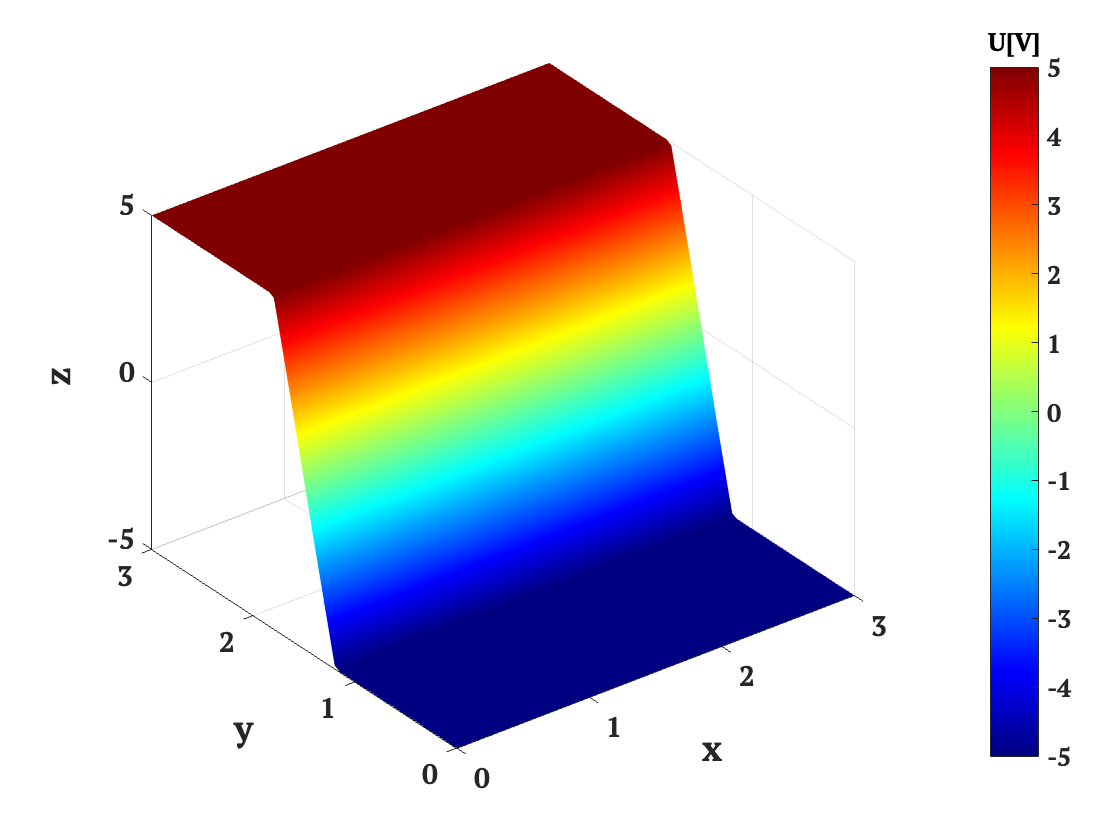}   \hfill
 \includegraphics[width=0.49\textwidth]{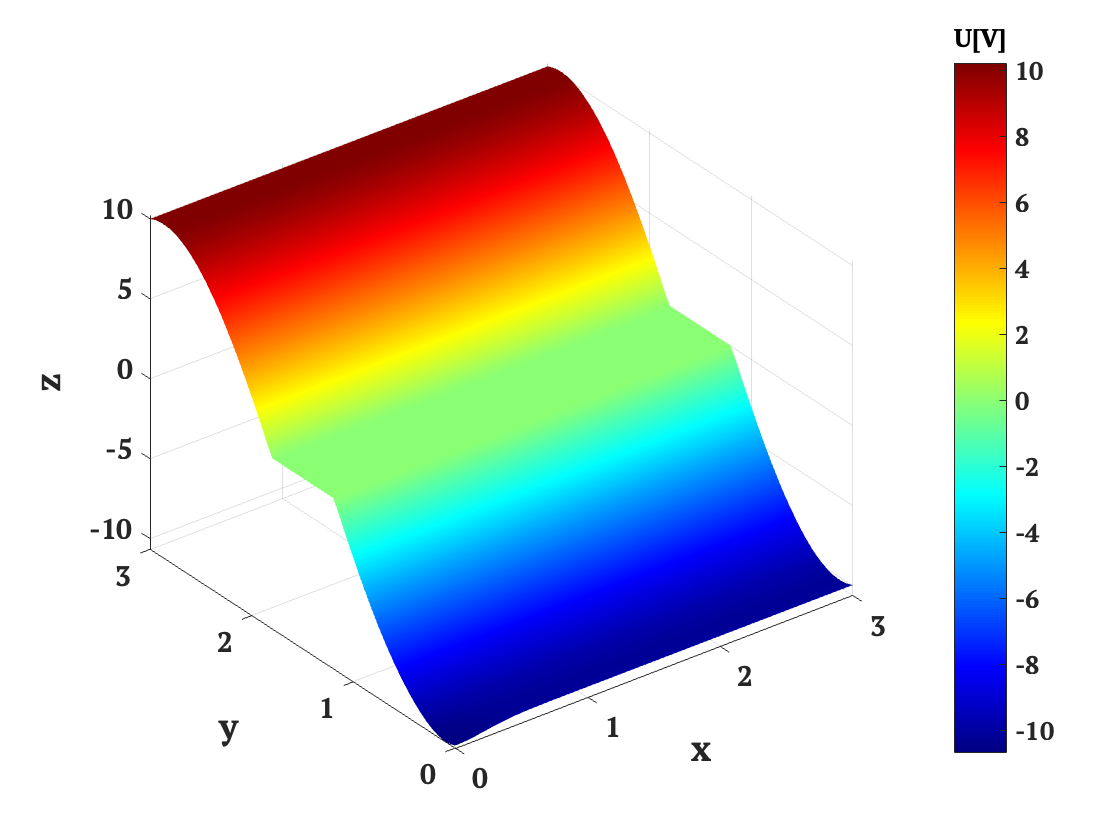}
\caption{Profile of approximated DG control and and co-control on adaptive meshes.}\label{controlcocontrol} 
\end{center}
\end{figure}
\begin{figure}[H]
\begin{center}
\includegraphics[width=0.45\textwidth]{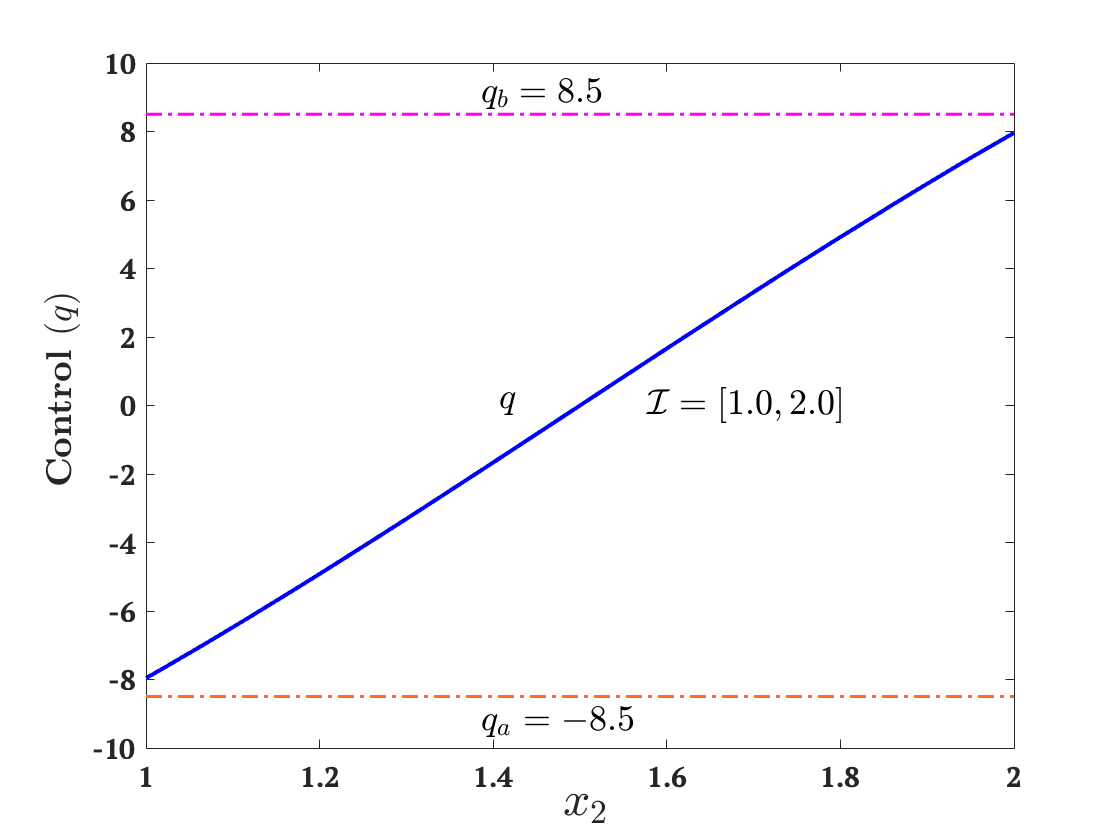}
\hfill
\includegraphics[width=0.45\textwidth]{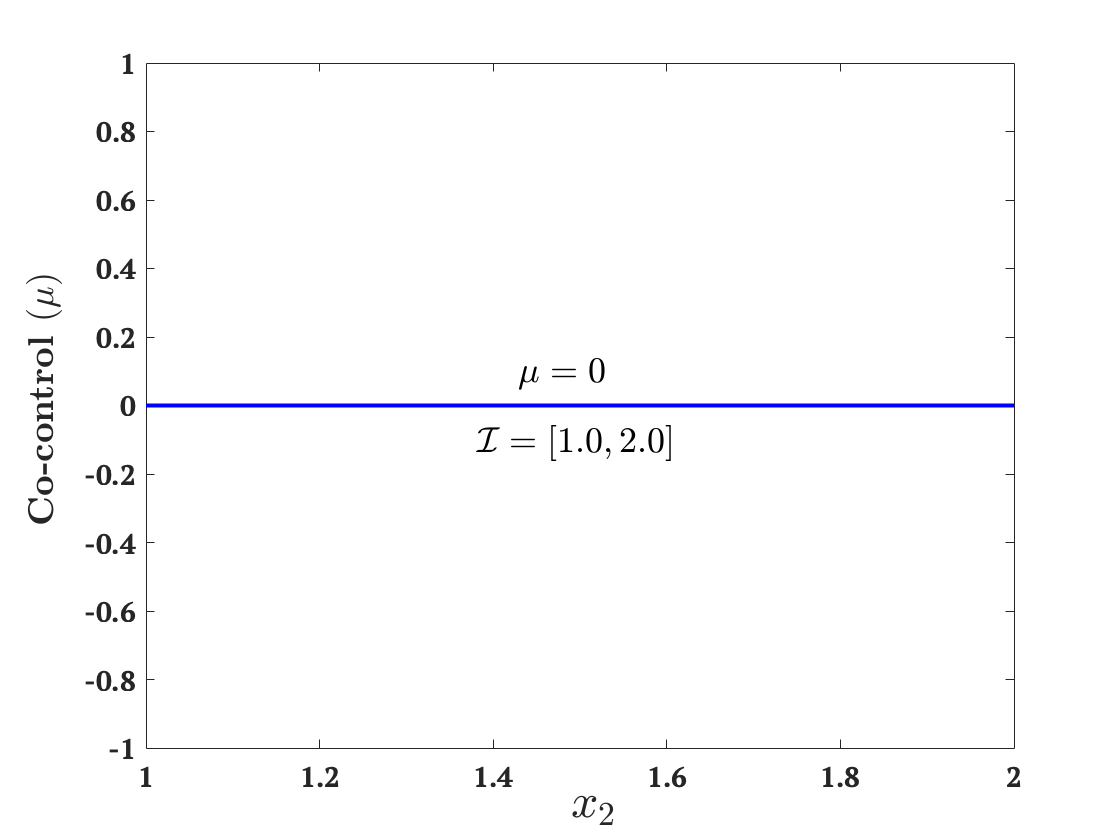}
\caption{Computation of control (left) and co-control  (right) on the inactive set $\mathcal{I}=\{0\}\times [1, 2]$.}\label{inactiveset}  \vspace{1cm}
\includegraphics[width=0.45\textwidth]{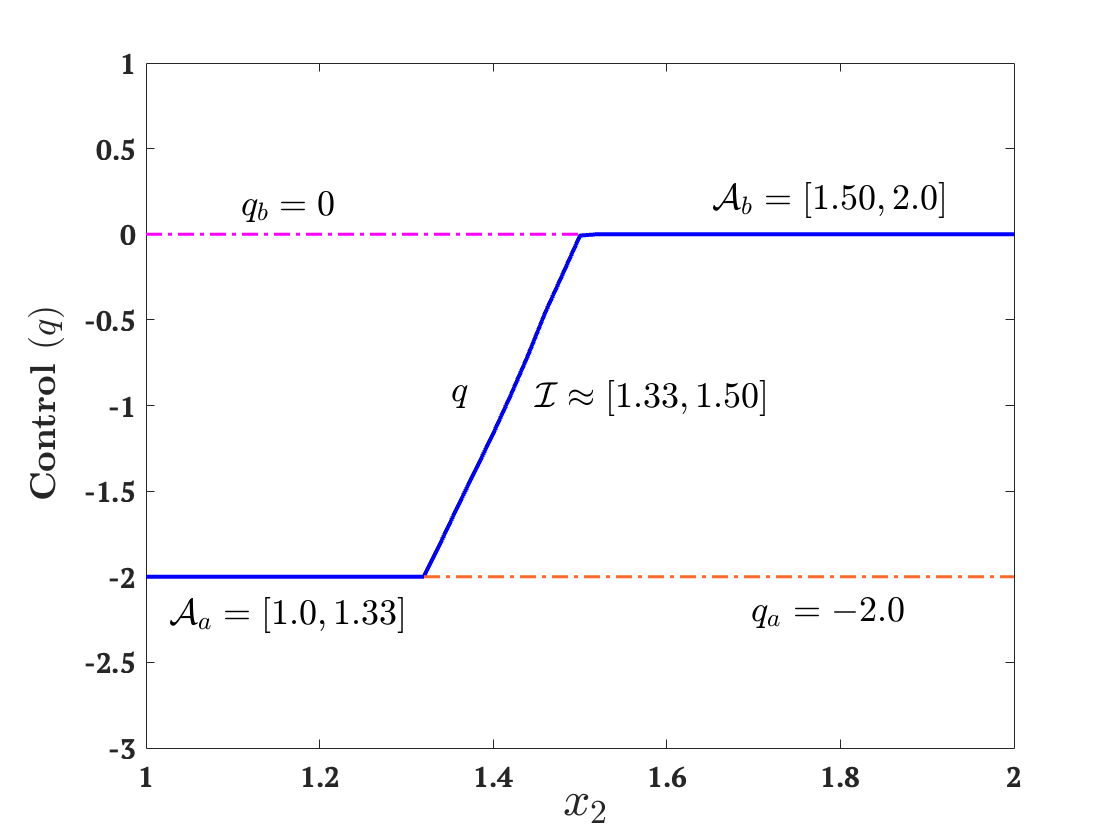}\hfill 
\includegraphics[width=0.45\textwidth]{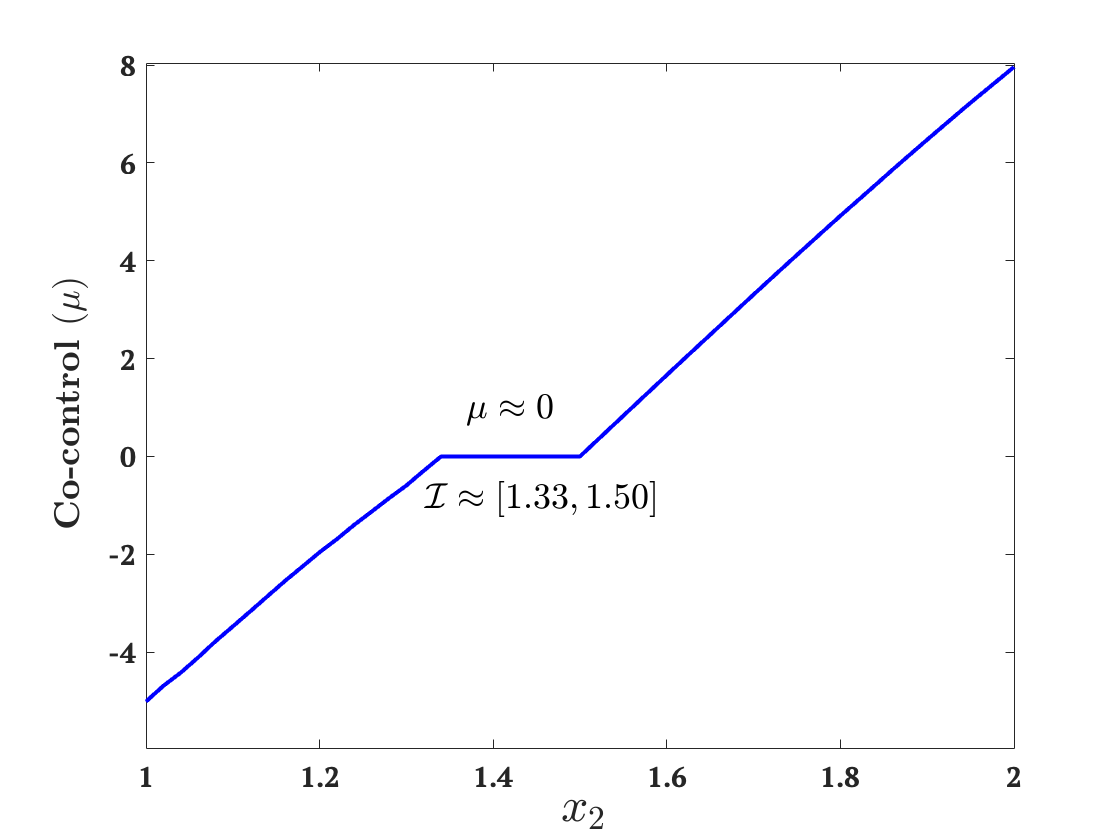}
\caption{Computation of control  (left) and co-control (right) on the active-inactive sets $\mathcal{A}=\{0\}\times (\mathcal{A}_a \cup \mathcal{A}_b)$ and $\mathcal{I}=\{0\}\times [1.33, 1.50]$, respectively.}\label{mixedqmu20} \vspace{1cm}
\includegraphics[width=0.45\textwidth]{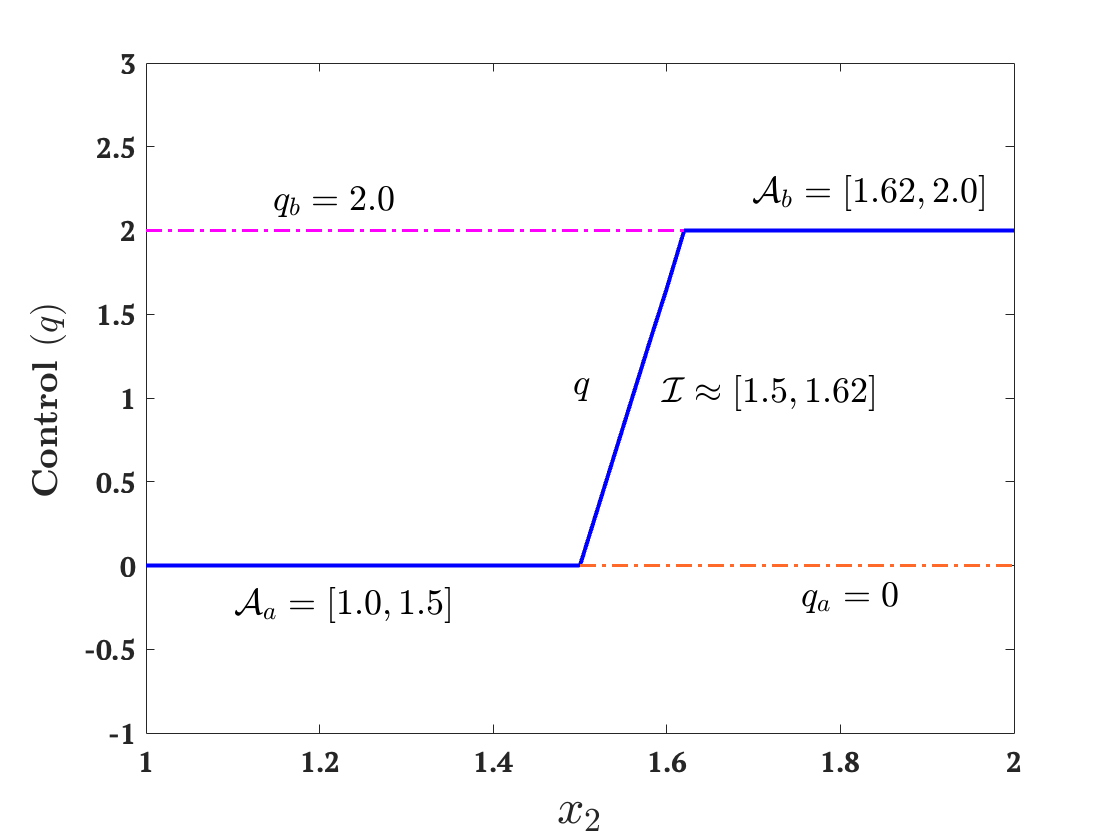} \hfill
\includegraphics[width=0.45\textwidth]{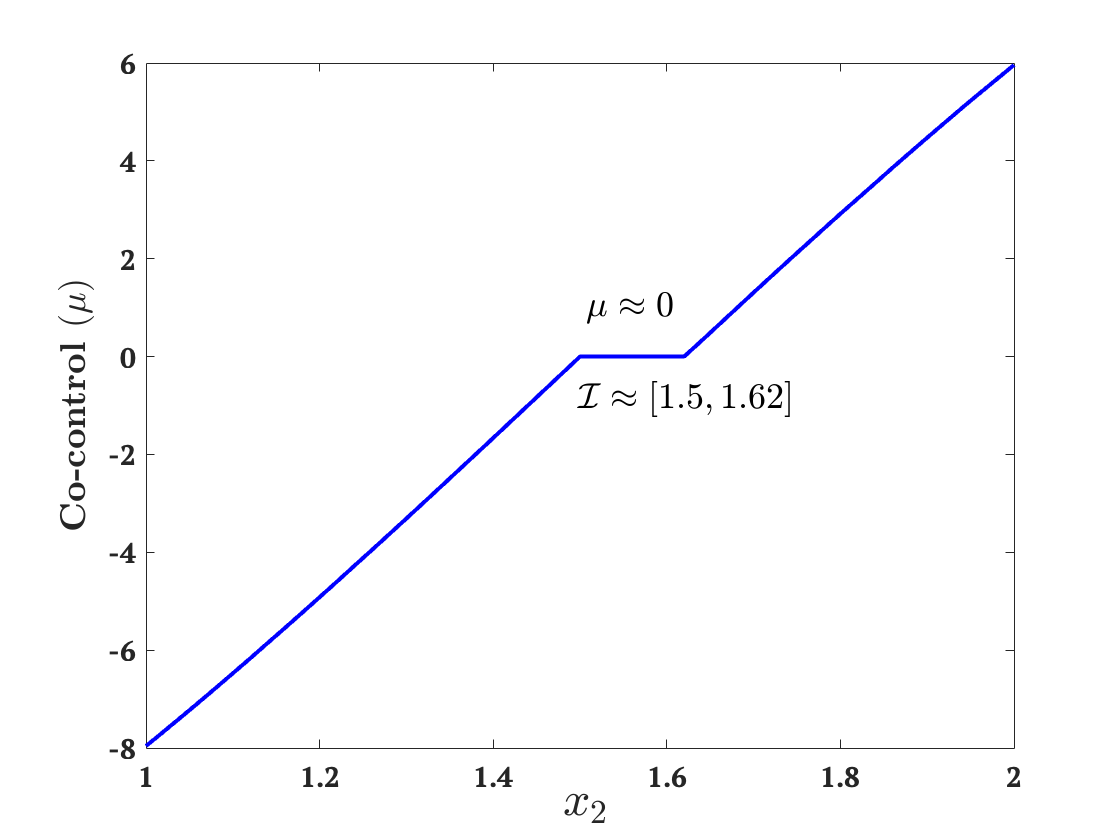}
\caption{Computation of control (left) and co-control (right) on the active-inactive sets $\mathcal{A}=\{0\}\times (\mathcal{A}_a \cup \mathcal{A}_b)$ and $\mathcal{I}=\{0\}\times [1.50, 1.62]$, respectively.} \label{qpp02mup02} 
\end{center}
\end{figure}
\begin{figure}[H]
\begin{center}
\includegraphics[width=0.45\textwidth]{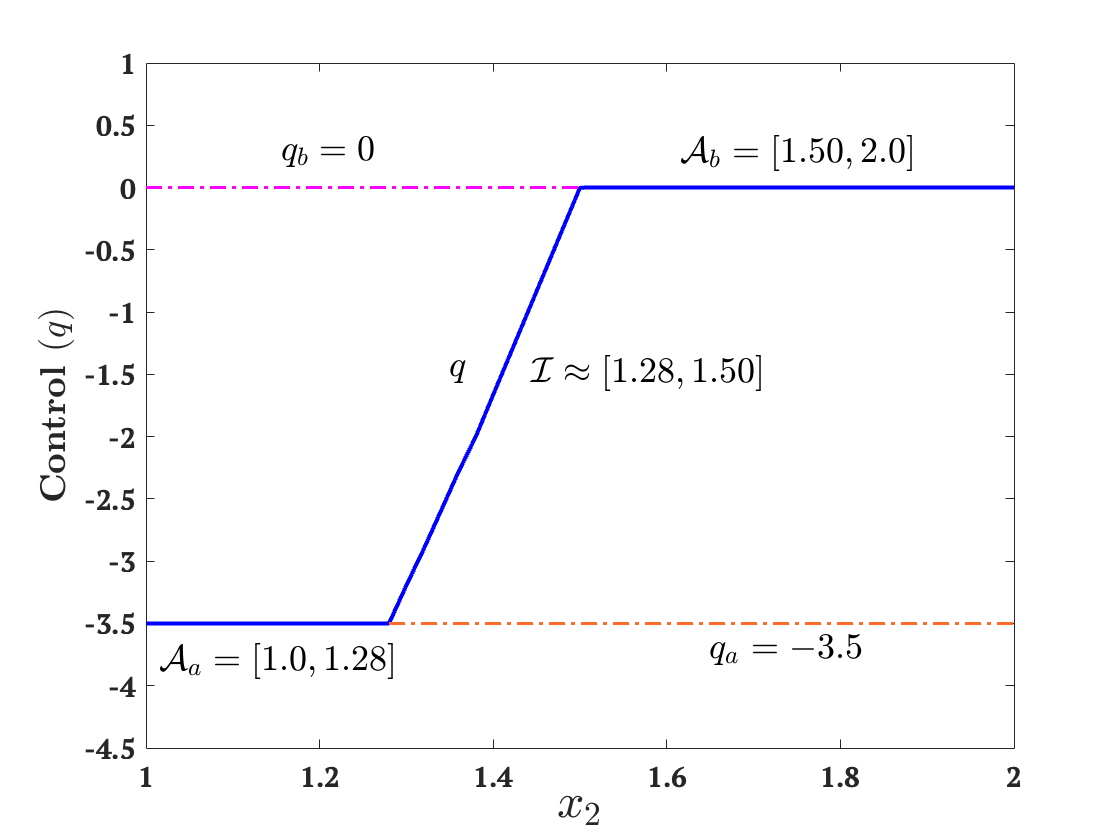}
\hfill
\includegraphics[width=0.45\textwidth]{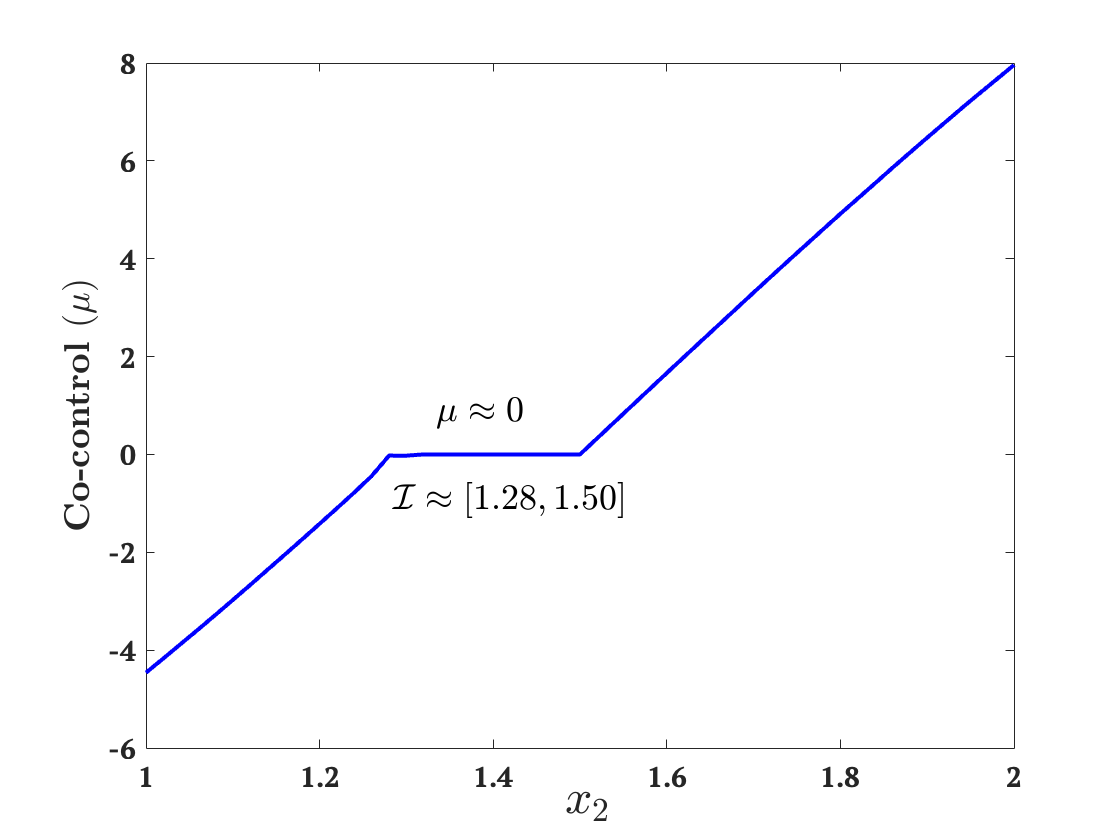}
\caption{Computation of control $(q)$ (left) and co-control $(\mu)$ (right) on the active-inactive sets $\mathcal{A}=\{0\}\times (\mathcal{A}_a \cup \mathcal{A}_b)$ and  $\mathcal{I}=\{0\}\times [1.28, 1.50]$.}\label{inactive3.50set} \vspace{1cm}
\includegraphics[width=0.45\textwidth]{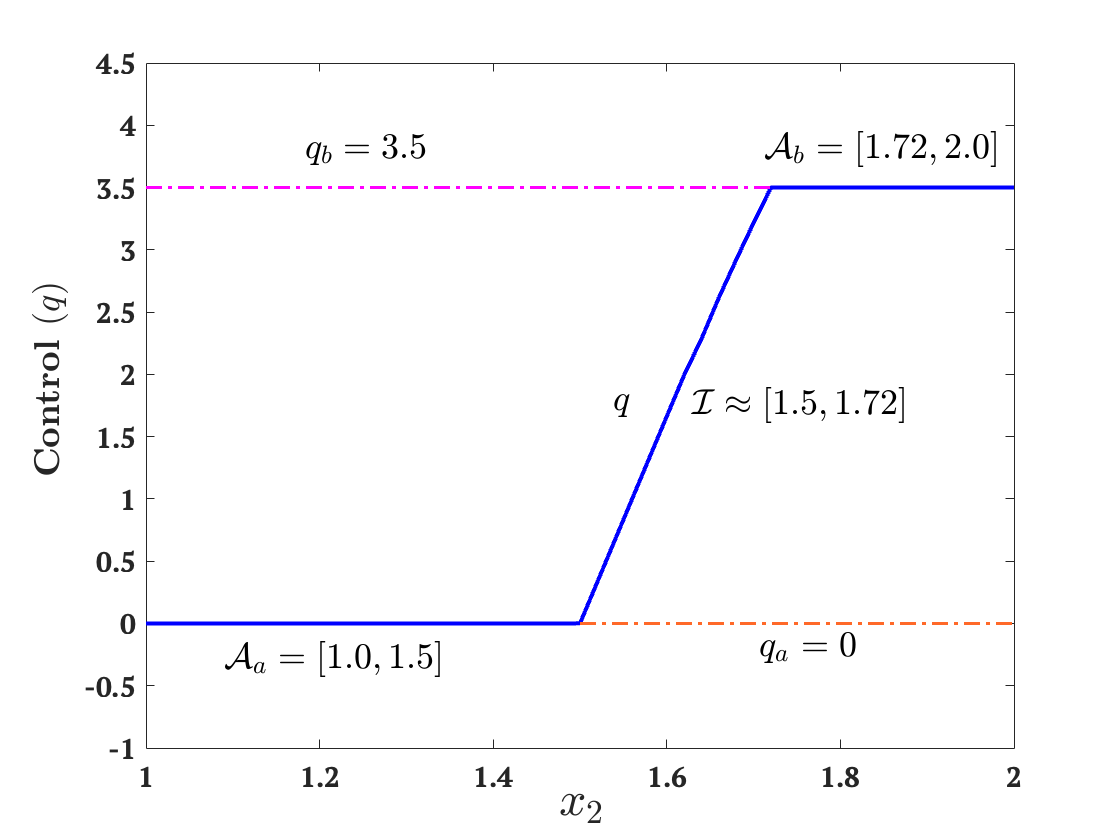}\hfill 
\includegraphics[width=0.45\textwidth]{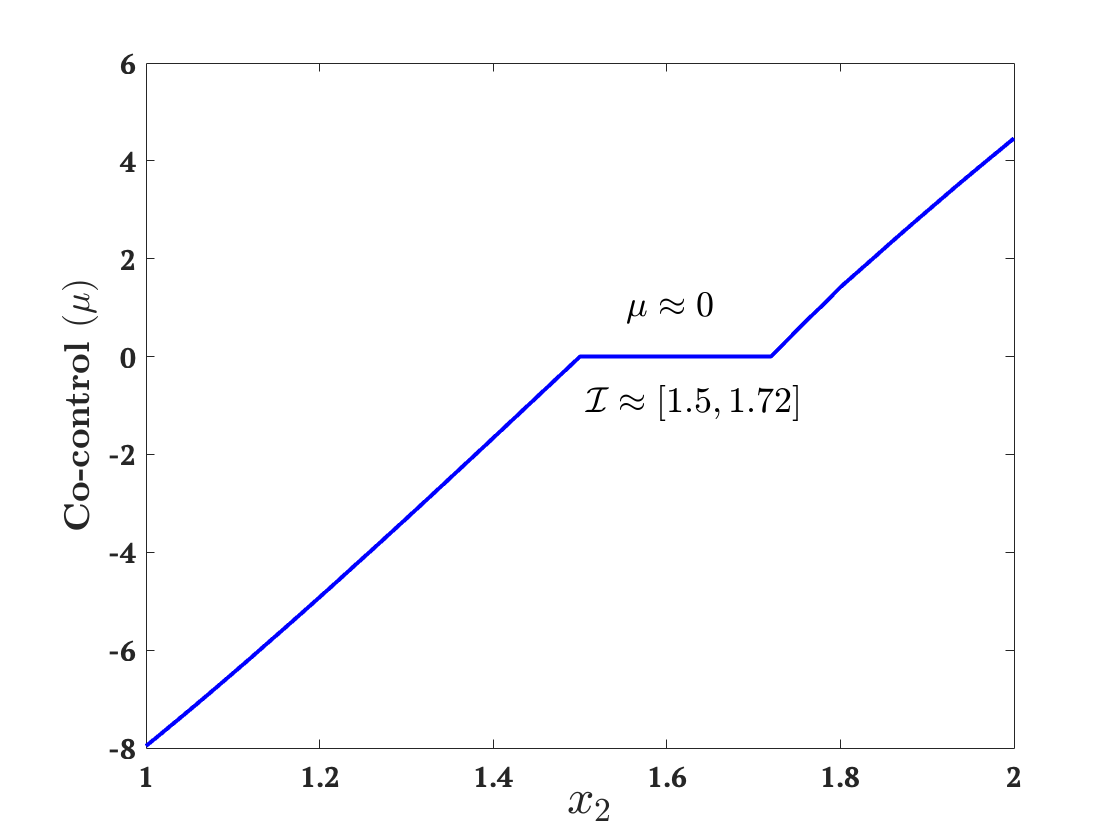}
\caption{Computation of control (left) and co-control  (right) on the active-inactive sets $\mathcal{A}=\{0\}\times (\mathcal{A}_a \cup \mathcal{A}_b)$ and $\mathcal{I}=\{0\}\times [1.50, 1.72]$, respectively.}\label{mixedqmu} \vspace{1cm}
\includegraphics[width=0.45\textwidth]{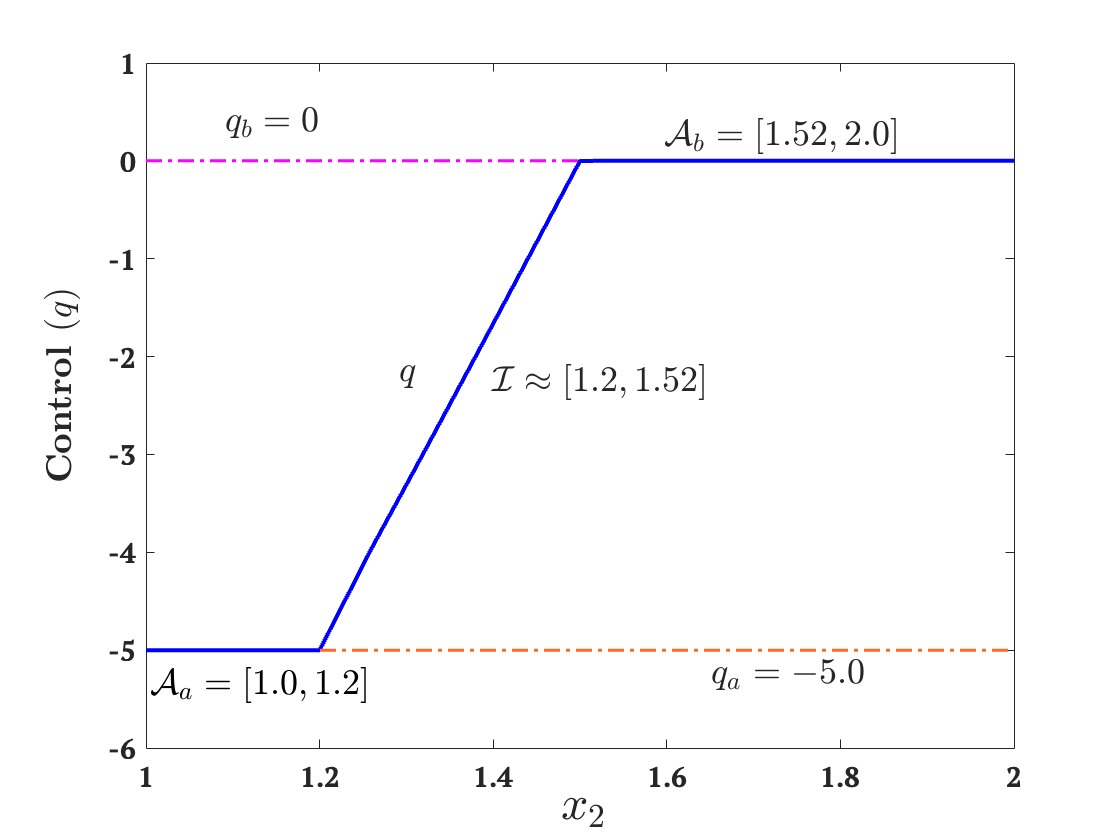} \hfill
\includegraphics[width=0.45\textwidth]{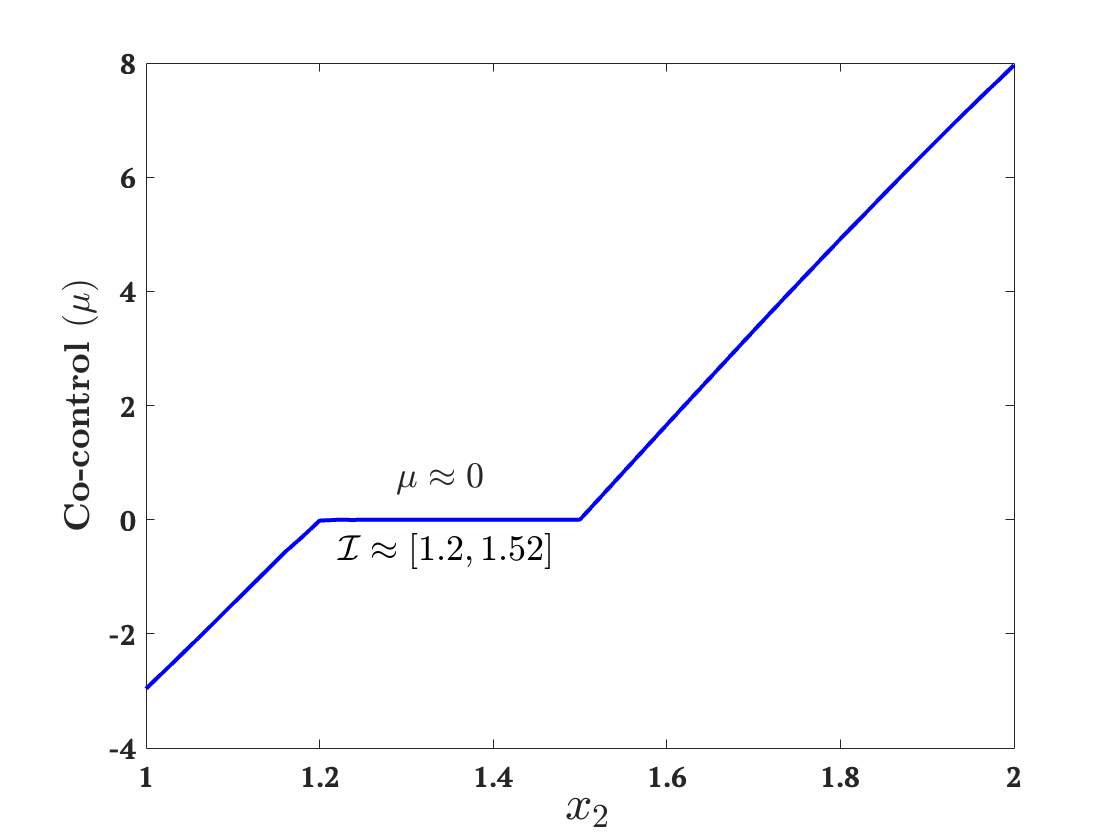}
\caption{Computation of control (left) and co-control (right) on the active-inactive sets $\mathcal{A}=\{0\}\times (\mathcal{A}_a \cup \mathcal{A}_b)$ and $\mathcal{I}=\{0\}\times [1.2, 1.52]$, respectively.} \label{qpp02mup02} 
\end{center}
\end{figure}
\begin{figure}[H]
\begin{center}
\includegraphics[width=0.45\textwidth]{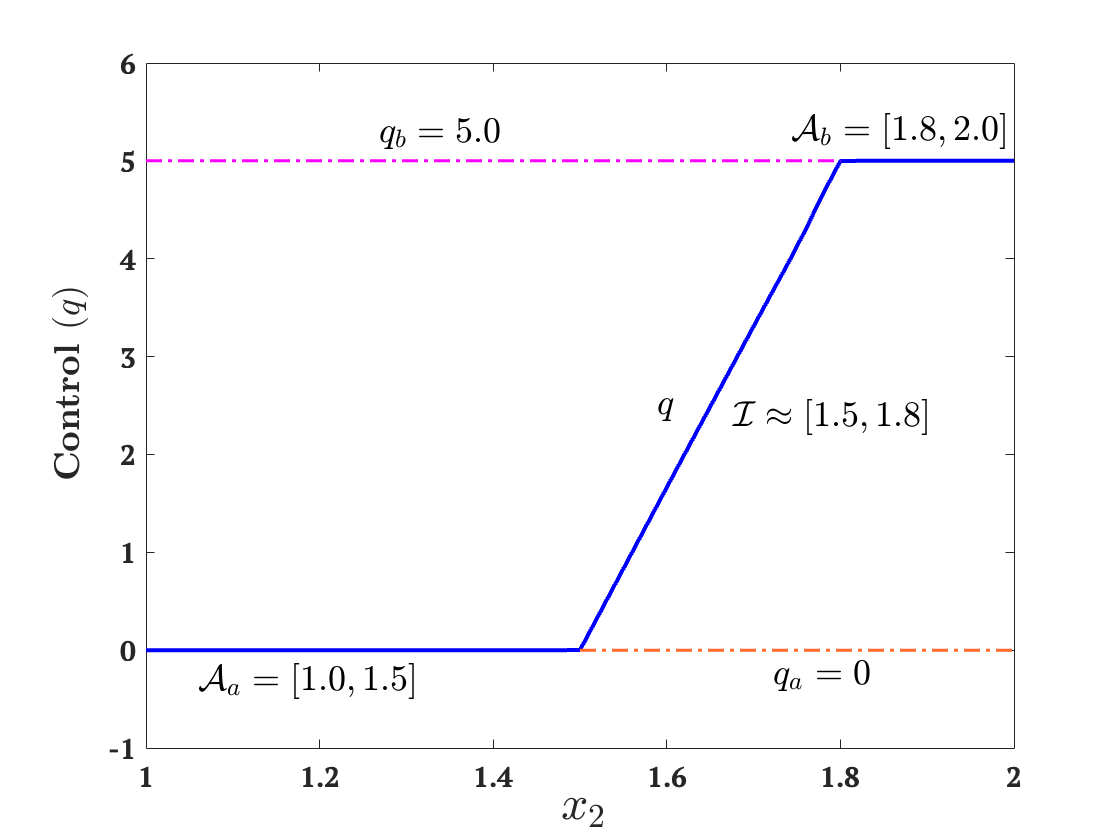} \hfill
\includegraphics[width=0.45\textwidth]{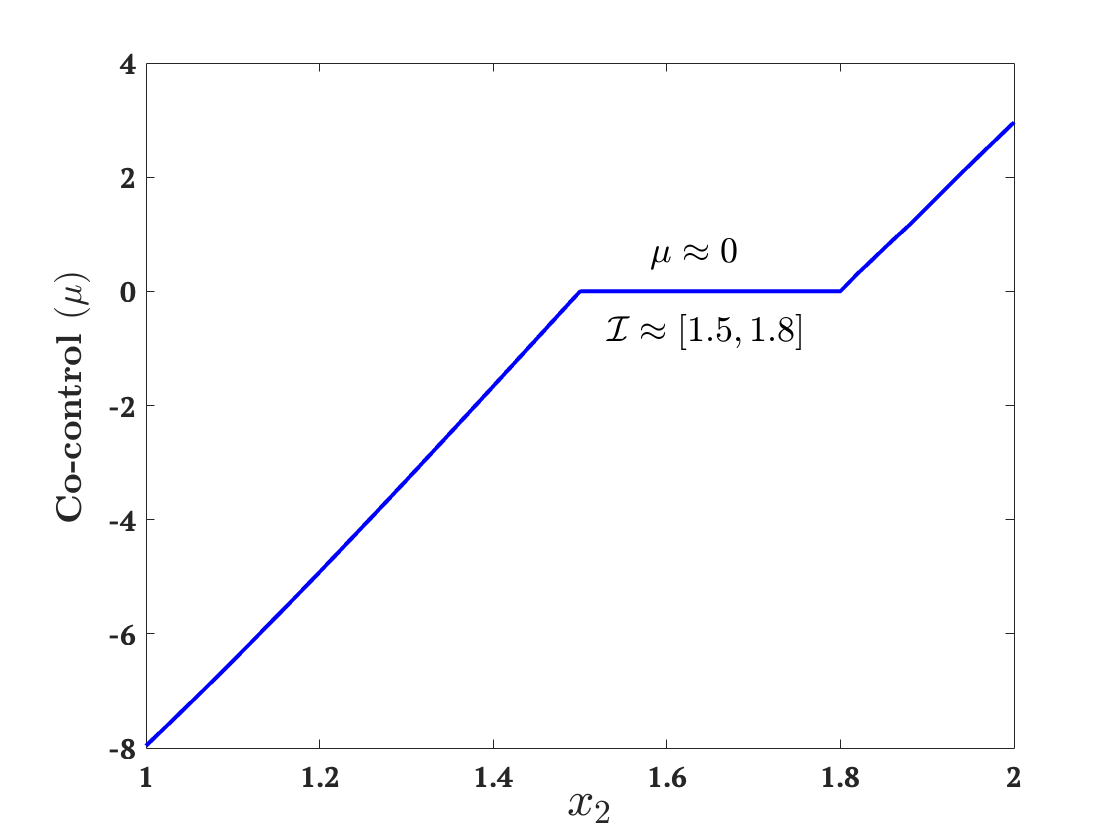}
\caption{Computation of control  (left) and co-control (right) on the active-inactive sets $\mathcal{A}=\{0\}\times (\mathcal{A}_a \cup \mathcal{A}_b)$ and $\mathcal{I}=\{0\}\times [1.5, 1.8]$, respectively.} \label{qpp05mup05} \vspace{1cm}
\includegraphics[width=0.45\textwidth]{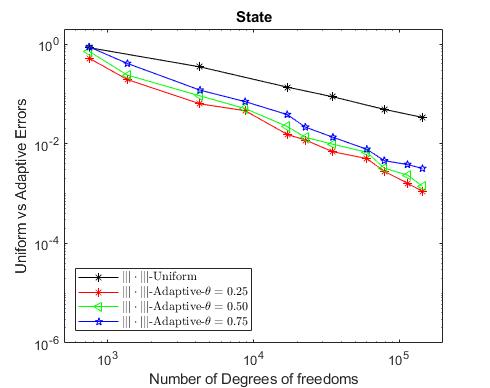}
\hfill
\includegraphics[width=0.45\textwidth]{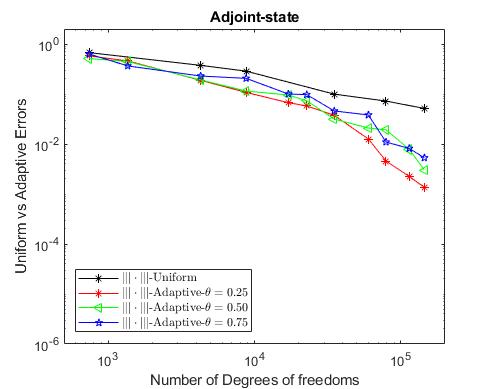}
\caption{The state (left) and adjoint-state (right) error plots on the uniform and adaptive  meshes.}\label{stateadjoint-stateerrors} \vspace{1cm}
\includegraphics[width=0.45\textwidth]{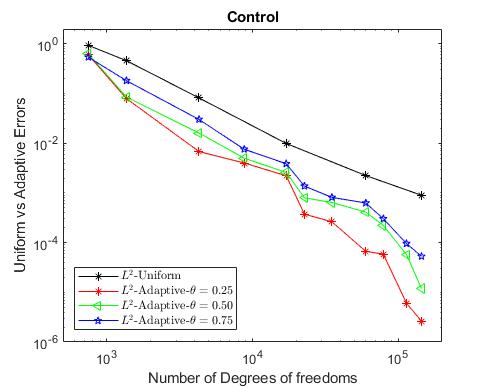}\hfill 
\includegraphics[width=0.45\textwidth]{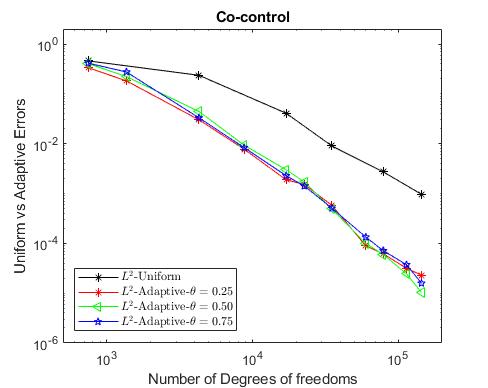}
\caption{The control (left) and co-control (right) error plots on the uniform and adaptive  mesh.}\label{controlcocontrolerrors} 
\end{center}
\end{figure}
\begin{figure}[H]
\begin{center}
\includegraphics[width=0.43\textwidth]{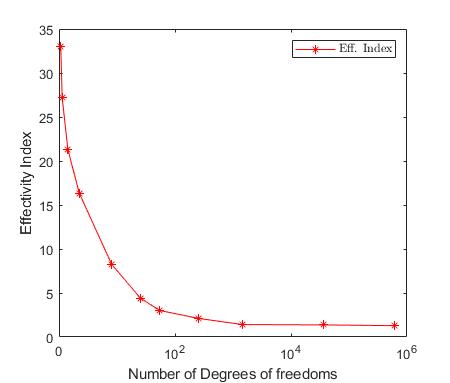} \hfill
\includegraphics[width=0.43\textwidth]{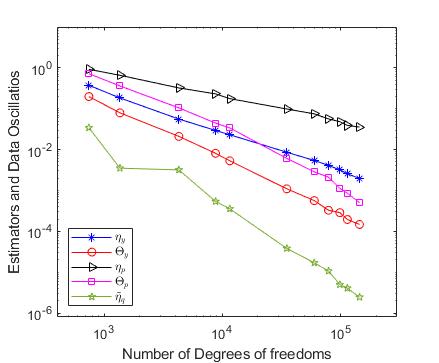}
\end{center}
\caption{Plots of the effective index (left) and the components of the error estimator and data oscillations (left).}\label{iffind} 
\end{figure}
\noindent 
parameter $\sigma_0=10^2$, $q_a=-5.0$, $q_b=5.0$ and time $t=0.5$, while the Figure \ref{controlcocontrol} shows the approximated control (left) and the co-control (right).

Further, the active and inactive sets are analyzed for the different choices for $q_a$ and $q_b$, as illustrated in Figures \ref{inactiveset}--\ref{qpp05mup05}. First, we selected $q_a=-8.5$ and $q_b=8.5$, and we saw that the controls are inactive in the set $\mathcal{I}=[1, 2]$ for this case (cf., Figures \ref{inactiveset} (left)) and the co-control $\mu$ is vanishing according to the definition, which is efficiently captured by the indicators (cf., Figures \ref{inactiveset} (right)). Any choice of $q_a$ and $q_b$ between $-8.0$ and $8.0$, the controls are active in some regions and inactive in some of the regions, as presented in Figures \ref{mixedqmu20}-\ref{qpp05mup05} and the corresponding co-controls thereof. We observe that the co-controls disappear in the portion in which the control is inactive $\mathcal{I}$; refer to right Figures \ref{mixedqmu20}-\ref{qpp05mup05}. Further, the active sets are denoted by $\mathcal{A}$, which are clearly mentioned in the figures.

Moreover, Figure \ref{stateadjoint-stateerrors} illustrates state (left) and adjoint-state (right) errors in uniform and adaptive meshes by setting tolerances $\epsilon=0.00001$ and
various marking parameters $\theta=0.25$, $\theta=0.50$, and $\theta=0.75$ in the bulk strategy. Similarly, Figure \ref{controlcocontrolerrors} displays the control and co-control errors in uniform and adaptive meshes using $\theta=0.25, 0.50$ and $0.75$, respectively. The effective index (left) and the components of the error estimator and the data oscillations plots (left) with the marking parameter $\theta=0.45$ are presented in Figure \ref{iffind}.
\begin{exam}
Consider the space domain $\Omega=[-1,1]\times[-1, 1]$ with boundary $\Gamma= \bar{\Gamma}_1 \cup \bar{\Gamma}_2\cup\bar{\Gamma}_3 \cup \bar{\Gamma}_4$, where $\Gamma_1=[-1,1]\times\{-1\},\; \Gamma_2=\{1\}\times [-1,1],\;\Gamma_3=[-1,1]\times\{1\},\; \Gamma_4=\{-1\}\times [-1,1]$. In addition, we consider the time domain $I=[0, 1]$, however the datum is considered as follows:
\begin{align*}
&z(x,t)=y(x,t)=0.1\times \big(1-e^{-1000\times(t-0.5)^2}\big)\times e^{[(x_1-t+0.5)^2+(x_2-t+0.5)^2]/0.04}\\
&q(x,t)=P_[q_a,q_b]\big(z(x,t)\big), \quad \mu(x,t)=z(x,t)-P_[q_a,q_b]\big(z(x,t)\big).
\end{align*}
Set $\Gamma_N=\Gamma_1\cup \Gamma_3$ and $\Gamma_D=\Gamma_2 \cup \Gamma_4$, respectively.  For simplicity, the functions $q_d(x,t)$, $g_{_D}(x,t)$, $g_{_N}(x,t)$ and $r_{_N}(x,t)$, respectively, are considered to be zero, while  $q_a=-0.25$  and $q_b=0.25$.
\end{exam}
We calculate the terms $f(x,t)$ and $y_d(x,t)$ using the state equation \eqref{contstate} and the adjoint-state equation \eqref{2.14adjoint-state}, respectively. Initially, we considered the mesh size $\tt N_x\times N_y=60\times 60$. From Figure \ref{appsolexp2state}, we observe that the time-step size drops in an interval around $t=0.5$ and remains constant away from this interval due to the term $\big(1-e^{-10000\times(t-0.5)^2}\big)$.
\begin{figure}[htbp]
\centering
\begin{minipage}[t]{0.5\textwidth} 
\noindent
Note that the term $\big(1-e^{-10000\times(t-0.5)^2}\big)$ changes exponentially from one to zero and  zero to one in the neighbourhood of $t=0.5$.  It is observed that the mesh  reflected via the indicators in the neighbourhood of singularity at the various time steps.
Furthermore, the mesh moves at a constant speed away from the neighbourhood of $t=0.5$ but the shape of the solution remains unchanged (cf., Figurers \ref{appsolexp2state}). Figures \ref{meshuniadptex2}--\ref{admeshex2} show that the mesh adapted very well by the derived error estimators. A comparison of uniform mesh (Figure \ref{unifmeshex2}) and adaptive meshes (Figures  \ref{meshuniadptex2}--\ref{admeshex2}) highlights the effectivity of the derived estimators.
\end{minipage}
\hfill
\begin{minipage}[t]{0.39\textwidth} \vfill
\centering
$\underset{\tt (a)~ \#E=8560,~ \#N=4401}{\includegraphics[width=\linewidth]{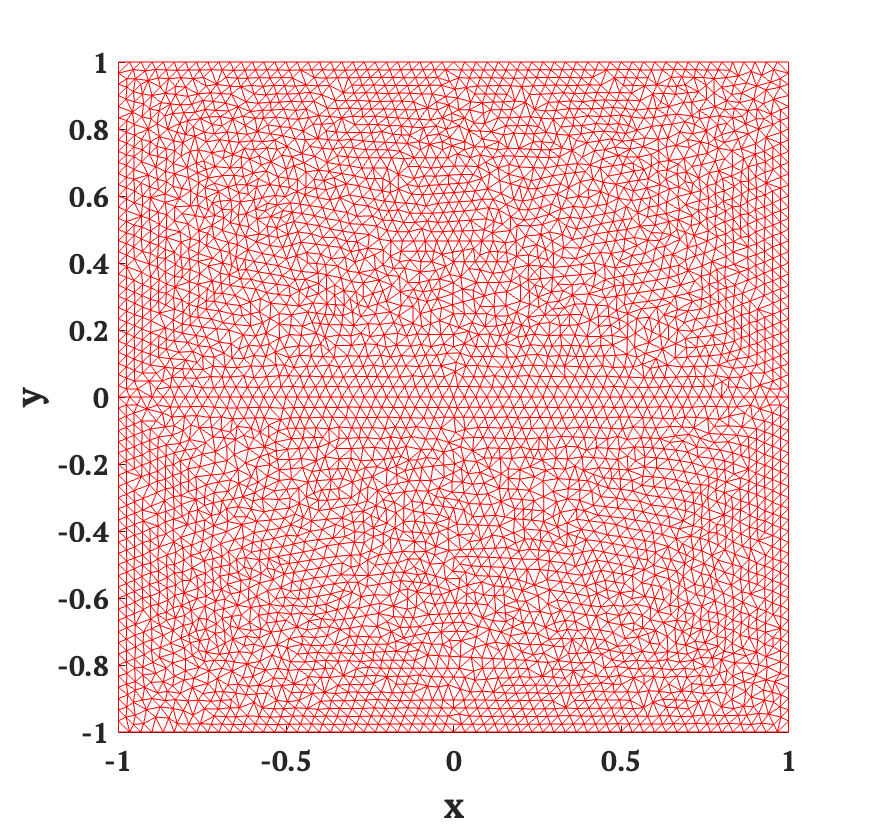}}$ 
\caption{Uniform mesh with mesh size $\tt N_x\times N_y=60\times 60$.}\label{unifmeshex2} 
\end{minipage}
\end{figure}
\begin{figure}[H]
\begin{center}
$\underset{\tt (b)~  Time~step~t=0.1}{\includegraphics[width=0.33\textwidth]{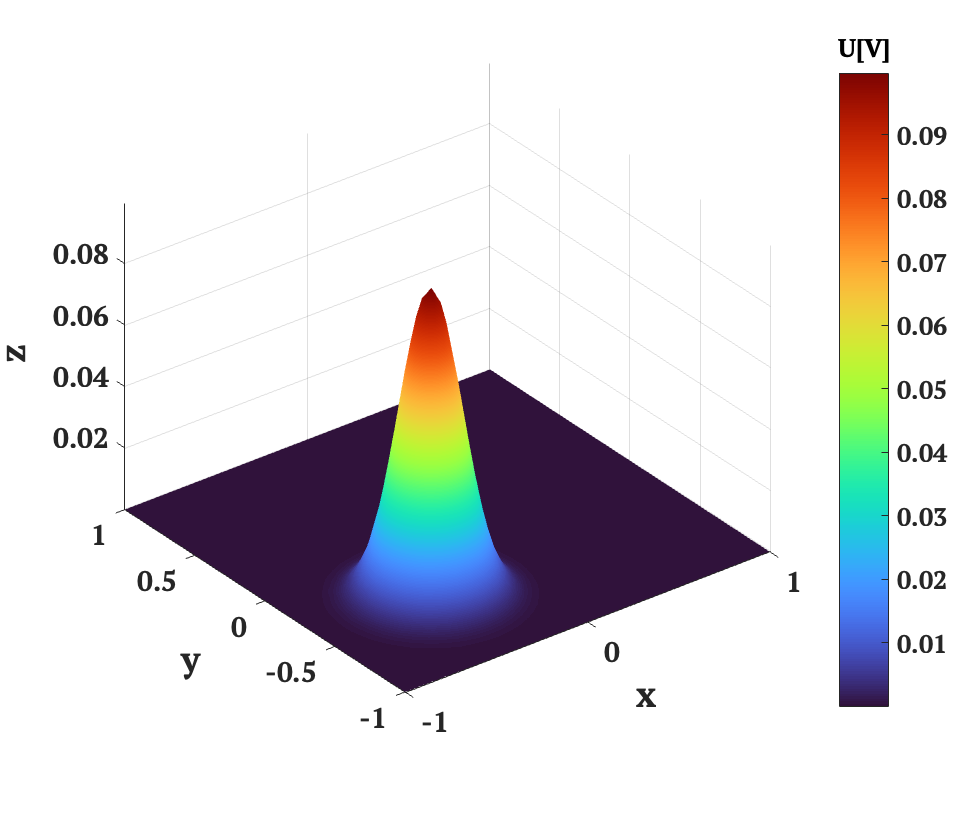}}$ \hfill 
$\underset{\tt (c)~  Time~step~t=0.48}{\includegraphics[width=0.325\textwidth]{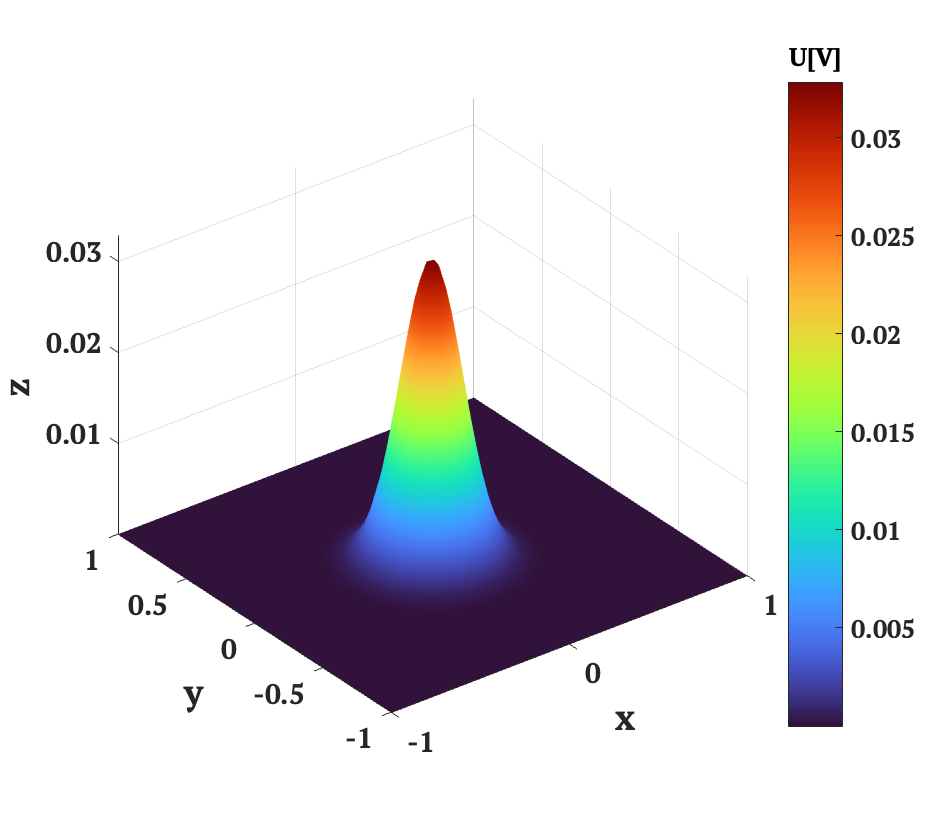}}$ \hfill 
$\underset{\tt (d)~  Time~step~t=1.0}{\includegraphics[width=0.325\textwidth]{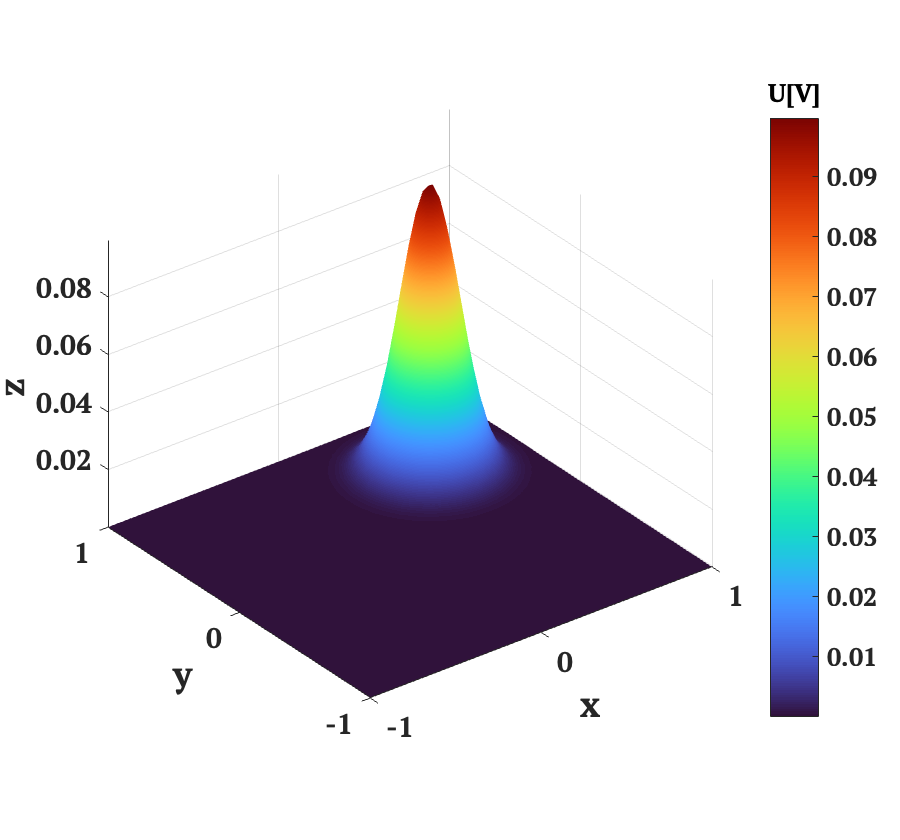}}$ 
\caption{Profile of the approximate state solutions with $\sigma_0=10^3$ at time steps $t=0.1$, $t=0.48$ and $t=1.0$, respectively.}\label{appsolexp2state} \vspace{0.1cm}
\end{center}
\begin{center}
$\underset{\tt (e)~  \#E=2606,~ \#N=1315}{\includegraphics[width=0.325\textwidth]{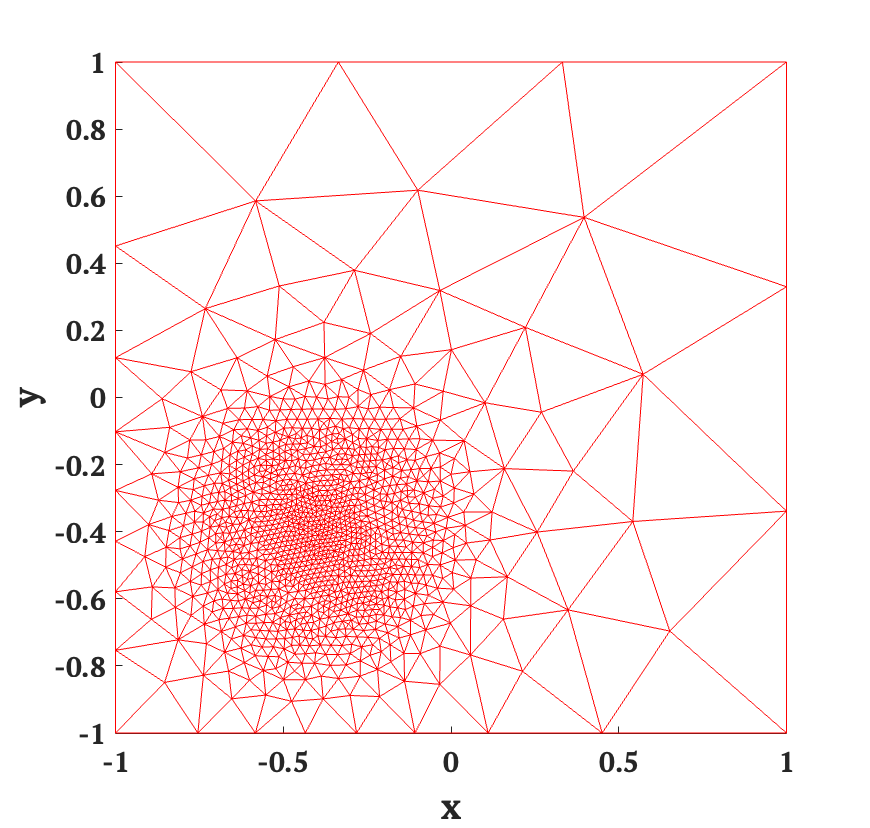}}$ \hfill 
$\underset{\tt (f)~ \#E=5465,~ \#N=2747}{\includegraphics[width=0.325\textwidth]{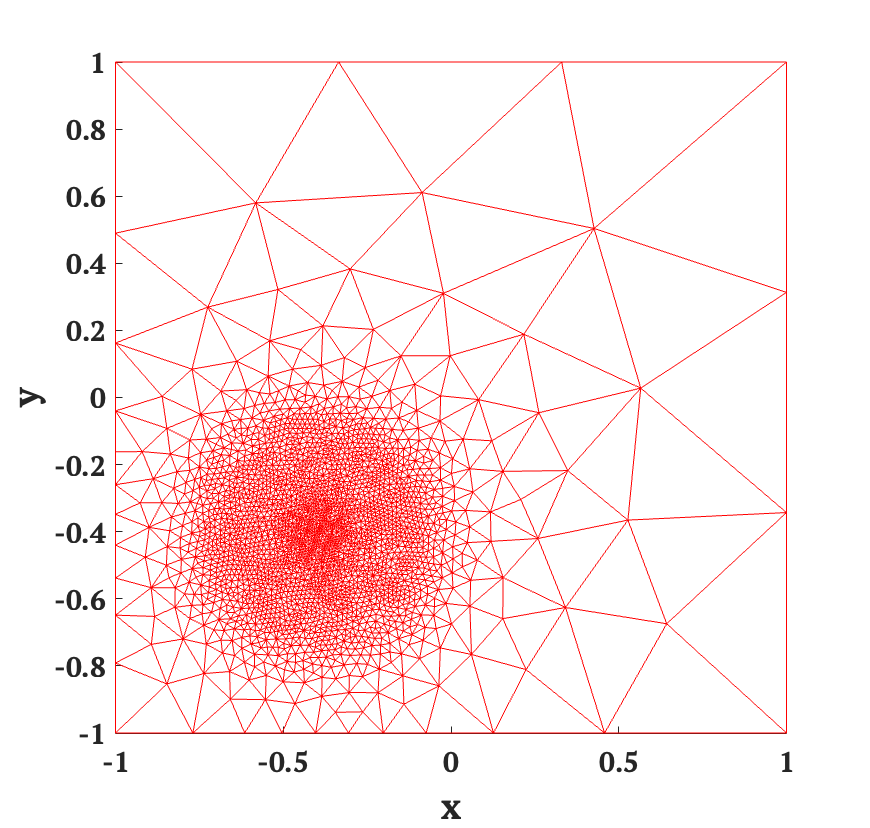}}$ \hfill 
$\underset{\tt (g)~  \#E=11538,~ \#N=5786}{\includegraphics[width=0.325\textwidth]{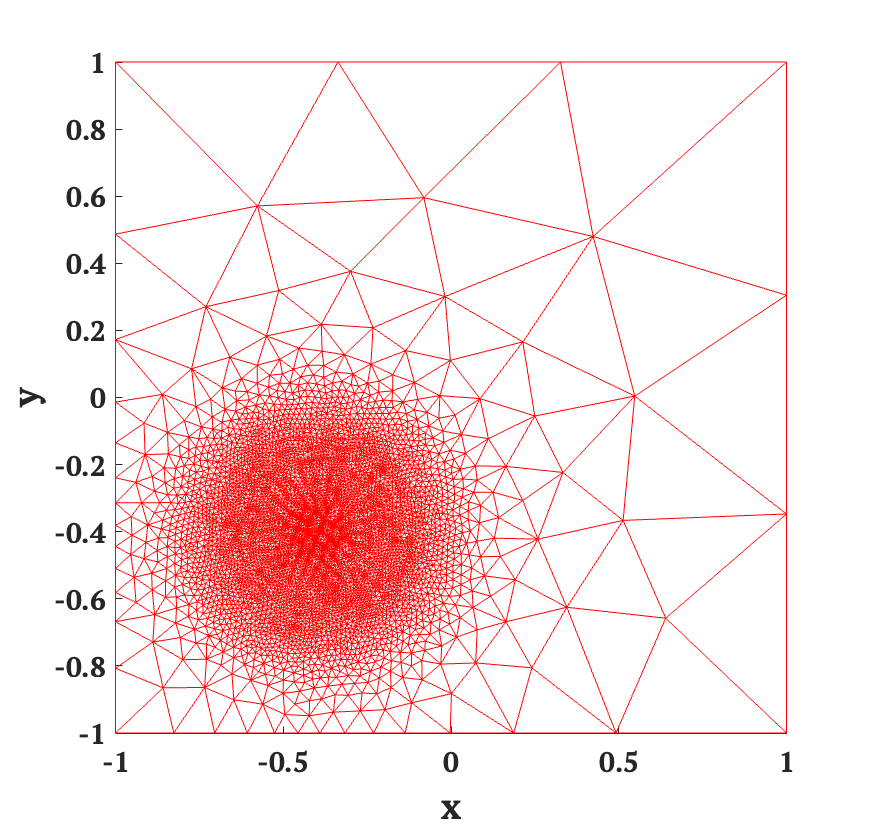}}$ 
\caption{Adaptive meshes after $1^{st}$, $2^{nd}$, and  $3^{rd}$-iterations with number of elements ($\tt \#E$) and nodes ($\tt \#N$), respectively,  at time $t=0.1$.}\label{meshuniadptex2} \vspace{0.1cm}
\end{center}
\end{figure}
\begin{figure}[H]
\begin{center}
$\underset{\tt (h)~  \#E=2660,~ \#N=1339}{\includegraphics[width=0.325\textwidth]{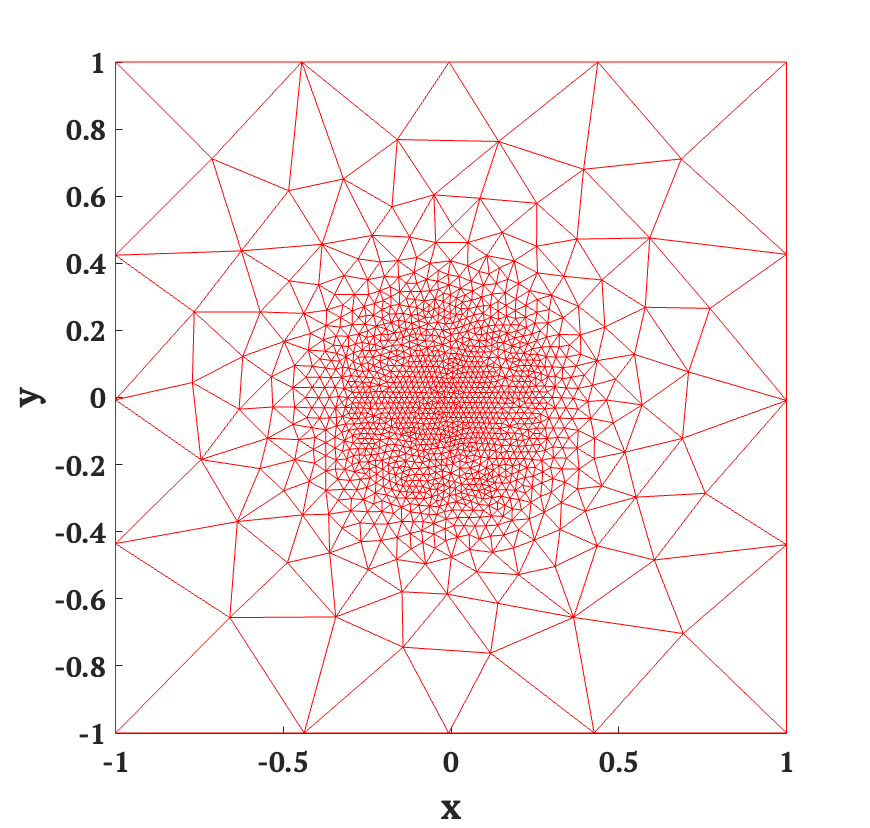}}$  \hfill 
$\underset{\tt (i)~ \#E=5530,~ \#N=2774}{\includegraphics[width=0.325\textwidth]{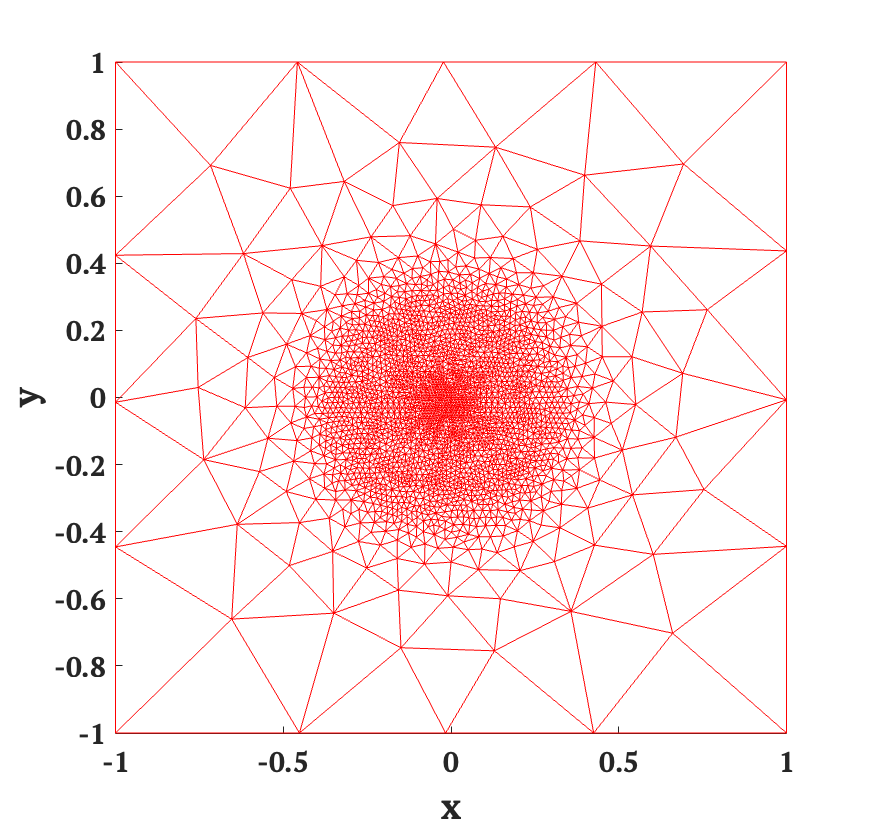}}$  \hfill 
$\underset{\tt (j)~ \#E=11720,~ \#N=5870}{\includegraphics[width=0.325\textwidth]{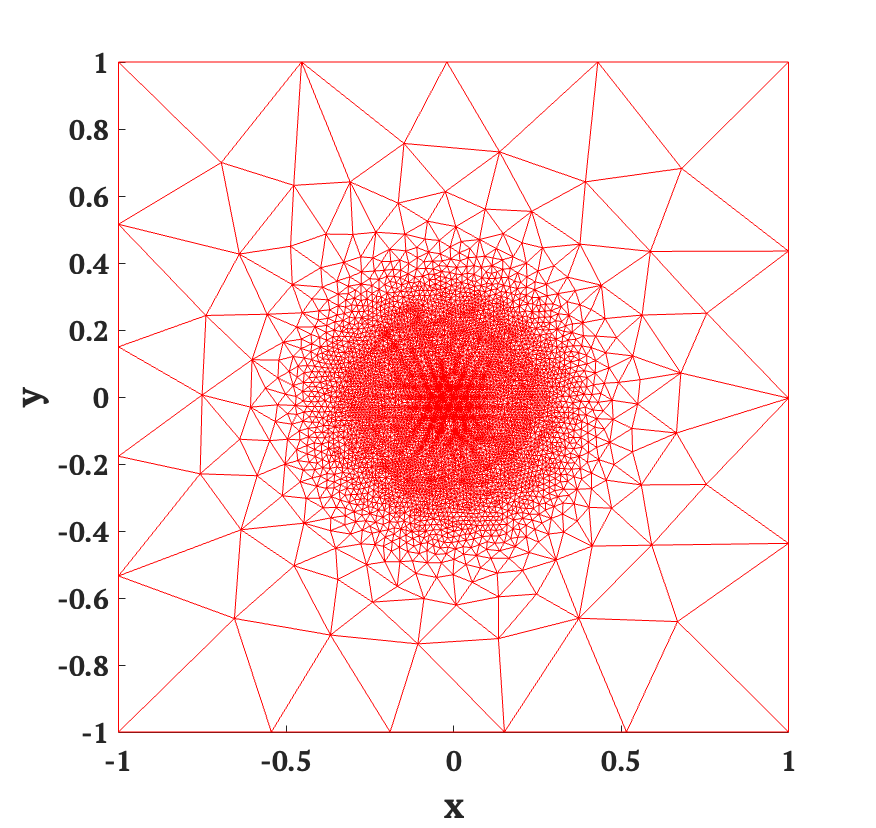} }$
\caption{Adaptive meshes after $1^{st}$, $2^{nd}$, and  $3^{rd}$-iterations with number of elements ($\tt \#E$) and nodes ($\tt \#N$), respectively,  at time $t=0.48$.}\label{appmeshuex2}  \vspace{0.48cm}
\end{center}
\begin{center}
$\underset{\tt (k)~  \#E=2548,~ \#N=1289}{\includegraphics[width=0.325\textwidth]{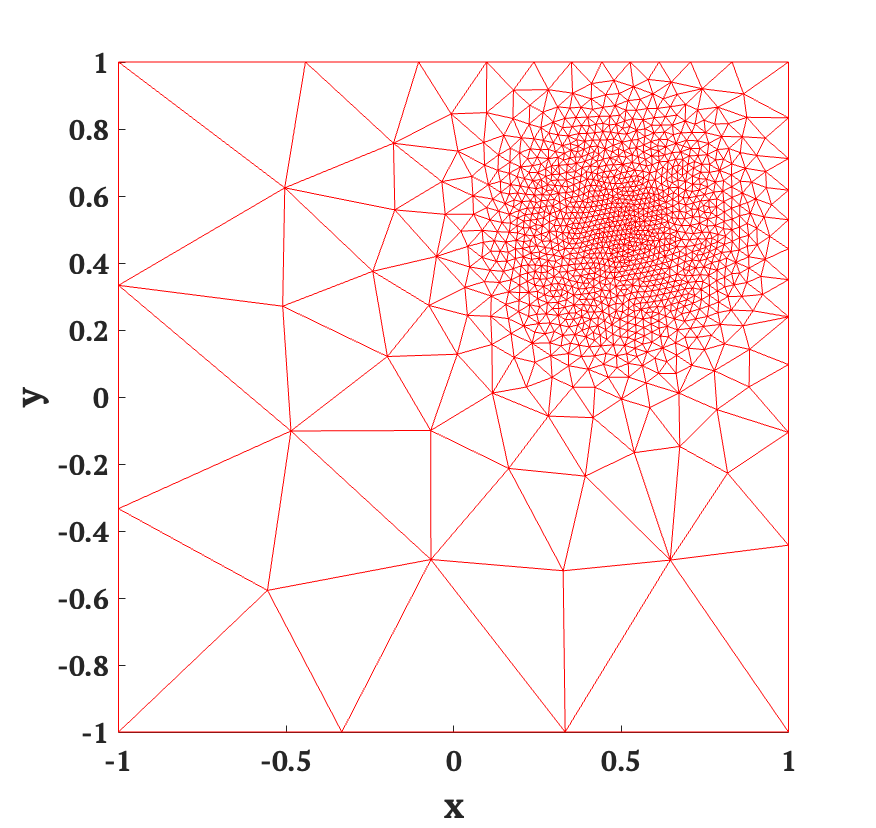} }$  \hfill
$\underset{\tt (l)~  \#E=5561,~ \#N=2801}{\includegraphics[width=0.325\textwidth]{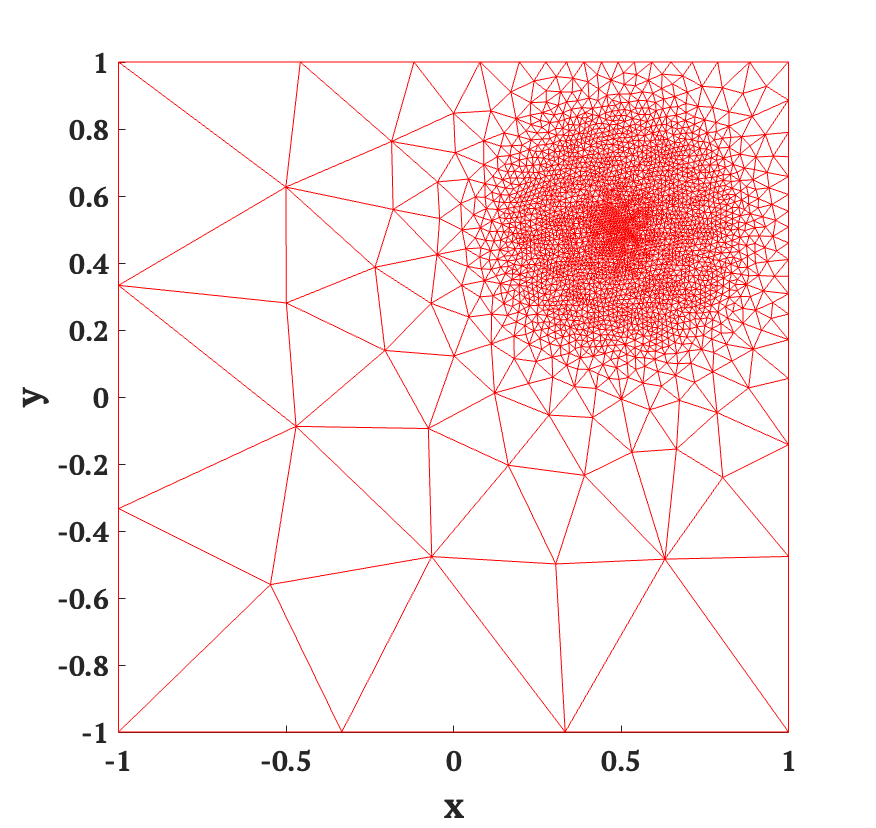}}$   \hfill
$\underset{\tt (m)~  \#E=11747,~ \#N=5898}{ \includegraphics[width=0.325\textwidth]{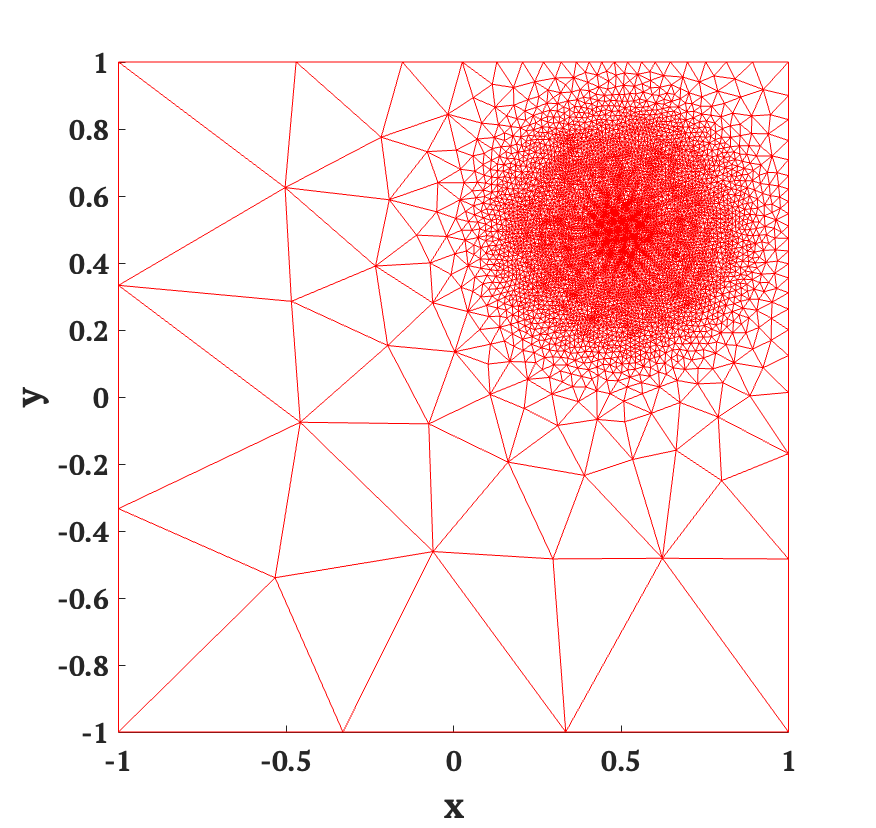}}$
\caption{Adaptive meshes after $1^{st}$, $2^{nd}$, and  $3^{rd}$-iterations with number of elements ($\tt \#E$) and nodes ($\tt \#N$), respectively,  at time step $t=1.0$.}\label{admeshex2} 
\end{center}
\end{figure}
\section{Conclusion} \label{section6666}
This study investigated the a posteriori error analysis for the SIPG method for parabolic BCPs with bilateral control constraints. We have employed piecewise-linear polynomials for the discretization of the state, adjoint-state, and control variables. Both lower and upper bounds for the error estimates are established showcasing the efficacy and dependability of the proposed error estimator through consideration of data oscillations. While the control error estimator $\eta_q$ proved effective in capturing the approximation error of the control, it was found to have limitations in providing guidance for refinement localization in certain critical instances. To address these deficiencies, we employ an alternative control indicator, $\bar{\eta}_q$, in numerical calculations. The results of these computations unequivocally highlighted the superiority of adaptive refinements over uniform meshes, underscoring the effectiveness of the proposed approach in achieving accurate solutions while optimizing computational efficiency. First example shows that the control estimator captured the active-inactive set accurately, while the second example depicted the dynamic behaviour of a time-dependent singularity. However, the adaptive meshes are displaying its singularity movement  as it evolves over time.  Our numerical findings emphasize the superiority of adaptive refinements over uniform meshes.
\bibliographystyle{plain}
\bibliography{RMRKBVR19022025.bib}
\end{document}